\documentclass[final,3p]{elsarticle}
\usepackage{geometry}
\usepackage{pifont}
\usepackage{mathrsfs,latexsym,float,enumerate,amssymb,amsmath,bm,color,lineno}
\usepackage[all]{xy}
\usepackage{latexsym}
\usepackage{amsthm,a4wide}
\usepackage{graphicx}
\usepackage[colorlinks=true]{hyperref}
\usepackage{amscd,verbatim,cuted}
\usepackage{makecell}
\usepackage{subfigure}
\usepackage{threeparttable}
\usepackage{amscd,verbatim,cuted}
\usepackage{makecell}
\usepackage{booktabs}
\usepackage{multirow}
\usepackage{epstopdf}
\usepackage{overpic}
\theoremstyle{plain}
\newtheorem{theorem}{Theorem}
\newtheorem{lemma}{Lemma}
\newtheorem{proposition}{Proposition}
\newtheorem{corollary}{Corollary}

\theoremstyle{definition}
\newtheorem{definition}{Definition}
\newtheorem{example}{Example}

\newtheorem{remark}{Remark}

\usepackage{pifont}

\journal{RIMNI}

\begin{document}

\begin{frontmatter}
\title{Power of Continuous Triangular Norms with Application to Intuitionistic
Fuzzy Information Aggregation\tnoteref{mytitlenote}}
\tnotetext[mytitlenote]{This work is supported by the Research Initiation Project for
Introduced Talent of Guizhou University of Finance
and Economics (No. 2023YJ23), the project of Young Scientific and Technical Talents Development of
Education Department of Guizhou Province (No.~[2024]80), the Universities Key Laboratory of
System Modeling and Data Mining in Guizhou
Province (No.~2023013), the Natural Science Foundation of Sichuan Province (No.~2022NSFSC1821), and the
High Level Innovative Talent Training Plan Project of Guizhou Province (No.~GCC[2023]006).}

\author[a1,a2]{Xinxing Wu\corref{mycorrespondingauthor}}
\cortext[mycorrespondingauthor]{Corresponding author}
\address[a1]{School of Mathematics and Statistics, Guizhou University of Finance and Economics, Guiyang,
Guizhou 550025, China}
\address[a2]{Zhuhai College of Science and Technology, Zhuhai, Guangdong 519041, China}
\ead{wuxinxing5201314@163.com}

\author[a1]{Xi Li}
\ead{LiXilimou@mail.gufe.edu.cn}

\author[a3]{Dan Huang}
\address[a3]{Xindu No.~1 Beixing Middle School of Chengdu, Chengdu, Sichuan 610504, China}
\ead{2327322647@qq.com}



\begin{abstract}
{The power of continuous Archimedean t-norms is fundamental for generalizing the multiplication and power operations of intuitionistic fuzzy sets (IFSs) within this framework. However, due to the lack of systematic research on the power of general continuous t-norms, it greatly limits the further generalization of the multiplication and power operations for IFSs via general continuous t-norms. This paper investigates the power of continuous t-norms and proposes a new intuitionistic fuzzy (IF) multiple-attribute decision-making (MADM) method. In the theory, the characterization of the power stability and the computational formula of power for continuous t-norms are obtained. Based on this, four fundamental operations induced by continuous t-norms for IFSs are introduced. Furthermore, various IF aggregation operators based on these four operations, namely the IF weighted average (IFWA), the IF weighted geometric (IFWG), and the IF mean weighted average and geometric (IFMWAG) operators, are developed, and their properties are analyzed. In the application, a new MADM method is designed based on the IFMWAG operator, which can remove the hindrance of the indiscernibility on the boundaries of some classical aggregation operators. The practical applicability and the comparative analysis with other MADM methods are furnished to show the advantages of the proposed MADM method.
}
\end{abstract}

\end{frontmatter}

\section{Introduction}
{Triangular norms (t-norms) and triangular conorms (t-conorms) originated from Schweizer and Sklar's work~\cite{SS1961}
in the context of probabilistic metric spaces~\cite{SS1983}. These mathematical constructs are instrumental in deriving
several classes of classical aggregation operators, which are essential for multiple-attribute decision-making (MADM).
Since Zadeh~\cite{Z1965} defined the fuzzy sets (FSs) by using membership degrees in 1965 to address issues of uncertainty
and ambiguity, many methodologies for expressing fuzzy information have been developed and successfully applied to fuzzy information aggregation and MADM problems.
However, Zadeh's fuzzy set can't represent the neutral state, i.e., neither opposing nor supporting. This limitation led Atanassov~\cite{Ata1986} to extend Zadeh's FS theory by introducing the notion of intuitionistic fuzzy sets (IFSs) in 1986.
Each IFS is characterized by a membership degree (MD) function and a non-membership degree (NMD) function, with the
restriction that the sum of these two degrees does not exceed $1$. Recognizing that experts often cannot make precise
decisions using exact numbers due to the complexity and uncertainty of available information, Atanassov and
Gargov~\cite{AG1989} further extended IFSs to interval-valued IFSs (IVIFSs) in 1989, replacing the MD and the NMD
by the closed intervals within $[0, 1]$.  To enlarge the range of information expression, Yager~\cite{Yager2017} introduced
the notion of q-rung orthopair fuzzy sets (q-ROFSs) in 2017. These sets are defined by membership degree (MD) and
non-membership degree (NMD) functions with the restriction that the sum of the $q^{\text{th}}$ powers of the MD and
NMD does not exceed $1$. In particular, when $q=2$, a q-ROFS simplifies to a Pythagorean fuzzy set (PFS), which was
also introduced by Yager~\cite{Yager2013}.

In recent years, many aggregation methods for intuitionistic fuzzy (IF) information have been developed.
Atanassov~\cite{Ata1986} and De et al.~\cite{DBR2000} proposed a few basic operations for IFSs by using the algebraic
product t-norm $T_{\textbf{P}}$, including ``addition--$\oplus$", ``{product}--$\otimes$", ``intersection--$\cap$",
``union--$\cup$", ``{complement}", ``scalar multiplication", and ``power". Later, Deschrijver and Kerre~\cite{DK2002}
generalized the addition and multiplication operations by applying t-norms and t-conorms. Based on these operations,
Xu et al.~(\cite{XY2006,Xu2007,XC2012}) introduced several IF aggregation operators, such as the IF weighted, ordered
and hybrid geometric (average) operator, induced generalized IF Choquet integral operator, and induced generalized IF Dempster-Shafer operator, proving their properties of monotonicity, idempotency, and boundedness. Xia et al.~\cite{XXZ2012}
and Beliakov et al.~\cite{BBGMP2011} further generalized these aggregation operators using continuous Archimedean t-(co)norms, and demonstrated their monotonicity, idempotency, and boundedness.  Ye~\cite{Ye2017} proposed the IF hybrid weighted arithmetic and geometric aggregation (IFHWAGA) operator by combining the IF geometric and average operators. Beliakov et
al.~\cite{BBJCF2012} extended the median aggregation operator for IFSs and {\color{blue}interval-valued FSs}.
Based on the Einstein product, Wang and Liu~\cite{WL2012} derived the Einstein weighted averaging and Einstein
ordered weighted averaging operators under the IF setting, and applied them to IF MADM problems. Observing that being a
special form of t-norm, the Einstein product is equal to Hamacher t-norm $T_{\lambda}^{\textbf{H}}$ when the parameter
$\lambda=2$, the main results of Wang and Liu~\cite{WL2012} are direct corollaries of those obtained by Xia et
al.~\cite{XXZ2012}. Garg~\cite{Garg2017} presented some new IF aggregation operators by considering the hesitation
degree based on the Einstein product. Considering the MADM problems expressed by the trapezoidal intuitionistic fuzzy
numbers (TrIFNs), Wan and Yi~\cite{WY2016} proposed some closed operational laws and various power average operators
for TrIFNs based on strict t-norms. Liu and Chen~\cite{LC2017} presented a new MAGDM methodology via the Heronian
aggregation operators for IFSs. To capture the interrelationships among multiple-attribute in the practical MADM
problem, Xu and Yager~\cite{XY2011} introduced the IF Bonferroni mean operator via the {algebraic product t-norm $T_{\textbf{P}}$}. Das et al.~\cite{DGM2017} proposed the IF extended Bonferroni mean operator based on strict t-conorms.
Liu \cite{Liu2014} applied the Hamacher aggregation and the power {M}aclaurin symmetric mean operators for MADM problems
under the IVIFS or q-ROFS setting.}

{Nevertheless, He et
al.~\cite{HH2016,HHC2015} pointed out that the operational laws in \cite{XXZ2012,Xu2007,XC2012} have the disadvantage that
if the rating of an alternative on some attribute is the maximum IF number (IFN) $\langle 1, 0\rangle$,
regardless of the ratings of this alternative on other attributes, the overall rating of this alternative always is
$\langle 1, 0\rangle$ (see \cite[Example~1]{HHC2015}). Clearly, this is impractical. Furthermore, we point out in
Example~\ref{Exm-5-Wu} of this paper that the operational laws presented  in
\cite{CC2016,FR2023,Garg2017,Huang2014,Liu2014,SCY2022,SSSDRT2023,WL2012,Xu2011,ZX2017} still have this disadvantage. To overcome this issue, He et al.~\cite{HHC2015,HCZLT2014} introduced some new interactional
operations for IFSs (see \cite[Definition~6]{HHC2015} and \cite[Definition~5]{HCZLT2014}).
However, those operational laws have the disadvantage that if the rating of an
alternative on some attribute is the minimum IFN $\langle 0, 1\rangle$, regardless of the ratings
of this alternative on other attributes, the overall rating of this alternative
always is $\langle 0, 1\rangle$ (see Example~\ref{Exm-5-Wu}), which is also impractical.
Furthermore, we point out in Example~\ref{Exm-5-Wu} of this paper that the operational laws presented in
\cite{Garg2016,HCZLT2014,HHC2015,LC2017,SCMY2023,XY2006,XY2011,Ye2017,ZX2017} also have this disadvantage.
Besides, all results on IF aggregation operators via t-(co)norms
only work on continuous t-(co)norms with the continuous additive generators,
namely, continuous Archimedean t-(co)norms. This is because these t-(co)norms
have the natural and the explicit power operation by using their
additive generators. However, many continuous t-(co)norms do not have
additive generators. This restricts the generalization of IF
aggregation operators via general continuous t-(co)norms. Meanwhile,
due to the complexity of the definition of the power operation for general
continuous t-norms, which involves pseudo-inverse operations,
theoretical research results in this area are extremely scarce.
To enrich the theoretical research and the practical application of the
power operation for continuous t-norms, using the power operation
for continuous t-(co)norms defined by using the pseudo-inverse (see~\cite[Remark~3.5]{KMP2000},
\cite{AFS2006,GMMP2009}), which will be proved to be equivalent to the power
operation introduced by Walker and Walker~\cite{WW2002} (see Remark~\ref{Remark-4}),
this paper investigates the basic properties of the power for continuous
t-norms and applies them to better aggregate IF information based on
continuous t-norms, which can completely overcome the two disadvantages
mentioned above.}

More precisely, we first prove that a continuous t-norm $T$ is power
stable if and only if every point in $[0, 1]$ is a power stable point, and if and only
if $T=T_{\mathbf{M}}$ (minimum) or $T$ is strict, or $T$ is an ordinal sum of
strict t-norms by using continuous t-norms' representation theorem (\cite{Ba1986,Ling1965}) in
Section~\ref{Sec-III}. Then, we introduce four basic operational laws for IFSs
using continuous t-norms and reveal a few operational properties of these four
operations in Section~\ref{Sec-IV}. Moreover, we introduce the IF
weighted average (geometric) operators in Section~\ref{Sec-V}, which generalize
the main results presented in~\cite{WL2012,XXZ2012,Xu2007,XY2006}.
In Section~\ref{Sec-VI}, we combine the IF weighted average operator and
the IF weighted geometric operator to propose
the IF mean weighted average and geometric operator (IFMWAG), which can overcome
the disadvantage of indiscernibility on the boundaries of some classical IF
aggregation operators studied in~\cite{CC2016,FR2023,Garg2016,Garg2017,HCZLT2014,HHC2015,Huang2014,Liu2014,LC2017,SCMY2023,SCY2022,WL2012,
Xu2007,Xu2011,XY2006,Ye2017,ZX2017},
and prove that it is monotonous, idempotent, and bounded. Meanwhile,
we establish a novel MADM method under the IF framework and show a practical example and
comparative analysis with other decision-making methods to illustrate the effectiveness of the
developed MADM method. Finally, we conclude the investigation in Section~\ref{Sec-VII}.

\section{Preliminaries}\label{S-II}
\subsection{Intuitionistic fuzzy sets}
\begin{definition}[{\textrm{\protect\cite{Ata1986}}}]
An \textit{intuitionistic fuzzy set}~(IFS) $F$ on a set $X$ is defined in the following form:
\begin{equation}
F =\left\{\langle x; \mu_{_{F}}(x), \nu_{_{F}}(x)\rangle
\mid x \in X\right\},
\end{equation}
where $ \mu_{_{F}}\colon X \rightarrow [0,1] $ and $\nu_{_{F}}\colon X \rightarrow [0,1]$
are the \textit{membership degree} and the \textit{non-membership degree} of an element
$x \in X$ in $F$, respectively, and for all $x\in X$,
\begin{equation}
\mu_{_{F}}(x)+\nu_{_{F}}(x)\leq 1.
\end{equation}
Moreover, $\pi_{_F}(x) =1 -\mu_{_F}(x) -\nu_{_F}(x)$
is called the \textit{hesitancy degree} of an element $x$ in $F$.
\end{definition}

In \cite{Xu2007,XC2012}, every pair $\langle\mu, \nu\rangle$ in $[0, 1]^2$ with
$0\leq \mu +\nu\leq 1$
is called an \textit{IF number} (IFN) or an \textit{IF value} (IFV).
Generally, use $\langle \mu_{\alpha}, \nu_{\alpha}\rangle$ to represent an IFV $\alpha$
and let $\tilde{\mathbb{I}}$ be the set of all IFVs.
Additionally, $A(\alpha)=\mu_{\alpha}+\nu_{\alpha}$ and
$S(\alpha)=\mu_{\alpha}-\nu_{\alpha}$ are called the
\textit{accuracy degree} and the \textit{score degree} of $\alpha$, respectively.

Using the basic operations for IFSs, Xu et al.~(\cite{XY2006,Xu2007,XC2012})
developed the following basic operations for IFVs.

\begin{definition}
[{\textrm{\protect\cite{XC2012}}}]
\label{Def-Int-Operations}
Let $\alpha=\langle \mu_{\alpha}, \nu_{\alpha}\rangle$ and
$\beta=\langle\mu_{\beta}, \nu_{\beta}\rangle\in \tilde{\mathbb{I}}$.
For $\lambda>0$, define
\begin{enumerate}[(i)]
\item $\alpha^{\complement}=\langle\nu_{\alpha},\mu_{\alpha}\rangle$;
\item $\alpha\cup \beta=\langle \mu_{\alpha} \vee\mu_{\beta}, \nu_{\alpha} \wedge\nu_{\beta}\rangle$;
\item $\alpha\cap \beta=\langle \mu_{\alpha} \wedge \mu_{\beta}, \nu_{\alpha} \vee \nu_{\beta}\rangle$;
\item $\alpha\oplus \beta=\langle \mu_{\alpha}+\mu_{\beta}-\mu_{\alpha}\mu_{\beta}, \nu_{\alpha}\nu_{\beta}\rangle$;
\item $\alpha\otimes \beta=\langle \mu_{\alpha}\mu_{\beta}, \nu_{\alpha}+\nu_{\beta}-\nu_{\alpha}\nu_{\beta}\rangle$;
\item $\lambda\alpha=\langle 1-(1-\mu_{\alpha})^{\lambda}, (\nu_{\alpha})^{\lambda}\rangle$;
\item $\alpha^{\lambda}=\langle(\mu_{\alpha})^{\lambda}, 1-(1-\nu_{\alpha})^{\lambda}\rangle$.
\end{enumerate}
\end{definition}

To rank all IFVs, Xu and Yager~(\cite{XY2006,Xu2007}) presented a total order
`$\leq_{_{\textrm{Xu}}}$' as follows:

\begin{definition}
[{\textrm{\protect\cite[Definition~1]{XY2006}}}]
\label{de-order(Xu)}
{Let $\alpha$, $\beta\in \tilde{\mathbb{I}}$.}
\begin{itemize}
  \item If $S(\alpha)<S(\beta)$,
then $\alpha$ is smaller than $\beta$,
denoted as $\alpha<_{_{\textrm{Xu}}}\beta$;

  \item If $S(\alpha)=S(\beta)$, then

\begin{itemize}
  \item if $A(\alpha)=A(\beta)$, then $\alpha=\beta$;
  \item if $A(\alpha)<A(\beta)$, then $\alpha$ is smaller
  than $\beta$, denoted as $\alpha<_{_{\textrm{Xu}}}\beta$;
\end{itemize}
\end{itemize}
If $\alpha<_{_{\textrm{Xu}}}\beta$ or $\alpha=\beta$, then
denote it by $\alpha\leq_{_{\textrm{Xu}}} \beta$.
\end{definition}

Szmidt and Kacprzyk~\cite{SK2009} developed another partial order for ranking
IFVs by a comparison function, $\rho(\alpha)=\frac{1}{2}(1+\pi_\alpha)(1-\mu_\alpha)$.
However, it sometimes cannot differentiate between two IFVs. Although Xu and Yager's
order `$\leq_{_{\text{XY}}}$' is a total order for ranking IFVs, its process has
the subsequent shortcomings:
(1) It has high sensitivity to the parameters changes.
(2) It may cause some unreasonable results that the more we know, the smaller the IFV.
(3) It is not preserved under scalar multiplication operation, i.e.,
$\alpha\leq_{_{\text{XY}}} \beta$ might not imply $\lambda \alpha
\leq_{_{\text{XY}}} \lambda \beta$, where $\lambda>0$ (see~\cite[Example~1]{BBGMP2011}).
To conquer such shortcomings, Zhang and Xu~\cite{ZX2012} developed Szmidt and Kacprzyk's comparison
function \cite{SK2009} based on Hwang and Yoon's order preference idea~\cite{HY1981} of similarity to
an ideal point, and defined the ``\textit{$L$-value}" $L(\alpha)$ as follows:
\begin{equation}
L(\alpha) =\frac{1 -\nu_{\alpha}}{(1 -\mu_{\alpha}) +(1 -\nu_{\alpha})}
 =\frac{1 -\nu_{\alpha}}{1 +\pi_{\alpha}},
\end{equation}
where $\alpha=\langle \mu_{\alpha}, \nu_{\alpha}\rangle$ is an IFV.
Applying the similarity function $L(\_)$,
they~\cite{ZX2012} introduced another total order `$\leq_{_{\textrm{ZX}}}$'
for IFVs as follows.

\begin{definition}[{\textrm{\protect\cite{ZX2012}}}]
\label{de-order(ZX)}
Let $\alpha$, $\beta\in \tilde{\mathbb{I}}$.
\begin{itemize}
  \item If $L(\alpha)< L(\beta)$, then $\alpha$ is
  smaller than $\beta$, denoted as $\alpha<_{_{\textrm{ZX}}} \beta$;
  \item If $L(\alpha)=L(\beta)$, then
  \begin{itemize}
    \item if $A(\alpha)=A(\beta)$, then $\alpha=\beta$;
    \item if $A(\alpha)<A(\beta)$, then $\alpha$ is
    smaller than $\beta$, denoted as $\alpha<_{_{\textrm{ZX}}} \beta$.
  \end{itemize}
\end{itemize}
If $\alpha<_{_{\textrm{ZX}}} \beta$ or $\alpha=\beta$, then denote
it by $\alpha\leq_{_{\textrm{ZX}}} \beta$.
\end{definition}

The following example shows that the order `$\leq_{_{\textrm{ZX}}}$'
in Definition \ref{de-order(ZX)} is likewise not preserved under
scalar multiplication operation.

\begin{example}
Take $\alpha=\langle 0.5, 0.4\rangle$, $\beta=\langle 0.51, 0.41\rangle$,
and $\lambda=0.6$.
Since $L(\alpha)=\frac{0.6}{1.1}<\frac{0.59}{1.08} =L(\beta)$, we have
$\alpha <_{_{\textrm{ZX}}}\beta$.
But $\lambda \alpha=\langle 1-0.5^{0.6}, 0.4^{0.6}\rangle$,
$\lambda \beta =\langle 1-0.49^{0.6}, 0.41^{0.6}\rangle$, and
$L(\lambda \alpha)=\frac{1-0.4^{0.6}}{1+0.5^{0.6}-0.4^{0.6}}\approx 0.3906$,
$L(\lambda \beta)=\frac{1-0.41^{0.6}}{1+0.49^{0.6}-0.41^{0.6}}\approx 0.3886$, and thus
$\lambda \alpha>_{_{\textrm{ZX}}} \lambda \beta$. Therefore,
$\alpha<_{_{\textrm{ZX}}} \beta$
does not imply $\lambda \alpha<_{_{\textrm{ZX}}}\lambda \beta$.
\end{example}

\subsection{Triangular norm}

{Triangular norms (t-norms) were systematically investigated by Schweizer and
Sklar~\cite{SS1961,SS1983} in the framework of probabilistic metric
spaces aiming at an extension of the triangle inequality. As an extension of
the logical connective \textit{conjunction} in classical two-valued logic,
t-norms have been used widely in decision making~\cite{GMMP2009,KMP2000}
and fuzzy set theory \cite{ZDe2023}.}

\begin{definition}[{\textrm{\protect\cite{KMP2000}}}]
A mapping $T: [0, 1]^2\rightarrow [0, 1]$ is said to be a
\textit{triangular norm} (or briefly, \textit{t-norm}) on $[0, 1]$ if,
for any $x$, $y$, $z \in [0, 1]$, the following conditions
are satisfied:
\begin{itemize}
\item[(T1)] $T(x, y)=T(y, x)$ (commutativity);

\item[(T2)] $T(x, T(y, x))=T(T(x, y), x)$ (associativity);

\item[(T3)] $T(x, y)\leq T(x, z)$
for $y\leq z$ (monotonicity);

\item[(T4)] $T(x, 1)=x$ (neutrality).
\end{itemize}
\end{definition}

Schweizer and Sklar~\cite{SS1961} introduced
triangular conorms as a dual concept of t-norms as follows.

A \textit{triangular conorm} (or briefly, \textit{t-conorm}) is a mapping
$S :[0, 1]^2 \rightarrow [0, 1]$, which, for any  $x$, $y$,
$z \in [0, 1]$, satisfies (T1)--(T3) and (S4):
\begin{itemize}
\item[(S4)] $S(x, 0)=x$ (neutrality).
\end{itemize}
\begin{proposition}[{\textrm{\protect\cite[Proposition~1.15]{KMP2000}}}]
A mapping $T$ is a t-norm if and only if there is a t-conorm $S$ such that,
for any $(x, y) \in [0, 1]^2$,
\begin{equation}
\label{eq-7.1}
T(x, y)=1-S(1-x, 1-y).
\end{equation}
\end{proposition}

The t-norm $T$ given by formula~\eqref{eq-7.1} is called the \textit{dual t-norm}
of $S$. Analogously, we can give the definition of the dual t-conorm of a t-norm $T$.

Because of the associativity, by~\cite[Definition~3.23]{GMMP2009}
and \cite[Remark~1.10]{KMP2000}, we can extend a t-norm
$T$ to an $n$-ary function $T^{(n)}:[0, 1]^{n} \rightarrow [0, 1]$ as follows:
$$
T^{(n)}(x_1, \ldots, x_{n-1}, x_n) \triangleq
T(T^{(n-1)}(x_1, \ldots, x_{n-1}), x_n).
$$
In particular, if $x_1 =x_2 =\cdots =x_n =x$,
then briefly denote
$$
x_{T}^{(n)} =T^{(n)}(x, x, \ldots, x), \quad n\geq 2,
$$
and
$$
x_{T}^{(0)} =1 \text{ and } x_{T}^{(1)} =x.
$$

\begin{definition}[{\textrm{\protect\cite{KMP2000}}}]
Assume that $T$ is a t-norm.
\begin{enumerate}[(i)]
\item A point $x\in [0, 1]$ is said to be an \textit{idempotent element}
of $T$ if $x_{T}^{(2)}=x$.
\item A point $x \in (0, 1)$ is said to be a \textit{nilpotent element} of
$T$ if $x_{T}^{(m)}=0$ holds for some $m\in\mathbb{N}$.
\item A point $x \in (0, 1)$ is said to be a \textit{zero divisor} of $T$
if $T(x, y)=0$ holds for some $y\in (0, 1)$.
\end{enumerate}
The sets of all idempotent elements, all nilpotent elements,
and all zero divisors of $T$ are denoted by $\mathfrak{I}_T$,
$\mathfrak{N}_T$, and $\mathfrak{D}_T^{0}$, respectively.
\end{definition}

\begin{definition}[{\textrm{\protect\cite[Definition~3.44]{KMP2000},
\cite[Definition~3.45]{GMMP2009}}}]
Let $\left\{T_\lambda\right\}_{\lambda\in \mathscr{A}}$ be a class of t-norms
and $\left\{(a_\lambda, e_\lambda)\right\}_{\lambda\in \mathscr{A}}$ be a class of non-empty
open subintervals of $[0, 1]$ with $(a_{\lambda_1}, e_{\lambda_1})\cap
(a_{\lambda_2}, e_{\lambda_2})=\varnothing$ for $\lambda_1\neq \lambda_2$.
The mapping $T:[0, 1]^2\rightarrow [0, 1]$ defined by
\begin{equation}
T(x, y)=
\begin{cases}
a_\lambda +(e_\lambda -a_\lambda) \cdot T_\lambda \left(\frac{x -a_\lambda}{e_\lambda -a_\lambda},
\frac{y -a_\lambda}{e_\lambda -a_\lambda}\right), & (x, y)\in \left[a_\lambda, e_\lambda\right]^2,\\
\min\{x, y\}, & \text{otherwise},
\end{cases}
\end{equation}
is called the \textit{ordinal sum} of the \textit{summands}
$\left\langle a_\lambda, e_\lambda, T_\lambda\right\rangle, \lambda\in \mathscr{A}$,
denoted by $T=\left(\langle a_\lambda, e_\lambda, T_\lambda\rangle\right)_{\lambda\in \mathscr{A}}.$
\end{definition}

\begin{definition}[{\textrm{\protect\cite{KMP2000}}}]
Let $\varphi: [a, b]\rightarrow [c, d]$ be a monotone
function, where $[a, b]$ and $[c, d]$ are two closed
subintervals of $[-\infty, +\infty]$. Define the \textit{pseudo-inverse} $\varphi^{(-1)}:
[c, d]\rightarrow [a, b]$ of $\varphi$ as
\begin{equation*}
\varphi^{(-1)}(y) =\sup\{x\in [a, b]\mid
(\varphi(x) -y)(\varphi(b) -\varphi(a)) <0\}.
\end{equation*}
\end{definition}

\begin{lemma}[{\textrm{\protect\cite[Remark~3.4]{KMP2000}}}]
\label{Pseudo-Inverse-Lemma}
Assume that the function $\varphi:[a, b]\rightarrow
[c, d]$ is continuous and strictly decreasing. Then,
for any $y\in \mathrm{Ran}(\varphi)=[\varphi(b), \varphi(a)]$,
$\varphi^{(-1)}(y)=\varphi^{-1}(y)$.
\end{lemma}

\begin{definition}[{\textrm{\protect\cite{KMP2000}}}]
An \textit{additive generator} (AG) of a t-norm $T$ is a mapping
$G:[0, 1]\rightarrow [0, +\infty]$ having the following
properties: (1) $G$ is strictly decreasing;
(2) $G$ is right-continuous at $0$; (3) $G(1)=0$;
(4) For any $(x, y)\in[0, 1]^2$, it holds
$G(x)+G(y)\in \mathrm{Ran}(G)\cup
[G(0),  +\infty]$
and
$T(x, y)=G^{(-1)}
(G(x)+G(y)).$
\end{definition}

\begin{definition}[{\textrm{\protect\cite{KMP2000}}}]
A t-norm $T$ is
\begin{enumerate}[(1)]
  \item \textit{strictly monotone} if $T(x, y)<T(x, z)$,
  for $x>0$ and $y<z$.

  \item \textit{Archimedean} if, for any $0<x$, $y<1$,
  there exists some $m \in \mathbb{N}$ such that $x_{T}^{(m)}
  < y$.
\end{enumerate}
The set of all continuous Archimedean t-norms is denoted by
$\mathscr{T}_{_{\mathrm{ConA}}}$.
\end{definition}

\begin{lemma}[{\textrm{\protect\cite[Theorem~2.12]{KMP2000}}}]
\label{Archi-Char}
A t-norm $T$ is Archimedean if and only if, for any $x\in (0, 1)$,
$\lim_{n\rightarrow +\infty} x_{T}^{(n)}=0$.
\end{lemma}

\begin{definition}[{\textrm{\protect\cite{KMP2000}}}]
A continuous t-norm $T$ is
\begin{enumerate}[(i)]
\item \textit{nilpotent} if $\mathfrak{N}_T=(0, 1)$;
\item \textit{strict} if it is strictly monotone.
\end{enumerate}
\end{definition}

\begin{lemma}[{\textrm{\protect\cite[Proposition~2.15]{KMP2000}}}]
\label{Strict->Archi}
Every strict t-norm is Archimedean.
\end{lemma}

\begin{lemma}[{\textrm{\protect\cite[Corollary~3.30]{KMP2000}}}]
\label{Strict-Nil-Char}
Let $G$ be an AG of a t-norm $T\in \mathscr{T}_{_{\mathrm{ConA}}}$. Then
\begin{enumerate}[{\rm (i)}]
\item $T$ is nilpotent if and only if $G(0)<+\infty$;
\item $T$ is strict if and only if $G(0)=+\infty$.
\end{enumerate}
\end{lemma}

\begin{lemma}[{\textrm{\protect\cite[Theorem~5.1]{KMP2000}}}]
\label{Cont-Archi-Char}
A t-norm $T$ is in $\mathscr{T}_{_{\mathrm{ConA}}}$ if
and only if it has a continuous AG.
\end{lemma}

The following important representation for continuous t-norms,
which plays a key role in the proof of the result in the next section
(see Theorem~\ref{Power-Stab-Char}), can be derived from results in \cite{MS1957}
in the context of $I$-semigroups. For a detailed proof, one is referred
to~\cite[Proposition~5.11]{KMP2000} or \cite[Theorem~3.49]{GMMP2009}.

\begin{lemma}
[{\textrm{\protect\cite[Theorem~5.11]{KMP2000}, \cite[Theorem~3.49]{GMMP2009}, \cite{Ling1965}}}]
\label{Cont-Char-Thm}
The following are equivalent:
\begin{itemize}
  \item[{\rm (i)}] $T$ is a continuous t-norm.
  \item[{\rm (ii)}] One of the following statements holds:
  \begin{itemize}
  \item[{\rm ii-1)}] $T=T_{\textbf{M}}$;
  \item[{\rm ii-2)}] $T\in \mathscr{T}_{_{\mathrm{ConA}}}$;
  \item[{\rm ii-3)}] $T$ is an ordinal sum of t-norms in $\mathscr{T}_{_{\mathrm{ConA}}}$.
\end{itemize}
\end{itemize}
\end{lemma}

\section{Power stability of continuous t-norms}\label{Sec-III}

Let $T$ be a continuous t-norm. Given any $n \in \mathbb{N}$, the mapping
$\varphi: [0, 1]\rightarrow [0, 1]$ defined by $\varphi(v)=v_{T}^{(n)}$ is
non-decreasing, continuous, and surjective. The pseudo-inverse $\varphi^{(-1)}$
is called the \textit{$n$-th root} w.r.t. $T$, implying the following
formula:
\begin{equation}
u_{T}^{(\frac{1}{n})}=\varphi^{(-1)}(u) =\sup\left\{v \in [0,1] \mid v_{T}^{(n)} <u\right\}.
\end{equation}
By \cite[Remark~3.4 (vii)]{KMP2000}
(also see {\textrm{\protect\cite[Proposition~3.35 (ii)]{GMMP2009}}}), we have
\begin{equation}
\label{eq-7.4}
(u_{T}^{(\frac{1}{n})})_{T}^{({n})}=u \text{ for all } u\in [0, 1].
\end{equation}
Meanwhile, if $T\in \mathscr{T}_{_{\mathrm{ConA}}}$,
it follows from~\cite[Proposition~3.35 (iv)]{GMMP2009} that
\begin{equation}
\label{eq-7.5}
\lim_{n\rightarrow +\infty} u_{T}^{(\frac{1}{n})} =
\begin{cases}
0, & u =0, \\
1, & u \in (0, 1].
\end{cases}
\end{equation}
Dually, the \textit{$n$-th root} of $u$
w.r.t. a continuous t-conorm $S$ is defined by
\begin{equation}
u_{S}^{(\frac{1}{n})} =\inf\left\{v \in[0, 1]
\mid  v_{S}^{(n)} >u\right\}.
\end{equation}
Putting, for $u \in [0, 1]$ and $p$, $q \in \mathbb{N}$,
\begin{equation}
\label{eq-7.7}
u_{T}^{(\frac{q}{p})} = (u_{T}^{(\frac{1}{p})})_{T}^{(q)},
\end{equation}
it can be verified that, for any $u\in [0, 1]$ and
$k$, $p$, $q\in \mathbb{N}$, we have $u_{T}^{(\frac{kq}{kp})}=
u_{T}^{(\frac{q}{p})}$, implying that the mapping $\Gamma: [0, 1]
\times \mathbb{Q} \rightarrow [0, 1]$ defined by $\Gamma(u, r)
=u_{T}^{(r)}$ is well-defined
(see \cite[Proposition~3.35 (iii)]{GMMP2009}).

Following the above discussions, we can extend the rational powers of
$u\in [0, 1]$ under a continuous t-norm or a continuous t-conorm to
the positive real numbers as follows.

\begin{definition}
\label{x-Power-Def}
Let $T$ be a continuous t-norm and $S$ be a continuous t-conorm. For any
$u\in [0, 1]$ and any $t> 0$, define
\begin{equation}
\label{eq-7.8}
u_{T}^{(t)} =\inf\left\{u_{T}^{(r)} \mid r\in[0, t] \cap \mathbb{Q}\right\},
\end{equation}
and
\begin{equation}
\label{eq-7.9}
u_{S}^{(t)} =\sup\left\{u_{S}^{(r)} \mid r\in[0, t] \cap \mathbb{Q}\right\}.
\end{equation}
\end{definition}
\begin{remark}\label{re-4}
By Definition \ref{x-Power-Def}, it can be verified that,
\begin{itemize}
  \item  for any $u\in [0, 1]$ and any $\frac{q}{p}\in \mathbb{Q}^{+}$,
  $u_{T}^{(\frac{q}{p})}$ defined in Definition~\ref{x-Power-Def} is equal
  to that defined by formula~\eqref{eq-7.7}.
  \item for any $u\in (0, 1]$ and any $t\in (0, +\infty)$, $u_{T_{\mathbf{P}}}^{(t)}=u^t$.
\end{itemize}
\end{remark}

\begin{lemma}
[{\textrm{\protect\cite[Remark~3.5]{KMP2000}, \cite[Proposition~3.35 (iii)]{GMMP2009}}}]
\label{le-7.8}
Assume that $T$ is a continuous t-norm and $S$ is a continuous t-conorm. Then,
for any $u\in [0, 1]$ and any $t_1$, $t_2\in\mathbb{Q}$,
\begin{enumerate}[{\rm (1)}]
\item $u_{T}^{(t_1+t_2)}=T(u_{T}^{(t_1)}, u_{T}^{(t_2)})$;
\item $u_{S}^{(t_1+t_2)}=S(u_{S}^{(t_1)}, u_{S}^{(t_2)})$.
\end{enumerate}
In particular, $u_{T}^{(\_)}$ is decreasing on $\mathbb{Q}^{+}$
and $u_{S}^{(\_)}$ is increasing on $\mathbb{Q}^{+}$.
\end{lemma}


\begin{proposition}
\label{0-Power-Pro}
Assume that $T$ is a continuous t-norm. Then, $0_{T}^{(t)} = 0$ for all $t \in (0, +\infty)$.
\end{proposition}
\begin{proof}
Based on Definition~\ref{x-Power-Def}, consider the following three cases:

(1)  For any $p \in\mathbb{N}$, $0_{T}^{(\frac{1}{p})} =\sup\{v \in [0, 1]
\mid v_{T}^{(p)} <0\} =\sup \varnothing =0$.

(2) For any $\frac{q}{p}\in\mathbb{Q}^+$, according to the above discussion,
it follows that $0_{T}^{(\frac{q}{p})}=(0_{T}^{(\frac{1}{p})})_{T}^{(q)}=0_{T}^{(q)}=0$.

(3)  For any $t\in (0, +\infty)\backslash \mathbb{Q}$, $0_{T}^{(t)}=\inf\{0_{T}^{(r)}
\mid r\in [0, t]\cap \mathbb{Q}\}=0$.
\end{proof}

Recently, Koles\'{a}rov\'{a} et al.~(\cite{KMR2014}) introduced the
notion of power stability. According to that, an aggregation function
$\textsf{Agg}: [0, 1]^2\rightarrow [0, 1]$ is said to be \textit{power stable}
if, for all $t>0$ and all $(u, v)\in [0, 1]^2$, there holds
\begin{equation}
(\textsf{Agg}(u, v))^t =\textsf{Agg}(u^t, v^t).
\end{equation}

Motivated by this, we introduce power stability for t-norms.
Dually, we give a definition for t-conorms.

\begin{definition}
Let $T$ be a continuous t-norm and $u\in [0, 1]$. Then,
\begin{enumerate}[(1)]
  \item $T$ is \textit{power stable} if, for all $t>0$
  and all $(u, v)\in [0, 1]^2$, there holds
\begin{align}
\label{eq-Power-Stab}
  (T(u, v))_{T}^{(t)} =T(u_{T}^{(t)},
  v_{T}^{(t)}). 
\end{align}

  \item $u$ is \textit{power stable} or $u$ is a
  \textit{power stable point} if, for any $t_1$, $t_2 >0$, we have
  $(u_{T}^{(t_1)})_{T}^{(t_2)}=u_{T}^{(t_1t_2)}.$
Since both $0$ and $1$ are power stable points for each continuous
t-norm $T$, which are called \textit{trivial power stable points} of $T$,
and each power stable point in $(0, 1)$ is called a \textit{non-trivial
power stable point} of $T$.
\end{enumerate}
\end{definition}

In general, a continuous t-norm is not necessarily power stable
(see Example~\ref{ex-7.1}). In the following, we will derive an
equivalent characterization for power stability of t-norms. In particular,
we prove that a continuous t-norm is power stable if and only if every point
in $[0, 1]$ is a power stable point, and if and only if $T=T_{\mathbf{M}}$
or $T$ is strict, or $T$ is representable as an ordinal sum of
strict t-norms (see Theorem~\ref{Power-Stab-Char}). 

\begin{example}\label{ex-7.1}
Define $G:[0, 1]\rightarrow [0, 1]$ as $G(x)=1-x$ and choose the
t-norm $T$ defined by $T(u, v)=G^{(-1)}(G(u)+G(v))$.
Clearly, $T\in \mathscr{T}_{_{\mathrm{ConA}}}$ and $T(\frac{1}{2}, \frac{1}{2})=0$.
Meanwhile, by direct calculation, we have $(\frac{1}{2})_{T}^{(\frac{1}{2})}=G^{-1}
(\frac{1}{4})=\frac{3}{4}$ (also see Lemma~\ref{Power-1/n-Thm}), implying that
$T((\frac{1}{2})_{T}^{(\frac{1}{2})}, (\frac{1}{2})_{T}^{(\frac{1}{2})})
=\frac{1}{2}\neq 0=(T(\frac{1}{2} , \frac{1}{2}))_{T}^{(\frac{1}{2})}$.
\end{example}

\begin{proposition}\label{pr-7.3}
Let $T$ be a continuous t-norm and $u \in \mathfrak{I}_T$.
Then, for any $t \in (0, +\infty)$, $u_{T}^{(t)} = u$.
\end{proposition}
\begin{proof}
If $u=0$, it has been proven in Proposition~\ref{0-Power-Pro}.
If $u\in\mathfrak{I}_T \backslash \{0\}$, consider the following three cases:

(1) For any $p\in \mathbb{N}$, from $u\in \mathfrak{I}_T$,
  it follows that $u_{T}^{(p)}=u$. Meanwhile, it can be verified that, for any
$v\in [0, u)$, $v_{T}^{(p)} \leq v<u$, implying that
 $u_{T}^{(\frac{1}{p})}=\sup\{v \in [0,1] \mid v_{T}^{(p)}<u\}
 =\sup\{v \mid 0\leq v<u\}=u.$

(2) For any $\frac{q}{p}\in\mathbb{Q}^+$, according
  to the above discussion, by $u \in \mathfrak{I}_T$, we have
  $u_{T}^{(\frac{q}{p})}=(u_{T}^{(\frac{1}{p})})_{T}^{(q)}=u_{T}^{(q)}=u$.

(3) For any $t\in (0, +\infty)\backslash \mathbb{Q}$, by (2),
we have $u_{T}^{(t)}=\inf\{u_{T}^{(r)} \mid r \in [0, t]\cap \mathbb{Q}\}=u$.
\end{proof}

Directly from Definition~\ref{x-Power-Def} and monotonicity of
t-(co)norms, we get the following result.

\begin{proposition}
\label{Power-Increasing-Pro-1}
Assume that $T$ is a continuous t-norm and $S$ is a continuous t-conorm. Then,
for $t\in (0, +\infty)$ and $u$, $v\in [0, 1]$ with $u \leq v$, we have
$u_{T}^{(t)}\leq v_{T}^{(t)}$ and $u_{S}^{(t)}\leq v_{S}^{(t)}$ .
\end{proposition}

\begin{lemma}
\label{Power-Dual-Thm}
Assume that $T$ is a continuous t-norm and $S$ is the dual t-conorm of $T$.
Then, for any $u\in [0, 1]$ and any $p$, $q\in \mathbb{N}$,
\begin{enumerate}[{\rm (i)}]
\item $u_{T}^{(p)} =1 -(1 -u)_{S}^{(p)}$ and $u_{S}^{(p)} =1 -(1 -u)_{T}^{(p)} $;

\item $u_{T}^{(\frac{1}{p})} =\min\{v \in[0,1] \mid v_{T}^{(p)} =u\}$
and $u_{S}^{(\frac{1}{p})} =\max\{v \in[0,1] \mid v_{S}^{(p)} =u\}$;

\item $(u_{T}^{(p)})_{T}^{(\frac{1}{p})}\leq u$ and $(u_{S}^{(p)})_{S}^{(\frac{1}{p})}\geq u$;

\item $u_{T}^{(\frac{q}{p})}\geq(u_{T}^{(q)})_{T}^{(\frac{1}{p})}$ and
$u_{S}^{(\frac{q}{p})}\leq(u_{S}^{(q)})_{S}^{(\frac{1}{p})}$ .
\end{enumerate}
\end{lemma}

\begin{proof}
(i) It follows directly from the duality between $T$ and $S$.

(ii) For $u=0$, it is clear that $0_{T}^{(\frac{1}{p})}=\sup\{v\in [0, 1]
\mid v_{T}^{(p)}<0\}=\sup \varnothing=0=\min\{v \in[0, 1] \mid v_{T}^{(p)}=0\}$.
For $u\in (0, 1]$, it is clear that $\mathscr{M}^{-}=\{v \in [0,1]
\mid v_{T}^{(p)}<u\}\neq\varnothing$ and that
$\mathscr{M}=\{v\in [0, 1] \mid v_{T}^{(p)}=u\}$ is a nonempty
closed subset of $[0, 1]$, since the function $(\_)_{T}^{(p)}$
is continuous and increasing with $\mathrm{Ran}((\_)_{T}^{(p)})=[0, 1]$.
Since $\mathscr{M}$ is closed, $\min\mathscr{M}$ exists, and denoted by
$\xi=\min\mathscr{M}$. Clearly, $\xi_{T}^{(p)}=u$ and $\xi \leq u_{T}^{(\frac{1}{p})}$
by formula~\eqref{eq-7.4}. For any $v \in\mathscr{M}^{-}$,
since $(\_)_{T}^{(p)}$ is increasing, we have $\xi \geq v$, and thus
$\xi \geq \sup\mathscr{M}^{-}=u_{T}^{(\frac{1}{p})} \geq \xi$, i.e.,
$u_{T}^{(\frac{1}{p})}=\min\{v \in [0, 1] \mid v_{T}^{(p)}=u\}$.

Similarly, we can prove $u_{S}^{(\frac{1}{p})} =\max\{v \in [0,1] \mid
v_{S}^{(p)} =u\}$.

(iii) It follows directly from $\mathrm{(ii)}$ that {$(u_{T}^{(p)})_{T}^{(\frac{1}{p})}
 =\min\{v \in [0, 1] \mid v_{T}^{(p)} =u_{T}^{(p)}\}  \leq u$} and
$(u_{S}^{(p)})_{S}^{(\frac{1}{p})} =\max\{v \in [0, 1] \mid v_{S}^{(p)} =u_{S}^{(p)}\}
 \geq u$.

(iv)  By direct calculation, we obtain
$(u_{T}^{(\frac{q}{p})})_{T}^{p}
 =((u_{T}^{(\frac{1}{p})})_{T}^{(q)})_{T}^{(p)}
 =(u_{T}^{(\frac{1}{p})})_{T}^{(pq)}
 =((u_{T}^{(\frac{1}{p})})_{T}^{(p)})_{T}^{(q)}
 =u_{T}^{(q)}$,
implying that $u_{T}^{(\frac{q}{p})} \in \{v\in[0, 1] \mid v_{T}^{(p)}=u_{T}^{(q)}\}$,
and thus $u_{T}^{(\frac{q}{p})} \geq \min\{v\in[0,1]  \mid v_{T}^{(p)}=u_{T}^{(q)}\}
=(u_{T}^{(q)})_{T}^{(\frac{1}{p})}$ by (ii). Similarly, one can prove
$u_{S}^{(\frac{q}{p})}\leq (u_{S}^{(q)})_{S}^{(\frac{1}{p})}$.
\end{proof}

\begin{remark}
For the continuous t-norm $T$ defined in Example \ref{ex-7.1}, we have
$((\frac{1}{2})_{T}^{(2)})_{T}^{(\frac{1}{2})} =0 <\frac{1}{2} =
((\frac{1}{2})_{T}^{(\frac{1}{2})})_{T}^{(2)}$.
This means that Lemma \ref{Power-Dual-Thm} (iii) and (iv) may strictly hold.
\end{remark}
\begin{lemma}\label{le-7.10}
Assume that $T$ is a continuous t-norm and $S$ is the dual t-conorm of $T$.
Then, for any $u\in [0, 1]$ and any $p\in \mathbb{N}$, we have
\begin{equation}
\label{eq-7.15}
u_{S}^{(\frac{1}{p})} =1 -(1 -u)_{T}^{(\frac{1}{p})},
\end{equation}
and
\begin{equation}
\label{eq-7.16}
u_{T}^{(\frac{1}{p})} =1 -(1 -u)_{S}^{(\frac{1}{p})}.
\end{equation}
\end{lemma}

\begin{proof}
For convenience, denote $u_{S}^{(\frac{1}{p})}  =\sigma$ and
$1 -(1 -u)_{_T}^{(\frac{1}{p})} =\vartheta$. First, it can be verified that
$\vartheta_{S}^{(p)}
=S^{(p)}(\vartheta, \ldots, \vartheta)=1-T^{(p)}(1-\vartheta, \ldots, 1-\vartheta)
=1-T^{(p)}((1-u)_{T}^{(\frac{1}{p})}, \ldots, (1-u)_{T}^{(\frac{1}{p})})
=1-((1-u)_{T}^{(\frac{1}{p})})_{T}^{(p)}=u$.
This, together with Lemma~\ref{Power-Dual-Thm} (ii), implies that
$\vartheta\leq u_{S}^{(\frac{1}{p})}=\sigma.$

Second, from $u_{S}^{(\frac{1}{p})} =\sigma$, we have
$u =\sigma_{S}^{(p)} =S^{(p)}(\sigma, \ldots, \sigma) =1 -T^{(p)}(1 -\sigma,
\ldots, 1 -\sigma) =1 -(1 -\sigma)_{T}^{(p)}$,
i.e., $(1-\sigma)_{T}^{(p)} = 1 -u$. This, together with Lemma~\ref{Power-Dual-Thm} (ii),
implies that $1-\sigma \geq (1-u)_{T}^{(\frac{1}{p})}=1-\vartheta$, and
thus $\sigma \leq \vartheta$. Therefore, $\sigma=\vartheta$. Similarly, we have
$u_{T}^{(\frac{1}{p})}=1-(1-u)_{S}^{(\frac{1}{p})}$.
\end{proof}

\begin{lemma}\label{le-7.11}
Assume that $T$ is a continuous t-norm and $S$ is the dual t-conorm of $T$.
Then, for any $u \in [0, 1]$ and any $\frac{q}{p}\in \mathbb{Q}^{+}$, we have
\begin{equation}
u_{S}^{(\frac{q}{p})} =1 -(1 -u)_{T}^{(\frac{q}{p})},
\end{equation}
and
\begin{equation}
u_{T}^{(\frac{q}{p})} =1 -(1 -u)_{S}^{(\frac{q}{p})}.
\end{equation}
\end{lemma}

\begin{proof}
By Lemma~\ref{le-7.10}, we have
\begin{align*}
u_{S}^{(\frac{q}{p})}
& =(u_{S}^{(\frac{1}{p})})_{S}^{(q)} =S^{(q)}(u_{S}^{(\frac{1}{p})}, \ldots, u_{S}^{(\frac{1}{p})})\\
& =1 -T^{(q)}(1 -u_{S}^{(\frac{1}{p})}, \ldots, 1 -u_{S}^{(\frac{1}{p})})\\
& =1 -T^{(q)}((1 -u)_{T}^{(\frac{1}{p})}, \ldots, (1 -u)_{T}^{(\frac{1}{p})})\\
& =1 -\Big((1 -u)_{T}^{(\frac{1}{p})}\Big)_{T}^{(q)} =1 -(1 -u)_{T}^{(\frac{q}{p})},
\end{align*}
and
\begin{align*}
u_{T}^{(\frac{q}{p})}
& =(u_{T}^{(\frac{1}{p})})_{T}^{(q)} =T^{(q)}(u_{T}^{(\frac{1}{p})}, \ldots, u_{T}^{(\frac{1}{p})})\\
& =1 -S^{(q)}(1 -u_{T}^{(\frac{1}{p})}, \ldots, 1 -u_{T}^{(\frac{1}{p})})\\
& =1 -S^{(q)}((1 -u)_{S}^{(\frac{1}{p})}, \ldots, (1 -u)_{S}^{(\frac{1}{p})})\\
& =1 -\Big((1 -u)_{S}^{(\frac{1}{p})}\Big)_{S}^{(q)}=1 -(1 -u)_{S}^{(\frac{q}{p})}.
\end{align*}
\end{proof}

\begin{theorem}
\label{Power-Dual-Thm-Main}
Assume that $T$ is a continuous t-norm and $S$ is the dual t-conorm of $T$.
Then, for any $u\in [0, 1]$ and any $t\in (0, +\infty)$, we have
\begin{equation}
u_{S}^{(t)} =1 -(1 -u)_{T}^{(t)},
\end{equation}
and
\begin{equation}
u_{T}^{(t)} =1 -(1 -u)_{S}^{(t)}.
\end{equation}
\end{theorem}
\begin{proof}
If $t\in \mathbb{Q}$, this has been proven in Lemma~\ref{le-7.11}.
If $t\in (0, +\infty)\backslash \mathbb{Q}$, by Lemma~\ref{le-7.11}
and formulas~\eqref{eq-7.8} and \eqref{eq-7.9}, we have
\begin{align*}
u_{S}^{(t)}
&=\sup \left\{u_{S}^{(r)} \mid r\in [0, t]\cap \mathbb{Q}\right\} \\
&=\sup \left\{1-(1-u)_{T}^{(r)} \mid r\in [0, t] \cap \mathbb{Q}\right\} \\
&=1-\inf \left\{(1-u)_{T}^{(r)} \mid r\in [0, t] \cap \mathbb{Q}\right\} \\
&=1-(1-u)_{T}^{(t)},
\end{align*}
and
\begin{align*}
u_{T}^{(t)}
&=\inf \left\{u_{_T}^{(r)} \mid r\in [0, t]\cap \mathbb{Q}\right\} \\
&=\inf \left\{1-(1-u)_{S}^{(r)} \mid r\in [0, t]\cap \mathbb{Q}\right\} \\
&=1-\sup \left\{(1-u)_{S}^{(r)} \mid r\in [0, t]\cap \mathbb{Q}\right\} \\
&=1-(1-u)_{S}^{(t)}.
\end{align*}
\end{proof}

\begin{theorem}
\label{Power-Limit-Thm}
Let $T\in \mathscr{T}_{_{\mathrm{ConA}}}$
and $u\in [0, 1]$. Then,
\begin{enumerate}[{\rm (1)}]
\item $u_{T}^{(\_)}$ is a continuous function on $(0, +\infty)$;
\item $\lim_{t\to +\infty} u_{T}^{(t)} =
\begin{cases}
0, & u  \in [0, 1), \\
1, & u =1.
\end{cases}$
\end{enumerate}
\end{theorem}
\begin{proof}
It suffices to consider the following two cases:

(1) $u=0$ or $u=1$. From Proposition~\ref{pr-7.3}, it follows that
$0_{T}^{(\_)}\equiv 0$ and $1_{T}^{(\_)}\equiv 1$. Thus, both
$0_{T}^{(\_)}$ and $1_{T}^{(\_)}$ are continuous,
$\lim_{t\to +\infty} 0_{T}^{(t)}=0$, and
$\lim_{t\to +\infty} 1_{T}^{(t)}=1$.

(2) $u \in (0, 1)$. For any $\varepsilon>0$, by the uniform continuity
of $T$, it follows that there exists $\delta >0$ such that for any
$(u_{1}, v_{1})$, $(u_{2}, v_{2})\in [0, 1]^{2}$
with $\max\{|u_{1}-u_{2}|, |v_{1}-v_{2}|\}< \delta$,
$|T(u_{1}, v_{1})-T(u_{2}, v_{2})|<\varepsilon$.
From $\lim_{n\to +\infty} u_{T}^{(\frac{1}{n})}=1$ (by formula~\eqref{eq-7.5}),
it follows that there is $K\in \mathbb{N}$ such that for any $n \geq K$,
$|u_{T}^{(\frac{1}{n})}-1|< \delta$. For any $r_{1}$, $r_{2}\in \mathbb{Q}$
with $0\leq r_{2}-r_{1}< \frac{1}{K}$, by Lemma~\ref{le-7.8}, we have
\begin{align*}
&\quad |u_{T}^{(r_{1})}-u_{T}^{(r_{2})}|=u_{T}^{(r_{1})}-u_{T}^{(r_{2})}\\
&=u_{T}^{(r_{1})}-T(u_{T}^{(r_{1})}, u_{T}^{(r_{2}-r_{1})})\\
&=T(u_{T}^{(r_{1})}, 1)-T(u_{T}^{(r_{1})}, u_{T}^{(r_{2}-r_{1})})\\
&\leq T(u_{T}^{(r_{1})}, 1)-T(u_{T}^{(r_{1})}, u_{T}^{(\frac{1}{K})}).
\end{align*}
This, together with $|u_{T}^{(\frac{1}{K})}-1|<\delta$,
implies that $|u_{T}^{(r_{1})}-u_{T}^{(r_{2})}|<\varepsilon$.
This means that $u_{T}^{(\_)}$ is continuous on $\mathbb{Q}^{+}$.
Since $u_{T}^{(\_)}$ is decreasing on $(0, +\infty)$, by the definition of
$u_{T}^{(\_)}$ (see Definition~\ref{x-Power-Def}), $u_{T}^{(\_)}$
is continuous on $(0, +\infty)$ and $\lim_{t\to +\infty} u_{T}^{(t)}=0$
by Lemmas~\ref{Strict->Archi} and \ref{Archi-Char}.
\end{proof}

\begin{remark}
Alsina et al.~\cite[Page 32 (19)]{AFS2006} also proved the continuity
of $u_{T}^{(\_)}$. We include its proof here for completeness.
\end{remark}

Dually, by Theorem \ref{Power-Dual-Thm-Main}, the following result holds.

\begin{theorem}
\label{Dual-Power-Limit-Thm}
Let $S$ be the dual t-conorm of a t-norm $T\in \mathscr{T}_{_{\mathrm{ConA}}}$
and $u\in [0, 1]$. Then,
\begin{enumerate}[{\rm (1)}]
\item $u_{S}^{(\_)}$ is a continuous function on $(0, +\infty)$;
\item $\lim_{t\to +\infty} u_{S}^{(t)} =
\begin{cases}
0, & u =0, \\
1, & u \in (0, 1].
\end{cases}$
\end{enumerate}
\end{theorem}

\begin{lemma}\label{le-7.12}
For any $u$, $v\in [0, 1]$ and $t\geq 0$,
$\left(T_{\mathbf{M}}(u, v)\right)_{T_{\mathbf{M}}}^{(t)}
=T_{\mathbf{M}}(u_{T_{\mathbf{M}}}^{(t)}, v_{T_{\mathbf{M}}}^{(t)})$.
\end{lemma}

\begin{proof}
Assuming that $u \leq v$, by Proposition~\ref{Power-Increasing-Pro-1},
we have $u_{T_{\mathbf{M}}}^{(t)}\leq v_{T_{\mathbf{M}}}^{(t)}$,
implying that $(T_{\mathbf{M}}(u, v))_{T_{\mathbf{M}}}^{(t)}=u_{T_{\mathbf{M}}}^{(t)}
=T_{\mathbf{M}}(u_{T_{\mathbf{M}}}^{(t)}, v_{T_{\mathbf{M}}}^{(t)})$.
\end{proof}

\begin{lemma}
\label{Power-1/n-Thm}
Let $G$ be an AG of a t-norm $T\in \mathscr{T}_{_{\mathrm{ConA}}}$.
Then, for any $u\in (0, 1]$ and any $m\in \mathbb{N}$, we have
$u_{T}^{(m)}=G^{(-1)}(m  \cdot G(u))
=G^{-1}(\min\{m  \cdot G(u), G(0)\})$
and $u_{T}^{(\frac{1}{m})}=G^{-1}(\frac{G(u)}{m})$.
\end{lemma}
\begin{proof}
Clearly, $G$ is strictly decreasing and continuous by Lemma~\ref{Cont-Archi-Char}.
\begin{itemize}
\item[(1)] Fix $u\in (0, 1]$. We can prove
$u_{T}^{(m)}=G^{(-1)}(m \cdot G(u))$ by using mathematical induction on $m$.

\begin{itemize}
\item[1.1)] When $m=2$, by Lemma~\ref{Cont-Archi-Char}, we have
  $u_{T}^{(2)}=T(u, u)=G^{(-1)}(G(u)+G(u))=G^{(-1)}(2 \cdot G(u)).$

  \item[1.2)] Suppose that $m=k$, $u_{T}^{(k)}=G^{(-1)}(k \cdot G(u))$ holds.
  Then, when $m=k+1$, by Lemma~\ref{Cont-Archi-Char}, we have
  \begin{align*}
  u_{T}^{(k+1)}&=T(u_{T}^{(k)}, u)=G^{(-1)}(G(u_{T}^{(k)})+G(u))\\
  &=G^{(-1)}(G \circ G^{(-1)}(k \cdot G(u))+G(u)).
  \end{align*}
\end{itemize}
Next, consider the following cases:
\begin{itemize}
\item If $k\cdot G(u)\geq G(0)$, then
$G^{(-1)}(k\cdot G(u))=0$ and $G^{(-1)}((k+1)\cdot G(u))=0$,
implying that $u_{T}^{(k+1)}=G^{(-1)}(G \circ G^{(-1)}
(k \cdot G(u))+G(u))=0=G^{(-1)}((k+1)\cdot G(u))$.

\item If $k\cdot G(u)< G(0)$, then, since $G$ is
strictly decreasing and continuous on $[0, 1]$, $G^{(-1)}(k \cdot G(u))
=G^{-1}(k \cdot G(u))$. This implies that $G^{(-1)}
(G \circ G^{(-1)}(k \cdot G(u))+G(u))
=G^{(-1)}(G \circ G^{-1}(k \cdot G(u))+G(u))
=G^{(-1)}((k+1)\cdot G(u))$.
\end{itemize}
Thus,  by 1.1) and 1.2), $u_{T}^{(m)}=G^{(-1)}(m \cdot G(u))$ holds for all $m$.

\item[(2)]
Fix $u\in (0, 1]$ and $m\in \mathbb{N}$. By Lemma~\ref{Power-Dual-Thm} (ii)
and the above discussion, we have
\begin{align*}
& u_{T}^{(\frac{1}{m})}
=\min\{v \in [0, 1] \mid v_{T}^{(m)}=u\} \\
=&\min\{v \in [0, 1] \mid G^{(-1)}(m\cdot G(v))=u\} \\
=&\min\{v \in [0, 1] \mid m \cdot G(v)=G(u)\}\quad
\text{(by $G(u)<G(0)$)} \\
=&\min\left\{ v \in [0, 1] \mid G(v)=\frac{G(u)}{m}\right\} \\
=&G^{-1}\left(\frac{G(u)}{m}\right) \quad \text{(by Lemma~\ref{Pseudo-Inverse-Lemma})}
\end{align*}
\end{itemize}
\end{proof}

\begin{corollary}
\label{Corollary-7.1}
Let $T\in \mathscr{T}_{_{\mathrm{ConA}}}$ and $u\in (0, 1]$. If there exists
$z\in [0, 1]$ and $p\in \mathbb{N}$ such that $z_{T}^{(p)}=u$, then
$u_{T}^{(\frac{1}{p})}=z$.
\end{corollary}

\begin{proof}
Take an AG $G$ of $T$. By Lemma~\ref{Power-1/n-Thm},
we have $z_{_T}^{(p)}=G^{(-1)}(p \cdot G(z))=u > 0$, implying that
$G(u)=p \cdot G(z)$. Therefore, $u_{T}^{(\frac{1}{p})}
=G^{-1}(\frac{G(u)}{p})
=G^{-1}(G(z))=z$.
\end{proof}

\begin{proposition}\label{pr-7.6}
Let $T\in \mathscr{T}_{_{\mathrm{ConA}}}$. Then, for any $u\in (0, 1]$ and any
$p$, $p^{\prime}\in \mathbb{N}$, we have $(u_{T}^{(\frac{1}{p})})_{T}^{(\frac{1}{p^{\prime}})}
=u_{T}^{(\frac{1}{pp^{\prime}})}$.
\end{proposition}

\begin{proof}
By Lemma~\ref{Power-1/n-Thm}, we have
$(u_{T}^{(\frac{1}{p})})_{T}^{(\frac{1}{p^{\prime}})}=G^{-1}
(\frac{G(u_{T}^{(\frac{1}{p})})}{p^{\prime}})=
G^{-1}(\frac{G(u)}{pp^{\prime}})
=u_{T}^{(\frac{1}{pp^{\prime}})}.$
\end{proof}

\begin{proposition}
\label{Power-t-Thm}
Let $G$ be an AG of a t-norm $T\in \mathscr{T}_{_{\mathrm{ConA}}}$.
Then, for any $u\in (0, 1]$ and any $t\in (0, +\infty)$,
$u_{T}^{(t)}=G^{( -1)}(t  \cdot G(u))
=G^{ -1}(\min\{t \cdot G(u),G(0)\})$.
Moreover, if the t-norm $T$ is strict, then
$u_{T}^{(t)}=G^{ -1}(t \cdot G(u))$.
\end{proposition}

\begin{proof}
For any $\frac{q}{p}\in \mathbb{Q}^{+}$, by Lemma~\ref{Power-1/n-Thm}, we have
$u_{T}^{(\frac{q}{p})}=(u_{T}^{(\frac{1}{p})})_{T}^{(q)}
=G^{(-1)}(q \cdot \frac{G(u)}{p})
=G^{-1}(\min\{\frac{q}{p} \cdot G(u), G(0)\}).$
This, together with Theorem~\ref{Power-Limit-Thm}, implies that, for any
$t \in (0, +\infty)$, $u_{T}^{(t)}=G^{(-1)}(\min\{t \cdot G(u), G(0)\})$.
\end{proof}

\begin{theorem}
\label{Power-Formula-Xinxing}
If a t-norm $T$ satisfies {Lemma~\ref{Cont-Char-Thm} ii-3)}, i.e.,
$T=\left(\langle a_\lambda, e_\lambda, T_\lambda \rangle\right)_{\lambda\in \mathscr{A}}$, then
\begin{enumerate}[{\rm (1)}]
 \item For $u \in (a_{\lambda}, e_{\lambda})$ and $t>0$,
\begin{align*}
u_{T}^{(t)}& =a_{\lambda} +(e_{\lambda} -a_{\lambda})
\left(\frac{u -a_{\lambda}}{e_{\lambda} -a_{\lambda}}\right)_{T_{\lambda}}^{(t)}\\
& =a_{\lambda} +(e_{\lambda} -a_{\lambda})G_{\lambda}^{( -1)}
\left(t\cdot G_{\lambda}\left(\frac{u -a_{\lambda}}
{e_{\lambda} -a_{\lambda}}\right)\right),
\end{align*}
where $G_{\lambda}$ is an AG of $T_{\lambda}$;

\item For $u\in [0, 1]\backslash \cup_{\lambda\in \mathscr{A}}(a_{\lambda}, e_{\lambda})$
and $t>0$, $u_{T}^{(t)}=u$.
\end{enumerate}
\end{theorem}

\begin{proof}
(1) Fix $z\in (a_{\lambda_0}, e_{\lambda_0})$ for some $\lambda_0\in \mathscr{A}$.
We first prove that
\begin{equation}
\label{eq-7.23}
\forall p\in \mathbb{N},\
z_{T}^{(\frac{1}{p})} =a_{\lambda_{0}} +(e_{\lambda_{0}} -a_{\lambda_{0}})
\cdot\left(\frac{z -a_{\lambda_{0}}}{e_{\lambda_{0}} -
a_{\lambda_{0}}}\right)_{T_{\lambda_{0}}}^{(\frac{1}{p})}.
\end{equation}

For any $x\in [0, 1]$,  it can be verified that
\begin{enumerate}[(a)]
\item If $x\in \left[0, a_{\lambda_{0}}\right]$,
$x_{T}^{(p)}\leq x\leq a_{\lambda_{0}}<z$;
\item If $x\in \left[e_{\lambda_{0}}, 1\right]$,
$x_{T}^{(p)}\geq (e_{\lambda_{0}})_{T}^{(p)}= e_{\lambda_{0}}>z$.
\end{enumerate}
This, together with Lemma~\ref{Power-Dual-Thm} (ii), implies that
$z_{T}^{(\frac{1}{p})}\in(a_{\lambda_{0}}, e_{\lambda_{0}})$. Meanwhile, for any
$x\in (a_{\lambda_{0}}, e_{\lambda_{0}})$, it is not difficult to check that
\begin{equation}
\label{eq-7.24}
x_{T}^{(p)} =a_{\lambda_{0}} +(e_{\lambda_{0}} -a_{\lambda_{0}})\cdot
\left(\frac{x -a_{\lambda_{0}}}{e_{\lambda_{0}} -a_{\lambda_{0}}}\right)_{T_{\lambda_{0}}}^{(p)}.
\end{equation}
Thus, for any $x\in[0, 1]$ with $x_{T}^{(p)}=z\in(a_{\lambda_{0}}, e_{\lambda_{0}})$, we have
$a_{\lambda_{0}} +(e_{\lambda_{0}} -a_{\lambda_{0}}) \cdot
\left(\frac{x -a_{\lambda_{0}}}{e_{\lambda_{0}} -a_{\lambda_{0}}}\right)_{T_{\lambda_{0}}}^{(p)} =z$,
implying that $0 <\frac{z -a_{\lambda_{0}}}{e_{\lambda_{0}} -a_{\lambda_{0}}}
 =\left(\frac{x -a_{\lambda_{0}}}{e_{\lambda_{0}} -a_{\lambda_{0}}}\right)_{T_{\lambda_{0}}}^{(p)}$.
Since $T_{\lambda_0}$ is a continuous Archimedean t-norm, by Corollary~\ref{Corollary-7.1}, we have
$$
\frac{x-a_{\lambda_{0}}}{e_{\lambda_{0}}-a_{\lambda_{0}}}
=\left(\frac{z-a_{\lambda_{0}}}{e_{\lambda_{0}}-a_{\lambda_{0}}}\right)_{T_{\lambda_{0}}}^{(\frac{1}{p})},
$$
implying that
$$
\left\{x\mid x_{T}^{(p)} =z\right\}
 =\left\{a_{\lambda_{0}} +(e_{\lambda_{0}} -a_{\lambda_{0}}) \cdot
\left(\frac{z -a_{\lambda_{0}}}{e_{\lambda_{0}}
 -a_{\lambda_{0}}}\right)_{T_{\lambda_{0}}}^{(\frac{1}{p})}\right\}.
$$
This, together with Lemma~\ref{Power-Dual-Thm} (ii), implies that
\begin{equation*}
z_{T}^{(\frac{1}{p})}=a_{\lambda_{0}}+(e_{\lambda_{0}}-a_{\lambda_{0}})
\cdot\left(\frac{z-a_{\lambda_{0}}}{e_{\lambda_{0}}-
a_{\lambda_{0}}}\right)_{T_{\lambda_{0}}}^{(\frac{1}{p})}.
\end{equation*}

Second, for any $\frac{q}{p}\in \mathbb{Q}^{+}$, by formulas~\eqref{eq-7.23}
and \eqref{eq-7.24} and Proposition~\ref{Power-t-Thm}, we have
\begin{align*}
z_{T}^{(\frac{q}{p})}&=
(z_{T}^{(\frac{1}{p})})_{T}^{(q)}\\
&=
\Big(a_{\lambda_{0}}+(e_{\lambda_{0}}-a_{\lambda_{0}})
\cdot\Big(\frac{z-a_{\lambda_{0}}}{e_{\lambda_{0}}-
a_{\lambda_{0}}}\Big)_{T_{\lambda_{0}}}^{(\frac{1}{p})}\Big)_{T}^{(q)}\\
&=a_{\lambda_{0}}+(e_{\lambda_{0}}-a_{\lambda_{0}})
\cdot\Big(\Big(\frac{z-a_{\lambda_{0}}}{e_{\lambda_{0}}-
a_{\lambda_{0}}}\Big)_{T_{\lambda_{0}}}^{(\frac{1}{p})}\Big)_{T_{\lambda_0}}^{(q)}\\
&=a_{\lambda_{0}}+(e_{\lambda_{0}}-a_{\lambda_{0}})
\cdot\Big (\frac{z-a_{\lambda_{0}}}{e_{\lambda_{0}}-
a_{\lambda_{0}}}\Big)_{T_{\lambda_{0}}}^{(\frac{q}{p})}.
\end{align*}
This, together with Theorem~\ref{Power-Limit-Thm} and
Proposition~\ref{Power-t-Thm}, implies that for any
$t \in (0, +\infty)$, we have $z_{T}^{(t)}=a_{\lambda_{0}}
(e_{\lambda_{0}}-a_{\lambda_{0}})
\cdot (\frac{z-a_{\lambda_{0}}}{e_{\lambda_{0}}-
a_{\lambda_{0}}})_{T_{\lambda_{0}}}^{(t)}=a_{\lambda_{0}}
+(e_{\lambda_{0}}-a_{\lambda_{0}})\cdot G_{\lambda_{0}}^{(-1)}
(t\cdot G_{\lambda_{0}}
(\frac{x-a_{\lambda_{0}}}{e_{\lambda_{0}}-a_{\lambda_{0}}}))$.

(2) It follows directly from Proposition~\ref{pr-7.3} and
$[0, 1]\backslash\cup_{\lambda\in \mathscr{A}}(a_{\lambda}, e_{\lambda})
\subseteq \mathfrak{I}_{T}$.
\end{proof}

\begin{corollary}
\label{Power-Continuous-Thm}
If the t-norm $T$ is continuous, then, for any $u \in [0, 1]$,
the function $u_{T}^{(\_)}$ is continuous on $(0, +\infty)$.
\end{corollary}

\begin{proof}
If $T =T_{\mathbf{M}}$ or $T\in \mathscr{T}_{_{\mathrm{ConA}}}$,
by Proposition~\ref{pr-7.3} and Theorem~\ref{Power-Limit-Thm},
$u_{T}^{(\_)}$ is continuous.

If $T$ is an ordinal sum of continuous Archimedean t-norms,
by Theorem~\ref{Power-Formula-Xinxing},
$u_{T}^{(\_)}$ is continuous.
\end{proof}

\begin{remark}
\label{Remark-4}
(1) Walker and Walker~\cite{WW2002} gave another form of power operation for continuous t-norms via
  multiplicative generators. By Theorem~\ref{Power-Formula-Xinxing}, it can be verified that our
  Definition~\ref{x-Power-Def} is equivalent to \cite[Definition~22]{WW2002}.

(2) Boixader and Recasens~\cite[Definition 4.2.1]{BR2016} defined the
{\it $n$-th root} $u_{T}^{(\frac{1}{n})}$ of $u$ for a continuous t-norm $T$ as follows:
  $$
  u_{T}^{(\frac{1}{n})}=\sup\left\{v  \in [0,1] \mid v_{T}^{(n)} \leq u\right\},
  $$
  and for $\frac{q}{p}\in \mathbb{Q}^{+}$, $u_{T}^{(\frac{q}{p})}=(u_{T}^{(\frac{1}{p})})_{T}^{(q)}$.
  In~\cite[Definition 4.2.3]{BR2016}, another method is developed to extend the
rational powers to irrational powers, as follows:
\begin{itemize}
  \item If $t\in(0, +\infty)$, let $\left\{t_{n}\right\}\subseteq \mathbb{Q}$
be a sequence with $\lim_{n\rightarrow +\infty}
t_{n}=t$. Given any $u \in [0, 1]$, the power $u_{T}^{(t)}$ is
\begin{equation}
\label{Eq-Limit}
u_{T}^{(t)}=\lim_{n\rightarrow +\infty}x_{T}^{(t_{n})}.
\end{equation}
\end{itemize}
Analogously to the above proof, it not difficult to check that, for the power $x_{T}^{(t)}$
defined by Boixader and Recasens, we have
\begin{enumerate}[(i)]
  \item If $T=T_{\mathbf{M}}$, then $u_{T_{\mathbf{M}}}^{(t)}=u$ holds for all $u\in [0, 1]$.
  \item If $T\in \mathscr{T}_{_{\mathrm{ConA}}}$ with an AG $G$, then $u_{T}^{(t)}
  =G^{(-1)}(t\cdot G(u))$.

  \item If $T$ is an ordinal sum of t-norms in $\mathscr{T}_{_{\mathrm{ConA}}}$,
  i.e., $T=\left(\langle a_\lambda, e_\lambda, T_\lambda \rangle\right)_{\lambda\in \mathscr{A}}$, then
  $$
  u_{T}^{(t)} =
  \begin{cases}
  a_{\lambda} + (e_{\lambda} -a_{\lambda}) \cdot G_{\lambda}^{( -1)}
\left(t \cdot G_{\lambda}\left(\frac{u -a_{\lambda}}{e_{\lambda} -a_{\lambda}}\right)\right),
 & u \in [a_{\lambda}, e_{\lambda}], \\
  u, & u \in [0, 1]\backslash \cup_{\lambda\in \mathscr{A}}[a_{\lambda}, e_{\lambda}],
  \end{cases}
  $$
  where $G_{\lambda}$ is an AG of $T_{\lambda}$.
\end{enumerate}
Meanwhile, they claimed that the continuity of $T$ assures that the limit
defined in Eq.~\eqref{Eq-Limit} independently exists from the choice
of the sequence $\left\{t_{n}\right\}$. Actually,
the limit $\lim_{n\to +\infty}u_{T}^{(t_{n})}$ exists and is
independent of the sequence $\left\{t_{n}\right\}$ if and only if
$u_{T}^{(\_)}|_{\mathbb{Q}^{+}}$ is continuous. Similarly to the proof of
Corollary~\ref{Power-Continuous-Thm}, by (i)--(iii),
it can be verified that this claim is true, i.e., \cite[Definition~4.2.3]{BR2016} is well defined.
\end{remark}

\begin{proposition}
\label{Power-Mon-Thm}
Assume that $T$ is a continuous t-norm and $S$ is a continuous t-conorm.
Then, for any $u \in [0, 1]$ and any $t_1$, $t_2 >0$, we have
$T(u_{T}^{(t_1)}, u_{T}^{(t_2)}) =u_{T}^{(t_1 +t_2)}$ and
$S(u_{S}^{(t_1)}, u_{S}^{(t_2)}) =u_{S}^{(t_1 +t_2)}$.
In particular, $u_{T}^{(\_)}$ is decreasing on $(0, +\infty)$
and $u_{S}^{(\_)}$ is increasing on $(0, +\infty)$.
\end{proposition}

\begin{proof}
(1) Take $r_{n}^{(1)}\in (0, t_1)\cap \mathbb{Q}$, $r_{n}^{(2)}\in (0, t_2)\cap \mathbb{Q}$
such that $\lim_{n\to +\infty} r_{n}^{(1)} = t_1$ and $\lim_{n\to +\infty}
r_{n}^{(2)} = t_2$. Clearly, $\lim_{n\to +\infty}(r_{n}^{(1)} +r_{n}^{(2)}) = (t_1 +t_2)$.
By Corollary~\ref{Power-Continuous-Thm}, we have
$\lim_{n\to +\infty}u_{T}^{(r_{n}^{(1)})} =u_{T}^{(t_1)}$,
$\lim_{n\to +\infty}u_{T}^{(r_{n}^{(2)})} =u_{T}^{(t_2)}$, and
$\lim_{n\to +\infty}u_{T}^{(r_{n}^{(1)} +r_{n}^{(2)})} =u_{T}^{(t_1 +t_2)}$.
Since $T$ is continuous, by Lemma~\ref{le-7.8}, we have
$T(u_{T}^{(t_1)}, u_{T}^{(t_2)})=T(\lim\limits_{n\to +\infty} u_{T}^{(r_{n}^{(1)})},
\lim\limits_{n\to +\infty} u_{T}^{(r_{n}^{(2)})})=\lim_{n\to +\infty}
T(u_{T}^{(r_{n}^{(1)})}, u_{T}^{(r_{n}^{(2)})})=\lim_{n\to +\infty}
u_{T}^{(r_{n}^{(1)}+r_{n}^{(2)})}= u_{T}^{(t_1+t_2)}$.

(2) Given a continuous t-conorm $S$, take its dual t-norm $T$.
By Theorem~\ref{Power-Dual-Thm-Main} and the above proof,
we have $S(u_{S}^{(t_1)}, u_{S}^{(t_2)})=1-T(1-u_{S}^{(t_1)}, 1-u_{S}^{(t_2)})
=1-T((1-u)_{T}^{(t_1)}, (1-u)_{T}^{(t_2)})=1-(1-u)_{T}^{(t_1+t_2)}
=u_{S}^{(t_1+t_2)}$.
\end{proof}

\begin{theorem}\label{th-7.5}
Let $T \in \mathscr{T}_{_{\mathrm{ConA}}}$. Then, for
$u$, $v \in [0, 1]$ and $t >0$, we have
\begin{equation}
(T(u, v))_{T}^{(t)}=
\begin{cases}
0, & T(u, v) =0, \\
T(u_{T}^{(t)}, v_{T}^{(t)}), & T(u, v)  \in (0, 1].
\end{cases}
\end{equation}
\end{theorem}
\begin{proof}
By $T\in \mathscr{T}_{_{\mathrm{ConA}}}$ and Lemma~\ref{Cont-Archi-Char},
it follows that there exists an AG
$G: [0, 1]\rightarrow [0, +\infty]$ for $T$.

(1) If $T(u, v)=0$, by Proposition~\ref{0-Power-Pro},
$(T(u, v))_{T}^{(t)}=0$.

(2) If $T(u, v)\in (0, 1]$, then $u >0$ and $v >0$. Meanwhile,
it can be verified that $G(u)+G(v)< G(0)$, since
$G(u)+G(v)\geq G(0)$ implies $T(u, v)=
G^{(-1)}(G(u)+G(v))=0$.
Together with Lemma~\ref{Pseudo-Inverse-Lemma}, we have
$T(u, v)=G^{(-1)}(G(u)+G(v))
=G^{-1}(G(u)+G(v))$.

To prove that $(T(u, v))_{T}^{(t)}=T(u_{T}^{(t)}, v_{T}^{(t)})$
holds for all $t\in (0, +\infty)$, consider the following three cases:

{2.1) For any $p\in \mathbb{N}$, by Lemma~\ref{Power-1/n-Thm}, we have
$(T(u, v))_{T}^{(\frac{1}{p})} =G^{-1}
(\frac{G(T(u, v))}{p})
 =G^{-1}(\frac{G(G^{-1}
(G(u) +G(v)))}{p})
 =G^{-1}(\frac{G(u) +G(v)}{p})$
and
$u_{T}^{(\frac{1}{p})}=G^{-1}(\frac{G(u)}{p})$,
$v_{T}^{(\frac{1}{p})}=G^{-1}(\frac{G(v)}{p})$,
implying that
\begin{align*}
& T(u_{T}^{(\frac{1}{p})}, v_{T}^{(\frac{1}{p})})
 =G^{(-1)}\Big(G(u_{T}^{(\frac{1}{p})}) +G(v_{T}^{(\frac{1}{p})})\Big) \\
 =&G^{(-1)}\Big(G \circ G^{-1}\Big(\frac{G(u)}{p}\Big)
 +G \circ G^{-1}\Big(\frac{G(v)}{p}\Big)\Big) \\
 =&G^{(-1)}\Big(\frac{G(u) +G(v)}{p}\Big) \\
 =&G^{-1}\Big(\frac{G(u) +G(v)}{p}\Big)
\quad \text{(by $G(u) +G(v) < G(0)$)} \\
 =&(T(u, v))_{T}^{(\frac{1}{p})}.
\end{align*}

2.2) For any $\frac{q}{p}\in \mathbb{Q}^{+}$,
according to the above 2.1), we have
$(T(u, v))_{T}^{(\frac{q}{p})}
 =((T(u, v))_{T}^{(\frac{1}{p})})_{T}^{(q)}
 =(T(u_{T}^{(\frac{1}{p})}, v_{T}^{(\frac{1}{p})}))_{T}^{(q)}
 =T((u_{T}^{(\frac{1}{p})})_{T}^{(q)}, (v_{T}^{(\frac{1}{p})})_{T}^{(q)})
 =T(u_{T}^{(\frac{q}{p})}, v_{T}^{(\frac{q}{p})})$.
}

2.3) For any $t\in (0, +\infty)\backslash \mathbb{Q}$,
by Definition~\ref{x-Power-Def} and the above 2.2),
since $T$ is continuous and increasing, we have
$(T(u, v))_{T}^{(t)}=\inf\{(T(u, v))_{T}^{(r)}\mid r\in [0, t]\cap\mathbb{Q}\}
=\inf\{T(u_{T}^{(r)}, v_{T}^{(r)})\mid r\in [0, t]\cap \mathbb{Q}\}
=T(u_{T}^{(t)}, v_{T}^{(t)})$.
\end{proof}

\begin{corollary}
\label{Strict-Power-Stab-Thm}
Let $T$ be a strict t-norm. Then, for any $u$, $v\in [0, 1]$ and
any $t\in (0, +\infty)$, we have $(T(u, v))_{T}^{(t)}=T(u_{T}^{(t)}, v_{T}^{(t)})$.
\end{corollary}
\begin{proof}
If $T(u, v)\in (0, 1]$, by Lemma~\ref{Strict->Archi} and Theorem~\ref{th-7.5},
it is clear that $(T(u, v))_{T}^{t}=T(u_{T}^{(t)}, v_{T}^{(t)})$. If $T(u, v)=0$,
since $T$ is a strict t-norm, we have $u=0$ or $v=0$, and thus $u_{T}^{(t)}\equiv 0$
or $v_{T}^{(t)}\equiv 0$. This, together with Theorem~\ref{th-7.5}, implies that
$T(u_{T}^{(t)}, v_{T}^{(t)})=0=(T(u, v))_{T}^{(t)}$.
\end{proof}

\begin{corollary}\label{co-7.3}
Let $T\in \mathscr{T}_{_{\mathrm{ConA}}}$. If $T$ is not strict,
then for any $a\in (0, 1)$, there exists $N \in \mathbb{N}$ satisfying the following:
\begin{enumerate}[{\rm (i)}]
\item $a_{T}^{(\frac{N}{2})}> 0$;

\item $a_{T}^{(N)} =(T(a_{T}^{(\frac{N}{2})}, a_{T}^{(\frac{N}{2})}))_{T}^{(\frac{1}{2})}
     =((a_{T}^{(\frac{N}{2})})_{T}^{(2)})_{T}^{(\frac{1}{2})} =0$;

\item $T((a_{T}^{(\frac{N}{2})})_{T}^{(\frac{1}{2})},
(a_{T}^{(\frac{N}{2})})_{T}^{(\frac{1}{2})})
 =((a_{T}^{(\frac{N}{2})})_{T}^{(\frac{1}{2})})_{T}^{(2)} >0$.
\end{enumerate}
In particular, there exist $u$, $v\in [0, 1]$ and $t\in (0, 1)$
such that $(T(u, v))_{T}^{(t)} \neq T(u_{T}^{(t)}, v_{T}^{(t)})$.
\end{corollary}

\begin{proof}
Since $T$ is not strict, by Lemma~\ref{Strict-Nil-Char}, $T$
is nilpotent, i.e., each $a\in (0,1)$ is a nilpotent element of $T$.
Fix $a\in (0, 1)$ and take $N=\min\{n \in \mathbb{N} \mid a_{T}^{(n)}=0\}$.
Clearly, $\frac{N}{2} \leq N-1$. This, together with
Lemma~\ref{le-7.8}, implies that
$a_{T}^{(\frac{N}{2})}\geq a_{T}^{(N-1)}>0$. Thus,
$(T(a_{T}^{(\frac{N}{2})}, a_{T}^{(\frac{N}{2})}))_{T}^{(\frac{1}{2})}
=0_{T}^{(\frac{1}{2})}=0$ and
$T((a_{T}^{(\frac{N}{2})})_{T}^{(\frac{1}{2})},
(a_{T}^{(\frac{N}{2})})_{T}^{(\frac{1}{2})})
=((a_{T}^{(\frac{N}{2})})_{T}^{(\frac{1}{2})})_{T}^{(2)}
=a_{T}^{(\frac{N}{2})}>0.$
\end{proof}


\begin{theorem}
\label{Power-Stab-Char}
For a continuous t-norm $T$, the following
statements are equivalent:
\begin{enumerate}[{\rm (i)}]
  \item $T$ is power stable, i.e.,
the formula~\eqref{eq-Power-Stab} holds for all $(u, v)\in[0, 1]$ and all $t\in(0, +\infty)$.
  \item For any $u\in [0, 1]$, $u$ is power stable,
i.e., for any $t_{1}$, $t_{2}\in(0, +\infty), (u_{T}^{(t_{1})})_{T}^{(t_{2})}=u_{T}^{(t_{1}t_{2})}$.
  \item One of the following statements holds:
  \begin{enumerate}
    \item[{\rm iii-1)}] $T=T_{\mathbf{M}}$;
    \item[{\rm iii-2)}] $T$ is a strict t-norm;
    \item[{\rm iii-3)}] $T$ is an ordinal sum of strict t-norms,
i.e., there exists a (finite or countably infinite) index set $\mathscr{A}$,
a family of strict t-norms $\left\{T_{\lambda}\right\}_{\lambda\in \mathscr{A}}$,
and a family of pairwise disjoint open subintervals $\left\{(a_{\lambda}, e_{\lambda})
\right\}_{\lambda\in \mathscr{A}}$ of $[0, 1]$, such that
\begin{equation}
\label{eq-7.22}
T(u, v)= \begin{cases}
a_{\lambda} +(e_{\lambda} -a_{\lambda}) \cdot T_{\lambda}
\left(\frac{u -a_{\lambda}}{e_{\lambda} -a_{\lambda}},
\frac{v -a_{\lambda}}{e_{\lambda} -a_{\lambda}}\right), & (u, v)  \in [a_{\lambda}, e_{\lambda}]^{2}, \\
\min\{u, v\}, & \text{otherwise}.
\end{cases}
\end{equation}
\end{enumerate}
\end{enumerate}
\end{theorem}

\begin{proof}
(iii)$\Longrightarrow$(i).

\medskip

 iii-1)$\Longrightarrow$(i) and iii-2)$\Longrightarrow$(i) follow directly
 from Lemma~\ref{le-7.12} and  Corollary~\ref{Strict-Power-Stab-Thm}.

\medskip

 iii-3)$\Longrightarrow$(i). For any $u$, $v\in [0, 1]$ and
 $t>0$, formula~\eqref{eq-Power-Stab} always holds when $u=0$ or $v=0$;
 otherwise, consider the following cases:

\begin{itemize}
\item If $(u, v)\in\cup_{\lambda\in \mathscr{A}}\left[a_{\alpha}, e_{\lambda}\right]^{2}$,
then there exists $\lambda_{0} \in \mathscr{A}$ such that
$(u, v)\in\left[a_{\lambda_{0}}, e_{\lambda_{0}}\right]^{2}$.
Clearly, $T(u, v)\in \left[a_{\lambda_{0}}, e_{\lambda_{0}}\right]$.
\begin{itemize}
\item If $T(u, v)=a_{\lambda_{0}}$, since $T_{\lambda_{0}}$ is strict,
$u=a_{\lambda_{0}}$ or $v=a_{\lambda_{0}}$. This, together with
Proposition~\ref{pr-7.3} and $a_{\lambda_{0}}\in \mathfrak{I}_{_T}$,
implies that, for any $t> 0$,
$$
(T(u, v))_{T}^{(t)} =a_{\lambda_{0}} \text{ and }
(u_{T}^{(t)} =a_{\lambda_{0}} \text{ or } v_{T}^{(t)} =a_{\lambda_{0}}),
$$
and thus $(T(u, v))_{T}^{(t)}=a_{\lambda_{0}}=T(u_{T}^{(t)}, v_{T}^{(t)})$.

\item If $T(u, v)=e_{\lambda_{0}}$, then $u=v=e_{\lambda_{0}}$. This, together
with Proposition~\ref{pr-7.3} and $e_{\lambda_{0}}\in \mathfrak{I}_{T}$,
implies that, for any $t> 0$, $T(u_{T}^{(t)}, v_{T}^{(t)})
 =T(e_{\lambda_{0}}, e_{\lambda_{0}}) =e_{\lambda_{0}}
 =(e_{\lambda_{0}})_{T}^{(t)} =(T(u, v))_{T}^{(t)}$.

\item If $T(u, v)\in(a_{\lambda_{0}}, e_{\lambda_{0}})$, then
$u\in(a_{\lambda_{0}}, e_{\lambda_{0}}]$ and $v\in(a_{\lambda_{0}}, e_{\lambda_{0}}]$.
By formula~\eqref{eq-7.23}, since $T_{\lambda_{0}}$ is strict,
for any $p\in\mathbb{N}$, we have
\begin{align*}
&\quad (T(u, v))_{T}^{(\frac{1}{p})}\\
& =a_{\lambda_{0}} +(e_{\lambda_{0}} -a_{\lambda_{0}}) \cdot\Big(\frac{T(u, v) -a_{\lambda_{0}}}{e_{\lambda_{0}} -a_{\lambda_{0}}}\Big)_{T_{\lambda_{0}}}^{(\frac{1}{p})}~~\text{(by Eq.~\eqref{eq-7.23})} \\
& =a_{\lambda_{0}} +(e_{\lambda_{0}} -a_{\lambda_{0}})\cdot
\Big(T_{\lambda_{0}}\Big(\frac{u -a_{\lambda_{0}}}{e_{\lambda_{0}} -a_{\lambda_{0}}},
\frac{v -a_{\lambda_{0}}}{e_{\lambda_{0}}
 -a_{\lambda_{0}}}\Big)\Big)_{T_{\lambda_{0}}}^{(\frac{1}{p})}
~~\text{(by Eq.~\eqref{eq-7.22})} \\
& =a_{\lambda_{0}} +(e_{\lambda_{0}} -a_{\lambda_{0}}) \cdot T_{\lambda_{0}}\Big(\Big(\frac{u -a_{\lambda_{0}}}{e_{\lambda_{0}}
 -a_{\lambda_{0}}}\Big)_{T_{\lambda_{0}}}^{(\frac{1}{p})},
\Big(\frac{v -a_{\lambda_{0}}}{e_{\lambda_{0}} -a_{\lambda_{0}}}\Big)
_{T_{\lambda_{0}}}^{(\frac{1}{p})}\Big)~~\text{(by Corollary~\ref{Strict-Power-Stab-Thm})} \\
& =a_{\lambda_{0}} +(e_{\lambda_{0}} -a_{\lambda_{0}}) \cdot
T_{\lambda_{0}}\Big(\frac{u_{T}^{(\frac{1}{p})} -a_{\lambda_{0}}}{e_{\lambda_{0}}
-a_{\lambda_{0}}}, \frac{v_{T}^{(\frac{1}{p})} -a_{\lambda_{0}}}{e_{\lambda_{0}} -a_{\lambda_{0}}}\Big)
~~\text{(by Eq.~\eqref{eq-7.23})}\\
& =T\Big(u_{T}^{(\frac{1}{p})}, v_{T}^{(\frac{1}{p})}\Big)
\quad \text{(by Eq.~\eqref{eq-7.22})}.
\end{align*}
Then, for any $\frac{q}{p}\in\mathbb{Q}^{+}$, we have
$(T(u, v))_{T}^{(\frac{q}{p})}
 =((T(u, v))_{T}^{(\frac{1}{p})})_{T}^{(q)}
 =(T (u_{T}^{(\frac{1}{p})}, v_{T}^{(\frac{1}{p})}))_{T}^{(q)}
 =T ((u_{T}^{(\frac{1}{p})})_{T}^{(q)}, (v_{T}^{(\frac{1}{p})})_{T}^{(q)})
 =T (u_{T}^{(\frac{q}{p})}, v_{T}^{(\frac{q}{p})})$.
Further, for any $t> 0$, since $T$ is increasing and continuous, we have
\begin{align*}
(T(u, v))_{T}^{(t)}
& =\inf \left\{(T(u, v))_{T}^{(r)} \mid r \in [0, t] \cap \mathbb{Q}\right\} \\
& =\inf \left\{T(u_{T}^{(r)}, v_{T}^{(r)}) \mid r\in [0, t] \cap \mathbb{Q} \right\} \\
& =T\Big(\mathrm{inf}\{u_{T}^{(r)} \mid r\in [0, t] \cap \mathbb{Q}\},\\
&\quad \quad~~
\inf\{v_{T}^{(r)} \mid r\in [0, t]\cap \mathbb{Q}\}\Big) \\
& =T(u_{T}^{(t)}, v_{T}^{(t)}).
\end{align*}
\end{itemize}
\end{itemize}
\begin{itemize}
\item If $(u, v)\in(0, 1]^2\backslash \cup_{\lambda\in \mathscr{A}}
\left[a_{\lambda}, e_{\lambda}\right]^{2}$, by formula~\eqref{eq-7.22},
$T(u, v)=\min\{u, v\}$. Without loss of generality, assume that
$u\leq v$. Then, $T(u, v)=u$. For any $t>0$, consider the following
two subcases:
\begin{itemize}

\item If $u\in[0, 1]\backslash \cup_{\lambda\in \mathscr{A}}
\left[a_{\lambda}, e_{\lambda}\right]$, then
$T(u, u) =\min\{u, u\}=u$. i.e., $u\in\mathfrak{I}_{T}$.
This, together with Proposition~\ref{pr-7.3}, implies that
$(T(u, v))_{T}^{(t)}=u_{T}^{(t)}=u$. Furthermore, by
Proposition~\ref{Power-Increasing-Pro-1}, since $T$ is increasing,
$T(u_{T}^{(t)}, v_{T}^{(t)})=T(u, v_{T}^{(t)})\geq T(u, u)=u$.
Clearly, $T(u_{T}^{(t)}, v_{T}^{(t)})\leq u_{T}^{(t)}=u$,
and thus $T(u_{T}^{(t)}, v_{T}^{(t)})=u=(T(u, v))_{T}^{(t)}$.

\item If $u\in \cup_{\lambda\in \mathscr{A}}\left[a_{\lambda}, e_{\lambda}\right]$,
i.e., $u \in \left[a_{\lambda_{0}}, e_{\lambda_{0}}\right]$ for some $\lambda_{0}\in \mathscr{A}$,
by $(u, v) \in (0, 1]^2\backslash \cup_{\lambda\in \mathscr{A}}\left[a_{\lambda}, e_{\lambda}\right]^{2}$,
we have $v \in [0, 1]\backslash [a_{\lambda_{0}}, e_{\lambda_{0}}]$,
i.e., $v <a_{\lambda_{0}}$ or $v >e_{\lambda_{0}}$, and thus $v >e_{\lambda_{0}}$ since $u\leq v$.
Therefore, by Propositions~\ref{Power-Increasing-Pro-1} and \ref{pr-7.3}
and $e_{\lambda_{0}}\in\mathfrak{I}_{T}$, we have
$v_{T}^{(t)}\geq(e_{\lambda_{0}})_{T}^{(t)}=e_{\lambda_{0}}.$
This, together with the fact that $T$ is increasing, implies that
$u_{T}^{(t)}\geq T(u_{T}^{(t)}, v_{T}^{(t)})\geq
T(u_{T}^{(t)}, e_{\lambda_{0}})=u_{T}^{(t)}$,
i.e., $T(u_{T}^{(t)}, v_{T}^{(t)}) =u_{T}^{(t)}$.
Clearly, $(T(u, v))_{T}^{(t)} =u_{T}^{(t)}$.
Hence, $(T(u, v))_{T}^{(t)} =T(u_{T}^{(t)}, v_{T}^{(t)})$.
\end{itemize}
\end{itemize}

Summing up the above, it is concluded formula~\eqref{eq-Power-Stab}
holds for all $(u, v)\in [0, 1]$ and all $t >0$.

\medskip

(i)$\Longrightarrow$(iii).

\medskip

Sine $T$ is continuous, by Lemma~\ref{Cont-Char-Thm}, one of the following statements holds:

iii-1$^{\prime}$) $T=T_{\mathbf{M}}$;

iii-2$^{\prime}$) $T\in \mathscr{T}_{_{\mathrm{ConA}}}$;

iii-3$^{\prime}$) $T$ is an ordinal sum
  of continuous Archimedean t-norms, i.e.,
  $T=(\langle a_{\lambda}, a_{\lambda}, T_{\lambda}\rangle)_{\lambda\in \mathscr{A}}$,
  where $\left\{T_{\lambda}\right\}_{\lambda\in \mathscr{A}}$ is a family t-norms
  in $\mathscr{T}_{_{\mathrm{ConA}}}$.

\begin{itemize}
\item If iii-1$^{\prime}$) holds, then iii-1) holds.
\item If iii-2$^{\prime}$) holds, by Corollary~\ref{co-7.3},
$T$ is strict, i.e., iii-2) holds.
\item If iii-3$^{\prime}$) holds, it can be shown that each
$T_{\lambda}$ ($\lambda\in \mathscr{A}$) is strict. In fact, suppose on the contrary that there
exists some $\lambda_{0}\in \mathscr{A}$ such that $T_{\lambda_{0}}$ is not strict. Fix any
$a\in (0, 1)$, by Corollary~\ref{co-7.3}, there exists $N\in \mathbb{N}$
satisfying the following:
\begin{enumerate}[(a)]
\item $a_{T_{\lambda_{0}}}^{(\frac{N}{2})}>0$;
\item $a_{T_{\lambda_{0}}}^{(N)}=(T_{\lambda_{0}}(a_{T_{\lambda_{0}}}^{(\frac{N}{2})},
a_{T_{\lambda_{0}}}^{(\frac{N}{2})}))_{T_{\lambda_{0}}}^{(\frac{1}{2})}=0$;
\item $T_{\lambda_{0}}((a_{T_{\lambda_{0}}}^{(\frac{N}{2})})_{T_{\lambda_{0}}}^{(\frac{1}{2})}, (a_{T_{\lambda_{0}}}^{(\frac{N}{2})})_{T_{\lambda_{0}}}^{(\frac{1}{2})})>0$.
\end{enumerate}
Take $\hat{x}=\hat{y}=a_{\lambda_{0}}+(e_{\lambda_{0}}-a_{\lambda_{0}})\cdot
a_{T_{{\lambda_{0}}}}^{(\frac{N}{2})}$. Clearly, $\hat{x}$, $\hat{y}\in(a_{\lambda_{0}},
e_{\lambda_{0}})$. First, it can be verified that
\begin{equation}
\label{eq-7.26}
\begin{split}
&(T(\hat{x}, \hat{y}))_{T}^{(\frac{1}{2})}\\
 =&\Big(a_{\lambda_{0}} +(e_{\lambda_{0}} -a_{\lambda_{0}}) \cdot T_{\lambda_{0}}
\Big(\frac{\hat{x} -a_{\lambda_{0}}}{e_{\lambda_{0}} -a_{\lambda_{0}}},
\frac{\hat{y} -a_{\lambda_{0}}}{e_{\lambda_{0}} -a_{\lambda_{0}}}\Big)\Big)_{T}^{(\frac{1}{2})} \\
  =&\left(a_{\lambda_{0}} +(e_{\lambda_{0}} -a_{\lambda_{0}}) \cdot
T_{\lambda_{0}}\left(a_{T_{\lambda_{0}}}^{(\frac{N}{2})},
a_{T_{\lambda_{0}}}^{(\frac{N}{2})}\right)\right)_{T}^{(\frac{1}{2})} \\
  =&\left(a_{\lambda_{0}} +(e_{\lambda_{0}} -a_{\lambda_{0}})\cdot
a_{T_{\lambda_{0}}}^{(N)}\right)_{T}^{(\frac{1}{2})} \\
  =&(a_{\lambda_{0}})_{T}^{(\frac{1}{2})} =a_{\lambda_{0}}.
\end{split}
\end{equation}
For any $z\in[0, 1]$ such that $z_{T}^{(2)}=\hat{x} \in (a_{\lambda_{0}} ,e_{\lambda_{0}})$,
noting that
\begin{align*}
z_{T}^{(2)}& =a_{\lambda_{0}} +(e_{\lambda_{0}} -a_{\lambda_{0}})\cdot T_{\lambda_{0}}
\left(\frac{z -a_{\lambda_{0}}}{e_{\lambda_{0}} -a_{\lambda_{0}}},
\frac{z -a_{\lambda_{0}}}{e_{\lambda_{0}} -a_{\lambda_{0}}}\right)\\
& =a_{\lambda_{0}} +(e_{\lambda_{0}} -a_{\lambda_{0}})
 \cdot\bigg(\frac{z -a_{\lambda_{0}}}{e_{\lambda_{0}}
 -a_{\lambda_{0}}}\bigg)_{T_{\lambda_{0}}}^{(2)},
\end{align*}
we have
$\left(\frac{z-a_{\lambda_{0}}}{e_{\lambda_{0}}-a_{\lambda_{0}}}\right)_{T_{\lambda_{0}}}^{(2)}
=a_{T_{\lambda_{0}}}^{(\frac{N}{2})}>0$.
Since $T_{\lambda_{0}}$ is a continuous Archimedean t-norm,
by Corollary~\ref{Corollary-7.1}, we get
$z =a_{\lambda_{0}} +(e_{\lambda_{0}} -a_{\lambda_{0}}) \cdot
\left(a_{T_{\lambda_{0}}}^{(\frac{N}{2})}\right)_{T_{\lambda_{0}}}^{(\frac{1}{2})}$.
This, together with Lemma~\ref{Power-Dual-Thm} (ii), implies that
\begin{align*}
\hat{x}_{T}^{(\frac{1}{2})}& =\hat{y}_{T}^{(\frac{1}{2})} =\min\{z\in[0, 1]\mid z_{T}^{(2)} =\hat{x}\}\\
& =a_{\lambda_{0}} +(e_{\lambda_{0}} -a_{\lambda_{0}}) \cdot
\left(a_{T_{\lambda_{0}}}^{(\frac{N}{2})}\right)_{T_{\lambda_{0}}}^{(\frac{1}{2})}
\in[a_{\lambda_{0}}, e_{\lambda_{0}}].
\end{align*}
Therefore,
\begin{equation}
\label{eq-7.27}
\begin{split}
& T(\hat{x}_{T}^{(\frac{1}{2})}, \hat{y}_{T}^{(\frac{1}{2})})\\
 =& a_{\lambda_{0}} +(e_{\lambda_{0}} -a_{\lambda_{0}}) \cdot T_{a_{\lambda_{0}}}
\Big(\Big(a_{T_{\lambda_{0}}}^{(\frac{N}{2})}
\Big)_{T_{\lambda_{0}}}^{(\frac{1}{2})},
\Big(a_{T_{\lambda_{0}}}^{(\frac{N}{2})}\Big)_{T_{\lambda_{0}}}^{(\frac{1}{2})}\Big)\\
 =&a_{\lambda_{0}} +(e_{\lambda_{0}} -a_{\lambda_{0}}) \cdot a_{T_{\lambda_{0}}}^{(\frac{N}{2})}\\
 >&a_{\lambda_{0}} =(T(\hat{x}, \hat{y}))_{T}^{(\frac{1}{2})} \quad
\text{(by $a_{T_{\lambda_{0}}}^{(\frac{N}{2})}>0$)},
\end{split}
\end{equation}
which contradicts with (i).
\end{itemize}

(ii)$\Longrightarrow$(iii). Similarly to the proof of (i)$\Longrightarrow$(iii),
it suffices to prove that iii-3$^{\prime}$)$+$(ii)$\Longrightarrow$iii-3),
since iii-1$^{\prime}$)$\Longrightarrow$iii-1) and
iii-3$^{\prime}$)$+$(ii)$\Longrightarrow$iii-3) hold trivially.

\medskip

If iii-3$^{\prime}$) holds, it can be shown that each
$T_{\lambda}$ ($\lambda\in \mathscr{A}$) is strict. In fact, suppose on the
contrary that there exists some $\lambda_{0}\in \mathscr{A}$ such that $T_{\lambda_{0}}$ is not strict.
By formulas~\eqref{eq-7.26} and \eqref{eq-7.27}, we have
$(\hat{x}_{T}^{(2)})_{T}^{(\frac{1}{2})}=a_{\lambda_{0}}<(\hat{x}_{T}^{(\frac{1}{2})})_{T}^{(2)}$,
which contradicts (ii).

\medskip

(i)+(iii)$\Longrightarrow$(ii). If (i) holds, by using mathematical induction,
it can be verified that, for any $u_{1}$, $u_{2}$, $\ldots, u_{k} \in [0, 1]$
and any $t> 0$,
\begin{equation}
\label{eq-7.28}
(T(u_{1}, u_{2}, \ldots, u_{k}))_{T}^{(t)} =
T((u_{1})_{T}^{(t)}, (u_{2})_{T}^{(t)}, \ldots, (u_{k})_{T}^{(t)}).
\end{equation}

If iii-1) or iii-2) holds, i.e., $T=T_{\mathbf{M}}$ or $T$ is strict,
by Propositions~\ref{pr-7.3}
and \ref{Power-t-Thm}, it is clear that (ii) holds. Otherwise, if iii-3)
holds, consider the following two cases for any fixed $u\in[0, 1]$.
\begin{itemize}
\item If $u\in[0, 1]\backslash \cup_{\lambda\in \mathscr{A}}(a_{\lambda}, e_{\lambda})$,
then $u$ is an idempotent element, i.e., $u\in \mathfrak{I}_{T}$.
By Proposition~\ref{pr-7.3}, $(u_{T}^{(t_{1})})_{T}^{(t_{2})} =u =
u_{T}^{(t_{1}t_{2})}$ holds for any $t_{1}$, $t_{2} >0$.

\item If $u \in \cup_{\lambda\in \mathscr{A}}(a_{\lambda}, e_{\lambda})$,
i.e., $u \in (a_{\lambda_{0}}, e_{\lambda_{0}})$ for some
$\lambda_{0}\in \mathscr{A}$, then

\begin{itemize}
\item For any $\frac{q_{1}}{p_{1}}$, $\frac{q_{2}}{p_{2}}\in \mathbb{Q}^{+}$,
we have
\begin{equation}
\label{eq-7.30}
\begin{split}
&\quad (u_{T}^{(\frac{q_{1}}{p_{1}})})_{T}^{(\frac{q_{2}}{p_{2}})}\\
& =(T^{(q_{1})}(u_{T}^{(\frac{1}{p_{1}})},\ldots,
u_{T}^{(\frac{1}{p_{1}})}))_{T}^{(\frac{q_{2}}{p_{2}})}\\
& =T^{(q_{1})}((u_{T}^{(\frac{1}{p_{1}})})_{T}^{(\frac{q_{2}}{p_{2}})},\ldots,
(u_{T}^{(\frac{1}{p_{1}})})_{T}^{(\frac{q_{2}}{p_{2}})})
~~\text{(by Eq.~\eqref{eq-7.28})}\\
&  =T^{(q_{1})}(((u_{T}^{(\frac{1}{p_{1}})})_{T}^{(\frac{1}{p_{2}})})_{T}^{(q_{2})},
\ldots, ((u_{T}^{(\frac{1}{p_{1}})})_{T}^{(\frac{1}{p_{2}})})_{T}^{(q_{2})})\\
&  =T^{(q_{1})}((u_{T}^{(\frac{1}{p_{1}p_{2}})})_{T}^{(q_{2})},\ldots,
(u_{T}^{(\frac{1}{p_{1}p_{2}})})_{T}^{(q_{2})})~~\text{(by Proposition~\ref{pr-7.6})}\\
&  =(u_{T}^{(\frac{1}{p_{1}p_{2}})})_{T}^{(q_{1}q_{2})}
 =u_{T}^{(\frac{q_{1}q_{2}}{p_{1}p_{2}})}
 =u_{T}^{(\frac{q_{1}}{p_{1}}\cdot\frac{q_{2}}{p_{2}})}.
\end{split}
\end{equation}

\item For any $t_{1}$, $t_{2}>0$, by Definition~\ref{x-Power-Def},
we have
$(u_{T}^{(t_{1})})_{T}^{(t_{2})} =(\inf\{u_{T}^{(r_{1})}\mid r_{1}\in
[0, t_{1}]\cap\mathbb{Q}\})_{T}^{(r_{2})} =\inf\{(\inf\{u_{T}^{(r_{1})}\mid r_{1}
\in[0, t_{1}]\cap\mathbb{Q}\})_{T}^{(r_{2})}\mid r_{2}\in [0, t_{2}]\cap \mathbb{Q}\}$.
 This, together with Proposition~\ref{Power-Increasing-Pro-1}, Lemma~\ref{le-7.8},
 and formula~\eqref{eq-7.30}, implies that
\begin{itemize}
\item For any $r_{1}\in[0, t_{1}]\cap \mathbb{Q}$ and any
$r_{2}\in[0, t_{2}]\cap\mathbb{Q}$, $(u_{T}^{(t_{1})})_{T}^{(t_{2})}
\leq (u_{T}^{(r_{1})})_{T}^{(r_{2})}=u_{T}^{(r_{1}r_{2})}$;

\item For any $r_{1}^{\prime}\in[t_{1}, +\infty)\cap\mathbb{Q}$ and
any $r_{2}^{\prime}\in[t_{2}, +\infty) \cap \mathbb{Q}$,
$(u_{T}^{(t_{1})})_{T}^{(t_{2})} \geq (u_{T}^{(t_{1})})_{T}^{(r_{2}^{\prime})} \geq
(u_{T}^{(r_{1}^{\prime})})_{T}^{(r_{2}^{\prime})} = u_{T}^{(r_{1}^{\prime}r_{2}^{\prime})}$.
\end{itemize}
Since $u_{T}^{(\_)}$ is continuous (by Corollary~\ref{Power-Continuous-Thm}), we have
$$
\lim_{\substack{\mathbb{Q}\ni r_{1}\nearrow t_{1}\\
\mathbb{Q}\ni r_{2}\nearrow t_{2}}}u_{T}^{(r_{1}r_{2})}
 =u_{T}^{(t_{1}t_{2})} \geq(u_{T}^{(t_{1})})_{T}^{(t_{2})},
$$
and
$$
\lim_{\substack{\mathbb{Q}\ni r_{1}^{\prime}\searrow t_{1}\\
\mathbb{Q}\ni r_{2}^{\prime}\searrow t_{2}}}u_{T}^{(r_{1}^{\prime}r_{2}^{\prime})}
 =u_{T}^{(t_{1}t_{2})} \leq (u_{T}^{(t_{1})})_{T}^{(t_{2})}.
$$
Therefore, $(u_{T}^{(t_{1})})_{T}^{(t_{2})}=u_{T}^{(t_{1}t_{2})}$.
\end{itemize}
\end{itemize}
\end{proof}

\section{New operational laws over IFVs}\label{Sec-IV}

By using the power operation of continuous t-norms formulated in Section~\ref{Sec-III},
the following operational laws are introduced for $\mathrm{IFVs}$,
which generalize Definition~\ref{Def-Int-Operations}.

\begin{definition}\label{def-8.1}
 Let $\alpha =\langle\mu_{\alpha},\nu_{\alpha}\rangle$,
 $\beta =\langle\mu_{\beta},\nu_{\beta}\rangle \in
 \tilde{\mathbb{I}}$, and $T$ be a continuous t-norm.
 For $\lambda>0$, define
\begin{enumerate}[(i)]
\item $\alpha\oplus_{_T} \beta
=\langle S(\mu_{\alpha},\mu_{\beta}),T(\nu_{\alpha},\nu_{\beta})\rangle$;
\item $\alpha\otimes_{_T}\beta
=\langle T(\mu_{\alpha},\mu_{\beta}),S(\nu_{\alpha},\nu_{\beta})\rangle$;
\item $\lambda_{_T}\alpha
=\langle (\mu_{\alpha})_{S}^{(\lambda)},(\nu_{\alpha})_{T}^{(\lambda)}\rangle$;
\item $\alpha^{\lambda_{_T}}
=\langle (\mu_{\alpha})_{T}^{(\lambda)},(\nu_{\alpha})_{S}^{(\lambda)}\rangle$;
\end{enumerate}
where $S$ is the dual t-conorm of $T$ and $u_{T}^{(\_)}$ and $u_{S}^{(\_)}$
are defined in Definition~\ref{x-Power-Def}.
\end{definition}

\begin{remark}
\label{remk5}
(1) The operations $\oplus_{_T}$ and $\otimes_{_T}$ were first introduced by
Deschrijver and Kerre~\cite{DK2002}, which were called
generalized union and generalized intersection, respectively. However, the multiplication
and power operations ((iii) and (iv)) of IFVs were not defined for general
t-norms. To date, all operations on $\mathrm{IFVs}$ are considered only
for some special families of t-norms having AGs,
for example, minimum, algebraic
product, Hamacher t-norms (Einstein product),
Frank t-norms, and strict t-norms, etc.

(2) It can be deduced that,
\begin{itemize}
  \item {if $T$ is taken as the algebraic product $T_{\textbf{P}}$,
  operations in Definition~\ref{def-8.1} reduce to the classical operational
  laws $\oplus$, $\otimes$, $\lambda\alpha$, and $\alpha^{\lambda}$ of
  $\mathrm{IFVs}$ in \cite[Definitions~1.2.2]{XC2012} and \cite[Definitions~3.2]{Xu2007}, i.e.,
  $\oplus =\oplus_{_{T_{\textbf{P}}}}$, $\otimes =\otimes_{_{T_{\textbf{P}}}}$,
  $\lambda\alpha =\lambda_{_{T_{\textbf{P}}}}\alpha$, and $\alpha^{\lambda}
   =\alpha^{\lambda_{_{T_{\textbf{P}}}}}$;}

  \item if $T$ is taken as a strict t-norm, Definition~\ref{def-8.1}
  is equivalent to \cite[Definitions~4]{DGM2017} and \cite[Definitions~5]{XXZ2012}.
  Although a pair of t-conorm $S$ and t-norm $T$ need not be dual in
  \cite[Definitions~4]{DGM2017}, which is not essential, one only needs to
  assume that the inequality $T(u, v) +S(1 -u, 1 -v)
  \leq 1$ holds for all $(u, v) \in [0,1]^2$,
  ensuring the operations $\oplus_{_T}$ and $\otimes_{_T}$ be closed in
  IFVs $\tilde{\mathbb{I}}$.
  \end{itemize}
  Therefore, all \cite[Definitions~1.2.2]{XC2012}, \cite[Definitions~3.2]{Xu2007},
  and \cite[Definitions~4]{DGM2017} are special cases of our Definition~\ref{def-8.1}.

(3) Clearly, $\cap=\otimes_{_{T_{\textbf{M}}}}$ and $\cup=\oplus_{_{T_{\textbf{M}}}}$.

(4) For $\alpha \in\{\langle 0, 1\rangle, \langle 1, 0\rangle, \langle 0, 0\rangle\}$
and $\lambda  \in (0, +\infty)$, it follows from Proposition~\ref{pr-7.3} that
$\lambda_{_T}\alpha=\alpha^{\lambda_{_T}}=\alpha$.
\end{remark}


\begin{theorem}
\label{Closed-Thm}
Let $T$ be a continuous t-norm. Then, for $\alpha$, $\beta \in \tilde{\mathbb{I}}$
and $\lambda>0$, all $\alpha \oplus_{_T}\beta$, $\alpha\otimes_{_T}\beta$,
$\lambda_{_T}\alpha$, and $\alpha^{\lambda_{_T}}$ are IFVs.
\end{theorem}

\begin{proof}
We only prove that $\lambda_{_T}\alpha$ is an IFV. The rest can be proved analogously.

For any $\lambda>0$, by Proposition~\ref{Power-Increasing-Pro-1} and
Theorem~\ref{Power-Dual-Thm-Main}, we have $(\mu_{\alpha})_{S}^{(\lambda)}=
1-(1-\mu_{\alpha})_{T}^{(\lambda)}\leq 1-(\nu_{\alpha})_{T}^{(\lambda)}$
since $\mu_{\alpha}+\nu_{\alpha}\leq1$, and thus
$(\mu_{\alpha})_{S}^{(\lambda)}+(\nu_{\alpha})_{T}^{(\lambda)} \leq 1,$
i.e., $\lambda_{_T}\alpha$ is an IFV.
\end{proof}

\begin{theorem}
\label{T-Norm-Operation-Thm}
Let $\gamma_1$, $\gamma_2$, $\gamma_3\in \tilde{\mathbb{I}}$, and $T$ be a
continuous t-norm. Then, for any $\zeta$, $\xi$, $\lambda>0$, we have
\begin{itemize}
 \item [{\rm (i)}] $\gamma_1\oplus_{_T}\gamma_2=\gamma_2\oplus_{_T}\gamma_1$;
 \item [{\rm (ii)}] $\gamma_1\otimes_{_T}\gamma_2=\gamma_2\otimes_{_T}\gamma_1$;
 \item [{\rm (iii)}] $(\gamma_1\oplus_{_T}\gamma_2)\oplus_{_T}\gamma_3
     =\gamma_1\oplus_{_T}(\gamma_2\oplus_{_T}\gamma_3)$;
 \item [{\rm (iv)}] $(\gamma_1\otimes_{_T}\gamma_2)\otimes_{_T}\gamma_3=
     \gamma_1\otimes_{_T}(\gamma_2\otimes_{_T}\gamma_3)$;
 \item [{\rm (v)}] $(\xi_{_{T}} \gamma_1)\oplus_{_T}(\zeta_{_{T}} \gamma_1)=(\xi+\zeta)_{_T}\gamma_1$;
 \item [{\rm (vi)}] $(\gamma_1^{\xi_{_T}})\otimes_{_T}(\gamma_1^{\zeta_{_T}})=\gamma_1^{(\xi+\zeta)_{T}}$.
 \end{itemize}
  Moreover, if $T$ satisfies Theorem \ref{Power-Stab-Char} (iii), then
 \begin{itemize}
 \item [{\rm (vii)}] $\lambda_{_T}(\gamma_1\oplus_{_T}
 \gamma_2)=\lambda_{_T}\gamma_1\oplus_{_T}\lambda_{_T}\gamma_2$;
 \item [{\rm (viii)}] $(\gamma_1\otimes_{_T}\gamma_2)^{\lambda_{_T}}=\gamma_1^{\lambda_{_T}}\otimes_{_T}\gamma_2^{\lambda_{_T}}$;
\item [{\rm (ix)}] $(\gamma_1^{\lambda_T})^{\xi_{T}}=\gamma_1^{(\lambda\cdot\xi)_T} $;
\item [{\rm (x)}] $\xi_{_{T}}(\lambda_{_T}\gamma_1)=(\lambda \cdot\xi)_{_T}\gamma_1$.
\end{itemize}
\end{theorem}

\begin{proof}
Assume that the dual t-conorm of $T$ is $S$. The statements (i)--(iv) follow directly from
the commutativity and associativity of $T$ and $S$.

(v) From Definition~\ref{def-8.1}, by direct calculation, we have
\begin{align*}
&(\xi_{_{T}} \gamma_1)\oplus_{_T}(\zeta_{_{T}} \gamma_1)\\
=&\left\langle (\mu_{\gamma_1})_{S}^{(\xi)},(\nu_{\gamma_1})_{T}^{(\xi)}\right\rangle
\oplus_{_T}\left\langle (\mu_{\gamma_1})_{S}^{(\zeta)},(\nu_{\gamma_1})_{T}^{(\zeta)}\right\rangle\\
=&\left\langle S((\mu_{\gamma_1})_{S}^{(\xi)},(\mu_{\gamma_1})_{S}^{(\zeta)}),
T((\nu_{\gamma_1})_{T}^{(\xi)},(\nu_{\gamma_1})_{T}^{(\zeta)})\right\rangle\\
=&\left\langle(\mu_{\gamma_1})_{S}^{(\xi+\zeta)},(\nu_{\gamma_1})_{T}^{(\xi+\zeta)}\right\rangle
\quad \text{(by Proposition~\ref{Power-Mon-Thm})}\\
=&(\xi +\zeta)_{_T}\gamma_1.
\end{align*}

(vi) From Definition~\ref{def-8.1}, by direct calculation, we have
\begin{align*}
&(\gamma_1^{\xi_{_T}})\otimes_{_T}(\gamma_1^{\zeta_{_T}})\\
=&\left\langle(\mu_{\gamma_1})_{_T}^{(\xi)},(\nu_{\gamma_1})_{_S}^{(\xi)}\right\rangle \otimes_{_T}\left\langle(\mu_{\gamma_1})_{_T}^{(\zeta)},(\nu_{\gamma_1})_{_S}^{(\zeta)}\right\rangle \\
=&\left\langle T((\mu_{\gamma_1})_{_T}^{(\xi)},(\mu_{\gamma_1})_{_T}^{(\zeta)}),
S((\nu_{\gamma_1})_{_S}^{(\xi)},(\nu_{\gamma_1})_{_S}^{(\zeta)})\right\rangle\\
=&\left((\mu_{\gamma_1})_{_T}^{(\xi+\zeta)},(\nu_{\gamma_1})_{_S}^{(\xi+\zeta)}\right)
\quad \text{(by Proposition~\ref{Power-Mon-Thm})}\\
=&\gamma_1^{(\xi+\zeta)_{_T}}.
\end{align*}

(vii) From Definition~\ref{def-8.1}, by direct calculation, we have
\begin{equation}
\begin{split}
\lambda_{_T}(\gamma_1\oplus_{_T}\gamma_2)&=\lambda_{_T}
\left\langle S(\mu_{\gamma_1}, \mu_{\gamma_2}), T(\nu_{\gamma_1},\nu_{\gamma_2})\right\rangle\\
&=\left\langle (S(\mu_{\gamma_1},\mu_{\gamma_2}))_{S}^{(\lambda)},
(T(\nu_{\gamma_1}, \nu_{\gamma_2}))_{T}^{(\lambda)})\right\rangle,
\end{split}
\end{equation}
and
\begin{equation}
(\lambda_{_T}\gamma_1)\oplus_{_T}(\lambda_{_T}\gamma_2)\\
=\left\langle S((\mu_{\gamma_1})_{_S}^{(\lambda)},(\mu_{\gamma_2})_{_S}^{(\lambda)}),
T((\nu_{\gamma_1})_{_T}^{(\lambda)},(\nu_{\gamma_1})_{_T}^{(\lambda)})\right\rangle.
\end{equation}
Since $T$ satisfies Theorem~\ref{Power-Stab-Char} (ii),
by Theorems~\ref{Power-Dual-Thm-Main} and \ref{Power-Stab-Char},
we obtain
$$
(T(\nu_{\gamma_1}, \nu_{\gamma_2}))_{_T}^{(\lambda)} =
T((\nu_{\gamma_1})_{_T}^{(\lambda)},(\nu_{\gamma_2})_{_T}^{(\lambda)}),
$$
and
\begin{align*}
(S(\mu_{\gamma_1}, \mu_{\gamma_2}))_{S}^{(\lambda)}
& =1 -(1 -S(\mu_{\gamma_1}, \mu_{\gamma_2}))_{T}^{(\lambda)}\\
& =1 -(T(1 -\mu_{\gamma_1}, 1 -\mu_{\gamma_2}))_{T}^{(\lambda)}\\
& =1 -T((1 -\mu_{\gamma_1})_{T}^{(\lambda)}, (1 -\mu_{\gamma_2})_{T}^{(\lambda)})\\
& =1 -T(1 -(\mu_{\gamma_1})_{S}^{(\lambda)}, 1 -(\mu_{\gamma_2})_{S}^{(\lambda)})\\
& =S((\mu_{\gamma_1})_{S}^{(\lambda)}, (\mu_{\gamma_2})_{S}^{(\lambda)}).
\end{align*}
Therefore,
$\lambda_{_T}(\gamma_1\oplus_{_T} \gamma_2)=(\lambda_{_T}\gamma_1)
\oplus_{_T}(\lambda_{_T}\gamma_2).$

(viii) Similarly to the proof of (vii), it can be verified that $(\gamma_1\otimes_{_T}\gamma_2)^{\lambda_{_T}}=\gamma_1^{\lambda_{_T}}\otimes_{_T}\gamma_2^{\lambda_{_T}}$.

(ix) and (x) From Definition~\ref{def-8.1}, since $T$ satisfies
Theorem~\ref{Power-Stab-Char} (ii), by Theorems~\ref{Power-Dual-Thm-Main}
and \ref{Power-Stab-Char}, we have
\begin{align*}
(\gamma_1^{\lambda_{_T}})^{\xi_{_T}}
&=((\mu_{\gamma_1})_{_T}^{(\lambda)}, (\nu_{\lambda})_{_S}^{(\lambda)})^{\xi_{_T}}\\
&=(((\mu_{\gamma_1})_{_T}^{(\lambda)})_{_T}^{(\xi)},((\nu_{\gamma_1})_{_S}^{(\lambda)})_{_S}^{(\xi)})\\
&=((\mu_{\gamma_1})_{_T}^{(\lambda\cdot\xi)}, (1-(1-\nu_{\gamma_1})_{_T}^{(\lambda)})_{_S}^{(\xi)})\\
&=((\mu_{\gamma_1})_{_T}^{(\lambda\cdot\xi)},1- ((1-\nu_{\gamma_1})_{_T}^{(\lambda)})_{_T}^{(\xi)})\\
&=((\mu_{\gamma_1})_{_T}^{(\lambda\cdot\xi)},1-(1-\nu_{\gamma_1})_{_T}^{(\lambda\cdot\xi)})\\
&=((\mu_{\gamma_1})_{_T}^{(\lambda\cdot\xi)}, (\nu_{\gamma_1})_{_S}^{(\lambda\cdot\xi)})\\
&=\gamma_1^{(\lambda\cdot\xi)_{_T}},
\end{align*}
and
\begin{align*}
\xi_{_T}(\lambda_{_T}\gamma_1)
&=\xi_{_T}((\mu_{\gamma_1})_{_S}^{(\lambda)},(\nu_{\lambda})_{_T}^{(\lambda)})\\
&=(((\mu_{\gamma_1})_{_S}^{(\lambda)})_{_S}^{(\xi)},((\nu_{\gamma_1})_{_T}^{(\lambda)})_{_T}^{(\xi)})\\
&=((1-(1-\mu_{\gamma_1})_{_T}^{(\lambda)})_{_S}^{(\xi)},(\nu_{\gamma_1})_{_T}^{(\lambda\cdot\xi)})\\
&=(1-((1-\mu_{\gamma_1})_{_T}^{(\lambda)})_{_T}^{(\xi)} ,(\nu_{\gamma_1})_{_T}^{(\lambda\cdot\xi)})\\
&=(1-(1-\mu_{\gamma_1})_{_T}^{(\lambda\cdot\xi)} ,(\nu_{\gamma_1})_{_T}^{(\lambda\cdot\xi)})\\
&=((\mu_{\gamma_1})_{_S}^{(\lambda\cdot\xi)}, (\nu_{\gamma_1})_{_T}^{(\lambda\cdot\xi)})\\
&=(\xi\cdot\lambda)_{_T}\gamma_1.
\end{align*}
\end{proof}

{\begin{remark}
\label{re-10}
  (1)  If  $T=T_{\textbf{P}}$ in Theorems~\ref{Closed-Thm} and \ref{T-Norm-Operation-Thm},
  then Theorem~\ref{Closed-Thm} and Theorem~\ref{T-Norm-Operation-Thm} (i)--(iii) and (v)--(ix) reduce to
  \cite[Theorem~1.2.2]{XC2012} and \cite[Theorem~1.2.3]{XC2012}
  (also see \cite[Theorems~3.2 and 3.3]{Xu2007}), respectively.

  (2) If $T$ in Theorems~\ref{Closed-Thm} and \ref{T-Norm-Operation-Thm}
   is a strict t-norm, then Theorem~\ref{Closed-Thm} and Theorem~\ref{T-Norm-Operation-Thm} (i),
   (ii), and (v)--(viii) reduce to \cite[Theorem 1]{XXZ2012}. It should be noted that,
   by Theorem~\ref{Power-Stab-Char}, \cite[Theorem~1 (3) and (4)]{XXZ2012} hold
   if and only if $T$ is strict, under the assumption of $T\in \mathscr{T}_{_{\mathrm{ConA}}}$ in \cite{XXZ2012}.

  (3) If $T$ in Theorem \ref{T-Norm-Operation-Thm} is the
  Acz\'{e}l-Alsina t-norms, i.e., $T(u, v)=T_{\lambda}^{\textbf{AA}}(u, v)
      = e^{-\left[(-\ln u)^{\lambda}+(-\ln v)^{\lambda}\right]^{\frac{1}{\lambda}}}$
      ($\lambda\in (0, +\infty)$), then Theorem~\ref{T-Norm-Operation-Thm} (i),
   (ii), and (v)--(viii) reduce to \cite[Theorem~1]{SCY2022}.

 (4) By Theorem~\ref{T-Norm-Operation-Thm} (v) and (vi),
 for any $\alpha\in\tilde{\mathbb{I}}$ and any $n\in \mathbb{N}$, it holds that
 $n_{_T}\cdot \alpha=\underbrace{\alpha\oplus_{_T}\cdots\oplus_{_T}\alpha}_{n}$
and
 $\alpha^{n_{_T}}=\underbrace{\alpha\otimes_{_T}\cdots\otimes_{_T}\alpha}_{n}.$

  (5) By Theorem~\ref{Power-Stab-Char}, one of   Theorem~\ref{T-Norm-Operation-Thm}
  (vii)--(x) holds if and only if $T$ satisfies Theorem~\ref{Power-Stab-Char} (iii),
  i.e., $T=T_{\textbf{M}}$, or $T$ is strict, or $T$ is an ordinal sum of strict t-norms.
\end{remark}
}

\section{IF aggregation operators}\label{Sec-V}

\subsection{IF weighted average (geometric) operator}

\begin{definition}\label{de-9.1}
Let $\bm{\Omega}=(\omega_1,\omega_2,\ldots,\omega_n)^{\top}$ be
the weight vector such that $\omega_\ell\in (0, 1]$ and
$\sum_{\ell=1}^{n}\omega_{\ell}=1$ and $T$ be a continuous t-norm.
Define the \textit{IF weighted average operator}
$\mathrm{IFWA}_{T, \bm{\Omega}}$ and \textit{IF weighted geometric operator}
$\mathrm{IFWG}_{T, \bm{\Omega}}$ induced by $T$ as
\begin{equation}\label{eq-9.1}
\begin{split}
\mathrm{IFWA}_{T, \bm{\Omega}}:\tilde{\mathbb{I}}^n & \rightarrow \tilde{\mathbb{I}}\\
(\gamma_1,\ldots,\gamma_n) & \mapsto ((\omega_1)_{_{T}}\gamma_1) \oplus_{_{T}}
\cdots  \oplus_{_{T}}((\omega_n)_{_{T}}\gamma_n),
\end{split}
\end{equation}
and
\begin{equation}\label{eq-9.2}
\begin{split}
\mathrm{IFWG}_{T, \bm{\Omega}}:\tilde{\mathbb{I}}^n & \rightarrow \tilde{\mathbb{I}}\\
(\gamma_1,\ldots,\gamma_n) & \mapsto (\gamma_1^{(\omega_1)_{_{T}}}) \otimes_{_{T}}
\cdots  \otimes_{_{T}}(\gamma_n^{(\omega_n)_{_{T}}}),
\end{split}
\end{equation}
respectively.
\end{definition}

\begin{remark}
\label{Remark-7}
{(i) By Theorem \ref{Closed-Thm}, it follows that $(\omega_1)_{_{T}}\gamma_1\oplus_{_{T}}
    \cdots \oplus_{_{T}}(\omega_n)_{_{T}}\gamma_n\in\tilde{\mathbb{I}}$
and $\gamma_1^{(\omega_1)_{_{T}}}\otimes_{_{T}}
\cdots \otimes_{_{T}}\gamma_n^{(\omega_n)_{_{T}}}\in\tilde{\mathbb{I}}$,
provided that $T$ is continuous and $\gamma_i\in\tilde{\mathbb{I}}$.
Thus, Definition~\ref{de-9.1} is well defined.}

(ii) By Theorem~\ref{T-Norm-Operation-Thm} (vii) and (viii), if $T$
satisfies Theorem~\ref{Power-Stab-Char} (iii) and $\bm{\Omega}=(\frac{1}{n}, \frac{1}{n},
\ldots, \frac{1}{n})^{\top}$,
then
\begin{equation}
\begin{split}
&\quad \mathrm{IFWA}_{T, \bm{\Omega}}(\gamma_1,\gamma_2,\ldots,\gamma_n)\\
&=\left(\frac{1}{n}\right)_{_{T}}(\gamma_1 \oplus_{_{T}}\gamma_2
 \oplus_{_{T}} \cdots  \oplus_{_{T}}\gamma_n),
\end{split}
\end{equation}
and
\begin{equation}
\begin{split}
&\quad \mathrm{IFWG}_{T, \bm{\Omega}}(\gamma_1, \gamma_2, \ldots, \gamma_n)\\
&=(\gamma_1  \otimes_{_{T}} \gamma_2  \otimes_{_{T}}
\cdots  \otimes_{_{T}}\gamma_n)^{(\frac{1}{n})_{_{T}}}.
\end{split}
\end{equation}
\end{remark}

\begin{theorem}
\label{IFWA-IFWG-Formula}
Let $T$ be a continuous t-norm and $\bm{\Omega}=(\omega_1,\omega_2,\ldots,\omega_n)^{\top}$
be the weight vector with $\omega_j\in (0, 1]$ and $\sum_{j=1}^{n}\omega_{j}=1$.
Then, for $\{\gamma_j=\langle \mu_{\gamma_j}, \nu_{\gamma_j}\rangle\}_{j=1}^{n}
\subseteq \tilde{\mathbb{I}}$,
\begin{equation}
\begin{split}
&\quad \mathrm{IFWA}_{T, \bm{\Omega}}(\gamma_1, \gamma_2, \ldots, \gamma_n)\\
& =\left\langle S^{(n)}\left((\mu_{\gamma_1})_S^{(\omega_1)},(\mu_{\gamma_2})_S^{(\omega_2)},
\ldots, (\mu_{\gamma_n})_S^{(\omega_n)}\right),\right.\\
&~~~~ \left.T^{(n)}\left((\nu_{\gamma_1})_T^{(\omega_1)},(\nu_{\gamma_2})_T^{(\omega_2)},
\ldots, (\nu_{\gamma_n})_T^{(\omega_n)}\right)\right\rangle,
\end{split}
\end{equation}
and
\begin{equation}
\begin{split}
&\quad \mathrm{IFWG}_{T, \bm{\Omega}}(\gamma_1, \gamma_2, \ldots,\gamma_n)\\
& =\left\langle T^{(n)}\left((\mu_{\gamma_1})_T^{(\omega_1)},(\mu_{\gamma_2})_T^{(\omega_2)},
\ldots, (\mu_{\gamma_n})_T^{(\omega_n)}\right),\right.\\
&~~~~ \left. S^{(n)}\left((\nu_{\gamma_1})_S^{(\omega_1)},(\nu_{\gamma_2})_S^{(\omega_2)},
\ldots, (\nu_{\gamma_n})_S^{(\omega_n)}\right)\right\rangle,
\end{split}
\end{equation}
where $S$ is the dual t-conorm of $T$.
\end{theorem}

\begin{proof}
By formulas~\eqref{eq-9.1} and \eqref{eq-9.2},
and Definition~\ref{def-8.1}, we have
{\begin{align*}
&\mathrm{IFWA}_{T, \bm{\Omega}}(\gamma_1,\gamma_2, \ldots,\gamma_n)\\
 =&((\mu_{\gamma_1})_S^{(\omega_1)},(\nu_{\gamma_1})_T^{(\omega_1)})
\oplus_{_{T}}((\mu_{\gamma_2})_S^{(\omega_2)},(\nu_{\gamma_2})_T^{(\omega_2)})\\
&\quad \oplus_{_T}
\cdots \oplus_{_{T}}((\mu_{\gamma_n})_S^{(\omega_n)},(\nu_{\gamma_n})_{_{T}}^{(\omega_n)})\\
 =&\left\langle S^{(n)}\left((\mu_{\gamma_1})_S^{(\omega_1)}, (\mu_{\gamma_2})_S^{(\omega_2)},
\ldots, (\mu_{\gamma_n})_S^{(\omega_n)}\right),\right.\\
&\quad \left.
T^{(n)}\left((\nu_{\gamma_1})_T^{(\omega_1)}, (\nu_{\gamma_2})_T^{(\omega_2)},
\ldots, (\nu_{\gamma_n})_T^{(\omega_n)}\right)\right\rangle,
\end{align*}
and
\begin{align*}
&\mathrm{IFWG}_{T, \bm{\Omega}}(\gamma_1, \gamma_2, \ldots,\gamma_n)\\
 =& ((\mu_{\gamma_1})_S^{(\omega_1)}, (\nu_{\gamma_1})_T^{(\omega_1)})
\otimes_{_{T}}((\mu_{\gamma_2})_S^{(\omega_2)},
(\nu_{\gamma_2})_T^{(\omega_2)})\\
&\quad \otimes_{_{T}}\cdots \otimes_{_{T}}((\mu_{\gamma_n})_S^{(\omega_n)},
(\nu_{\gamma_n})_{_{T}}^{(\omega_n)})\\
 =&\left\langle T^{(n)}\left((\mu_{\gamma_1})_T^{(\omega_1)}, (\mu_{\gamma_2})_T^{(\omega_2)},
\ldots, (\mu_{\gamma_n})_T^{(\omega_n)}\right),\right.\\
&\quad \left.
S^{(n)}\left((\nu_{\gamma_1})_S^{(\omega_1)}, (\nu_{\gamma_2})_S^{(\omega_2)},
\ldots, (\nu_{\gamma_n})_S^{(\omega_n)}\right)\right\rangle.
\end{align*}}
\end{proof}

In the following, consider $\mathrm{IFWA}_{T, \bm{\Omega}}$
and $\mathrm{IFWG}_{T, \bm{\Omega}}$ for some special forms of $T$:

\medskip

 (1) If $T=T_{\textbf{P}}$, by Theorem~\ref{IFWA-IFWG-Formula} and Remark~\ref{re-4},
 then $\mathrm{IFWA}_{T_{\textbf{P}}, \bm{\Omega}}=\mathrm{IFWA}_{\bm{\Omega}}$ and
      $\mathrm{IFWG}_{T_{\textbf{P}}, \bm{\Omega}}=\mathrm{IFWG}_{\bm{\Omega}}$,
      where $\mathrm{IFWA}_{\bm{\Omega}}$ and $\mathrm{IFWG}_{\bm{\Omega}}$ are defined
      in~\cite[Definition~1.3.1]{XC2012} and \cite[Definition~1.3.2]{XC2012}
      (also see~\cite[Definitions~3.3 and 3.4]{Xu2007}, respectively. Thus,
\cite[Theorems~1.3.1 and 1.3.2]{XC2012} and \cite[Theorems~3.4 and 3.6]{Xu2007}
are direct corollaries of Theorem~\ref{IFWA-IFWG-Formula}.

\medskip

  (2) If $T =T_{2}^{\textbf{H}}$ (Einstein product), by Theorem~\ref{IFWA-IFWG-Formula}
  and Example~\ref{T-Norm-Operation-Thm}, then $\mathrm{IFWA}_{T_{2}^{\textbf{H}}, \bm{\Omega}}
  = \mathrm{IFWA}_{\bm{\Omega}}^{\varepsilon}$, where $\mathrm{IFWA}_{\bm{\Omega}}^{\varepsilon}$ is
  defined in~\cite[Definition 14.1]{WL2012}. Thus, \cite[Theorems 4.1]{WL2012} is a
  direct corollary of Theorem~\ref{IFWA-IFWG-Formula}.

\medskip

  (3) If $T$ is a strict t-norm with an AG $G$,
  by Proposition~\ref{Power-t-Thm} and Theorem~\ref{IFWA-IFWG-Formula}, then
    $\mathrm{IFWA}_{T, \bm{\Omega}}(\gamma_{1}, \ldots, \gamma_{n})
  =\Big\langle 1 -G^{ -1}\Big(\sum_{\ell=1}^{n}\omega_{\ell}G(1 -\mu_{\gamma_{\ell}})\Big),
  G^{ -1}\Big(\sum_{\ell=1}^{n}\omega_{\ell}G(\nu_{\gamma_{\ell}})\Big)\Big\rangle,$
  and
  $\mathrm{IFWG}_{T, \bm{\Omega}}(\gamma_{1}, \ldots, \gamma_{n})
  =\Big\langle G^{ -1}\Big(\sum_{\ell=1}^{n}\omega_{\ell}G(\mu_{\gamma_{\ell}})\Big),$
  $G^{ -1}\Big(1 -\sum_{\ell=1}^{n}\omega_{\ell}G(1 -\nu_{\gamma_{\ell}})\Big)\Big\rangle$,
  which are \cite[Theorems~3 and 4]{XXZ2012}.

\begin{example}
(1) The family $(T_{\lambda}^{\textbf{H}})_{\lambda\in[0, +\infty]}$
of \textit{Hamacher t-norms} is
$$
T_{\lambda}^{\textbf{H}}(u, v)
=\begin{cases}
T_{\textbf{D}}(u, v), & \lambda=+\infty, \\
0,  & \lambda=u=v=0, \\
\frac{uv}{\lambda +(1 -\lambda)(u +v -uv)}, & \text{otherwise}.
\end{cases}
$$

(2) The family $(S_{\lambda}^{\textbf{H}})_{\lambda\in[0, +\infty]}$
of \textit{Hamacher t-conorms} is
$$
S_{\lambda}^{\textbf{H}}(u, v)
=\begin{cases}
S_{\textbf{D}}(u, v),  & \lambda=+\infty, \\
1,  & \lambda=0 \text{ and } u=v=1, \\
\frac{u +v -uv -(1 -\lambda)
uv}{1 -(1 -\lambda)(u +v -uv)}, & \text{otherwise}.
\end{cases}
$$
The t-norm $T_{0}^{\textbf{H}}$ and  the t-conorm $S_{2}^{\textbf{H}}$,
which are given by
$$
T_{0}^{\textbf{H}}(u, v)
=\begin{cases}
0, & u=v=0, \\
\frac{uv}{u +v -uv}, & \text{otherwise},
\end{cases}
$$
and
$$
T_{2}^{\textbf{H}}(u, v)
=\frac{uv}{1 +(1 -u)(1 -v)},
$$
respectively, are sometimes called the
\textit{Hamacher product} and the
\textit{Einstein product}, respectively.

{Clearly, $(T_{\lambda}^{\textbf{H}})_{\lambda\in[0, +\infty]}$ are strict.
For each $\lambda\in [0, +\infty]$, $T_{\lambda}^{\textbf{H}}$ and
$S_{\lambda}^{\textbf{H}}$ are dual to each other. The AGs $G_{\lambda}
^{\textbf{H}}$, $S_{\lambda}^{\textbf{H}}: [0, 1]\rightarrow [0, +\infty]$
of $T_{\lambda}^{\textbf{H}}$ and $S_{\lambda}^{\textbf{H}}$ are given by, respectively,
$$
G_{\lambda}^{\textbf{H}}(u)
=\begin{cases}
\frac{1 -u}{u}, & \lambda=0, \\
\ln\left(\frac{\lambda +(1 -\lambda) u}{u}\right), & \lambda \in (0, +\infty),
\end{cases}
$$
and
$$
S_{\lambda}^{\textbf{H}}(u)
=\begin{cases}
\frac{u}{1 -u}, & \lambda=0, \\
\ln\left(\frac{\lambda +(1 -\lambda)(1 -u)}{1 -u}\right), & \lambda \in (0, +\infty).
\end{cases}
$$}

Let $\gamma_{1} =\langle 0.1, 0.7\rangle$, $\gamma_{2} =\langle 0.4, 0.3\rangle$,
$\gamma_{3} =\langle 0.6, 0.1\rangle$, and $\gamma_{4} =\langle 0.2, 0.5\rangle
 \in \tilde{\mathbb{I}}$,
and $\bm{\Omega} =(0.2, 0.3, 0.1, 0.4)^{\top}$ be the weighted vector
of $\gamma_{1}$--$\gamma_4$. According to~\cite[Example 4.1]{WL2012}, we have
$\mathrm{IFWA}_{T_{\textbf{P}}, \bm{\Omega}}
(\gamma_{1}, \gamma_{2}, \gamma_{3}, \gamma_{4}) =\langle 0.2990, 0.3906\rangle$
and
$\mathrm{IFWA}_{T_{2}^{\textbf{H}}, \bm{\Omega}}
(\gamma_{1}, \gamma_{2}, \gamma_{3}, \gamma_{4}) =\langle 0.2891, 0.4026\rangle.$

Case 1. Consider the t-norm $T$ defined by
\begin{equation}
\label{eq-Exm-4-1}
\quad T(u, v)=
\begin{cases}
\frac{1}{2}T_{\textbf{P}}(2u, 2v), & {(u, v)\in[0, \frac{1}{2}]^{2}}, \\
\frac{1}{2} +\frac{1}{2}T_{2}^{\textbf{H}}(2u -1, 2v -1), & {(u, v)
\in [\frac{1}{2}, 1]^{2}}, \\
\min\{u, v\}, & \text{otherwise},
\end{cases}
\end{equation}
i.e., $T=(\langle 0, 0.5, T_{\textbf{P}}\rangle,
\langle 0.5, 1, T_{2}^{\textbf{H}}\rangle)$.
By direct calculation and Theorem~\ref{Power-Formula-Xinxing}, we have
\begin{equation}
\label{eq-Exm-4-1-a}
u_{T}^{(t)}=
\begin{cases}
2^{t -1} \cdot u^{t}, & u\in[0, \frac{1}{2}] \text{ and } t >0,\\
\frac{1}{2} +\frac{1}{\left(\frac{3 -2u}{2u -1}\right)^{t} +1}, & u \in [\frac{1}{2}, 1] \text{ and } t >0.
\end{cases}
\end{equation}
Then, by Theorem~\ref{Power-Dual-Thm-Main}, we have
\begin{align*}
 \mathrm{IFWA}_{T, \bm{\Omega}}(\gamma_{1}, \gamma_{2}, \gamma_{3}, \gamma_{4})
 =& \left\langle 1 -\frac{0.4^{0.1}}{2^{0.9}}, \frac{0.1^{0.1} \cdot 0.3^{0.3}}{2^{0.6}}\right\rangle\\
 =& \langle0.5110, 0.3652\rangle.
 \end{align*}

Case 2. Consider the t-norm $\hat{T}$ defined by
\begin{equation}
\label{eq-Exm-4-2}
\hat{T}(u, v)=
\begin{cases}
\frac{1}{2}T_{2}^{\textbf{H}}(2u, 2v), & (u, v) \in
[0, \frac{1}{2}]^{2}, \\
\frac{1}{2} +\frac{1}{2}T_{\textbf{P}}(2u -1, 2v -1), & (u, v)
\in [\frac{1}{2}, 1]^{2}, \\
\min\{u, v\}, & \text{otherwise},
\end{cases}
\end{equation}
i.e., $\hat{T} =(\langle 0, 0.5, T_{2}^{\textbf{H}}\rangle,
\langle 0.5, 1, T_{\textbf{P}}\rangle)$.
By direct calculation and Theorem~\ref{Power-Formula-Xinxing},
we have
\begin{equation}
\label{eq-Exm-4-2-a}
u_{\hat{T}}^{(t)}=
\begin{cases}
\frac{1}{(\frac{1 -u}{u})^{t} +1}, & u\in [0, \frac{1}{2}] \text{ and } t >0, \\
\frac{1}{2} +\frac{1}{2}(2u -1)^{t}, & u\in [\frac{1}{2}, 1] \text{ and } t >0.
\end{cases}
\end{equation}
Then, by Theorem~\ref{Power-Dual-Thm-Main},  we have
 \begin{align*}
 &\mathrm{IFWA}_{\hat{T}, \bm{\Omega}}(\gamma_{1}, \gamma_{2}, \gamma_{3}, \gamma_{4})\\
 =&\left\langle 1 -\frac{1}{1.5^{0.1} +1},\frac{2 \cdot\frac{1}{(\frac{7}{3})^{1.3} +1}
  \cdot\frac{1}{9^{0.1} +1}}{1 +\left(1 -\frac{2}{(\frac{7}{3})^{0.3} +1}\right)
 \left(1 -\frac{2}{1 -9^{0.1} +1}\right)}\right\rangle\\
 =& \langle 0.5101, 0.3837\rangle.
 \end{align*}
\end{example}

Combining together Proposition~\ref{Power-Increasing-Pro-1},
Theorem~\ref{T-Norm-Operation-Thm} (v), and Theorem~\ref{IFWA-IFWG-Formula},
we immediately obtains the following result.

\begin{proposition}
\label{IFWA-IFWG-Pro}
Let $T$ be a continuous t-norm and $\bm{\Omega} =(\omega_1,\omega_2,\ldots,\omega_n)^{\top}$
be the weight vector with $\omega_j\in (0, 1]$ and $\sum_{j=1}^{n}\omega_{j}=1$.
Then, for $\{\gamma_{j}=\langle \mu_{\gamma_{j}}, \nu_{\gamma_{j}}\rangle\}_{j=1}^{n}
\subseteq \tilde{\mathbb{I}}$, we have
\begin{enumerate}[{\rm (1)}]
  \item {\rm (Idempotency):} If $\gamma_{1} =\gamma_{2} =\cdots =\gamma_{n} =\gamma$,
 then
$$
\mathrm{IFWA}_{T, \bm{\Omega}}(\gamma_{1}, \ldots, \gamma_{n})
 =\mathrm{IFWG}_{T, \bm{\Omega}}(\gamma_{1}, \ldots, \gamma_{n})
 =\gamma.
$$
  \item {\rm (Boundedness):}
$$
\gamma^{ -} \leq
\mathrm{IFWA}_{T, \bm{\Omega}}(\gamma_{1}, \ldots, \gamma_{n})
\leq \gamma^{ +},
$$
 and
$$
\gamma^{ -} \leq
\mathrm{IFWG}_{T, \bm{\Omega}}(\gamma_{1}, \ldots, \gamma_{n})
\leq \gamma^{ +},
$$
where $\gamma^{ -} =\left\langle \min_{1\leq j\leq n}\{\mu_{\gamma_{j}}\},
\max_{1\leq j\leq n}\{\nu_{\gamma_{j}}\}\right\rangle$ and
$\gamma^{ +} =\left\langle \max_{1\leq j\leq n}\{\mu_{\gamma_{j}}\},
\min_{1\leq j\leq n}\{\nu_{\gamma_{j}}\}\right\rangle$.
  \item {\rm (Monotonicity):} Let $\{\beta_{j} =\langle \mu_{\beta_{j}}, \nu_{\beta_{j}}\rangle\}_{j=1}^{n}
\subseteq \tilde{\mathbb{I}}$ such that $\mu_{\beta_{j}}\geq\mu_{\gamma_{j}}$ and
$\nu_{\beta_{j}}\leq \nu_{\gamma_{j}}$.
 Then
$$
\mathrm{IFWA}_{T, \bm{\Omega}}(\gamma_{1}, \ldots, \gamma_{n})
\leq_{_{\mathrm{Xu}}} \mathrm{IFWA}_{T, \bm{\Omega}}(\beta_{1}, \ldots, \beta_{n}),
$$
and
$$
\mathrm{IFWG}_{T, \bm{\Omega}}(\gamma_{1}, \ldots, \gamma_{n})
\leq_{_{\mathrm{Xu}}} \mathrm{IFWG}_{T, \bm{\Omega}}(\beta_{1}, \ldots, \beta_{n}).
$$
\end{enumerate}
\end{proposition}

\begin{remark}
If $T$ is taken as a strict t-norm, then Proposition~\ref{IFWA-IFWG-Pro} reduces
to \cite[Properties~1--3]{XXZ2012}. Meanwhile, \cite[Proposition~4.1]{WL2012},
which contains a long proof in~\cite{WL2012}, is a direct corollary of
Proposition~\ref{IFWA-IFWG-Pro} if $T=T_{2}^{\textbf{H}}$.
\end{remark}

{\section{Applications}
\label{Sec-VI}
\subsection{IF mean weighted average
and geometric (IFMWAG) operator}
In~\cite{HHC2015,HH2016}, it was pointed out that the operational laws in
Definition~\ref{Def-Int-Operations} have the following disadvantage:
if the rating of an alternative on some attribute is $\langle 1, 0\rangle$,
regardless of the ratings of this alternative on other attribute,
the overall rating of this alternative always is $\langle 1, 0\rangle$
(see \cite[Example~1]{HHC2015}). Clearly, this is impractical. To overcome
this disadvantage, some new interactional {operations for IFVs} were introduced
(see \cite[Definition~6]{HHC2015} and \cite[Definition~5]{HCZLT2014}).
However, those operational laws have another disadvantage: if the rating of an
alternative on some attribute is $\langle 0, 1\rangle$, regardless of the ratings
of this alternative on other attribute, the overall rating of this alternative
always is $\langle 0, 1\rangle$ (see Example~\ref{Exm-5-Wu}). This is also impractical.
To overcome these two disadvantages, we adopt the following aggregation operator for
MADM problems under IF environment.
\begin{definition}
\label{Def-IFMWAG-Ope}
For a continuous t-norm $T$ and weight vector $\bm{\Omega}=(\omega_1,\omega_2,\ldots,\omega_n)^{\top}$
with $\omega_\ell\in (0, 1]$ and $\sum_{\ell=1}^{n}\omega_{\ell}=1$, define the
\textit{intuitionistic fuzzy mean weighted average and geometric operator}
$\mathrm{IFMWAG}_{T, \bm{\Omega}}$ induced by $T$ as
\begin{equation}\label{eq-10.1}
\mathrm{IFMWAG}_{T, \bm{\Omega}}
(\gamma_1, \ldots, \gamma_n) =
\left\langle \frac{\mu_{_A} +\mu_{_G}}{2}, \frac{\nu_{_A} +\nu_{_G}}{2} \right\rangle,
\end{equation}
where $\langle \mu_{_A}, \nu_{_A}\rangle =\mathrm{IFWA}_{T, \bm{\Omega}}(\gamma_1, \ldots, \gamma_n)$
and $\langle \mu_{_G}, \nu_{_G}\rangle =\mathrm{IFWG}_{T, \bm{\Omega}}(\gamma_1, \ldots, \gamma_n)$.
\end{definition}

\begin{remark}
(1) Because $\langle \mu_{_A}, \nu_{_A}\rangle$ and $\langle \mu_{_G}, \nu_{_G}\rangle$
are IFVs (by Remark~\ref{Remark-7}), $\mathrm{IFMWAG}_{T, \bm{\Omega}}
(\gamma_1, \ldots, \gamma_n)$ is an IFV. Thus, the aggregation operator
$\mathrm{IFMWAG}_{T, \bm{\Omega}}$ in Definition~\ref{Def-IFMWAG-Ope} is closed.

(2) By Definition~\ref{de-9.1}, it can be verified that (a)
$\frac{\mu_{_A} +\mu_{_G}}{2} =0$ if and only if $\mu_{\gamma_1} =
\cdots  =\mu_{\gamma_n} =0$; (b) $\frac{\mu_{_A} +\mu_{_G}}{2} =1$
if and only if $\mu_{\gamma_1} =\cdots  =\mu_{\gamma_n} =1$.
This overcomes the two weaknesses mentioned above.

(3) By Proposition~\ref{IFWA-IFWG-Pro}, it can be obtained that
the operator $\mathrm{IFMWAG}_{T, \bm{\Omega}}$ is idempotent,
bounded, and monotonous.

(4) If $T=T_{\textbf{P}}$, by direct calculation, we have
\begin{align*}
&\mathrm{IFMWAG}_{T_{\textbf{P}}, \bm{\Omega}}
(\gamma_1, \ldots, \gamma_n)\\
=&
\left\langle \frac{1 -\prod_{\ell=1}^{n}(1 -\mu_{\gamma_\ell})^{\omega_\ell}
 +\prod_{\ell=1}^{n}(\mu_{\gamma_\ell})^{\omega_\ell}}{2},\right.\\
& ~~\left. \frac{1 -\prod_{\ell=1}^{n}(1 -\nu_{\gamma_\ell})^{\omega_\ell}
 +\prod_{\ell=1}^{n}(\nu_{\gamma_\ell})^{\omega_\ell}}{2} \right\rangle.
\end{align*}
\end{remark}

{\begin{example}
\label{Exm-5-Wu}
Let $\gamma_1 =\langle \mu_{\gamma_1}, \nu_{\gamma_1}\rangle$, $\gamma_2 =\langle 0, 1\rangle$,
and $\bm{\Omega} =(\omega_1, \omega_2)^{\top}$ be the weight vector with $0<\omega_1$, $\omega_2
\leq 1$ and $\omega_1+\omega_2=1$.

Regardless of the values of $\gamma_1$ and $\bm{\Omega}$, (1) by applying the IF aggregation operators
in \cite[Definition 3.3]{Xu2007}, \cite[Eq.~(7)]{Xu2011}, \cite[Definition 1.3.1]{XC2012}, \cite[Definition~6]{XXZ2012},
\cite[Definition~4.1]{WL2012}, \cite[Definition~10]{Huang2014}, \cite[Definition~7]{Liu2014},
\cite[Definition~4.1]{CC2016},
\cite[Definitions~6 and 7]{ZX2017}, \cite[Definition~3.2]{Garg2017},
\cite[Definition~15]{SCY2022}, \cite[Definitions~4.1 and 4.14]{FR2023},
and \cite[Definition~16]{SSSDRT2023} for
$\gamma_1^{\complement}$ and $\gamma_2^{\complement}$, we obtain
$\langle 1, 0\rangle$; (2) by applying the IF aggregation operators in
\cite[Definition~5]{XY2006}, \cite[Definition~5]{XY2011},
\cite[Definition~8]{HCZLT2014}, \cite[Definition~8]{HHC2015},
\cite[Definition~3.2]{Garg2016}, \cite[Definition~10]{LC2017},
\cite[Definitions~8 and 9]{ZX2017}, \cite[Definition~3]{Ye2017}, and \cite[Definitions~12--14]{SCMY2023} for
$\gamma_1$ and $\gamma_2$, we obtain $\langle 0, 1 \rangle.$ Clearly, all of these are unreasonable.

Taking $T$ as the t-norm defined by Eq.~\eqref{eq-Exm-4-1}, by
direct calculation, it follows from Eq.~\eqref{eq-10.1} that
\begin{equation*}
\mathrm{IFMWAG}_{T, \bm{\Omega}}(\gamma_1, \gamma_2) =
\left\langle \frac{1 -(1 -\mu_{\gamma_1})_{T}^{(\omega_1)}}{2},
\frac{1 +(\nu_{\gamma_1})_{T}^{(\omega_1)}}{2} \right\rangle,
\end{equation*}
where $(\_)_{T}^{(\omega_1)}$ is given by Eq.~\eqref{eq-Exm-4-1-a}.
It can be verified that $\mathrm{IFMWAG}_{T, \bm{\Omega}}(\gamma_1, \gamma_2)
=\langle 0, 1\rangle$ if and only if $\gamma_1=\langle 0, 1\rangle$, demonstrating
that our proposed operator $\mathrm{IFMWAG}_{T, \bm{\Omega}}$ is more reasonable.
\end{example}}

\subsection{A new MADM method via the proposed IFMWAG operator}

The MADM is useful to identify the best alternative among the available resources. In this subsection,
we establish a decision-making algorithm via the proposed IFMWAG operator induced by some continuous
t-norm. To do so, consider ``$p$" different alternatives denoted by $\mathfrak{A}_1$, $\mathfrak{A}_2$,
$\ldots, \mathfrak{A}_p$, which are assessed by an expert using ``$q$" different attributes
$\mathscr{O}_1$, $\mathscr{O}_2$, $\ldots, \mathscr{O}_q$. To deal with the uncertainties in the data, the expert
will be asked to utilize the pairs of IFVs $\gamma_{\ell j}=\langle \mu_{\ell j}, \nu_{\ell j}\rangle$
for $1\leq \ell\leq p$ and $1\leq j\leq q$ to express their ratings towards each alternative. Here,
$\mu_{\ell j}$ and $\nu_{\ell j}$ represent ``agreeness" and ``disagreeness" ratings of the expert
for $\mathfrak{A}_\ell$ under $\mathscr{O}_j$. Assume that $\bm{\Omega}=(\omega_1, \omega_2, \ldots, \omega_q)^{\top}$
is the weight vector for the given attributes and $\Delta$ is the partial information about their information,
if not prior known. This MADM process aims to determine the best alternative(s). The following steps are proposed
based on the stated operators to develop the decision-making algorithm.

{\bf Step~1:}
(Derive the decision matrix)
An expert assesses each alternative $\mathfrak{A}_\ell$ ($\ell=1, 2, \ldots, p$) on each attribute $\mathscr{O}_j$
($j=1, 2, \ldots, q$) with IFVs $\gamma_{\ell j}=\langle \mu_{\ell j},
\nu_{\ell j}\rangle$ features, construct an IF decision matrix {$R=(\gamma_{\ell j})_{p\times q}$}.

{\bf Step~2:}
(Normalize the IF decision matrix) Transform $R=(\gamma_{\ell j})_{p\times q}$ into the normalized
IF decision matrix $\overline{R}=(\bar{\gamma}_{\ell j})_{p\times q}=(\langle \bar{\mu}_{\ell j},
\bar{\nu}_{\ell j}\rangle)_{p\times q}$ as follows:
$$
\bar{\gamma}_{\ell j}=
\begin{cases}
\gamma_{\ell j}, & \mathscr{O}_{j} \text{ is a benefit type attribute}, \\
\gamma_{\ell j}^{\complement}, & \mathscr{O}_{j} \text{ is a cost type attribute},
\end{cases}
$$
where $\gamma_{\ell j}^{\complement}=\langle {\nu}_{\ell j}, {\mu}_{\ell j}\rangle$.

{\bf Step~3:} (Determine the optimal weight) Obtain an optimization model
for weight determination, based on $L$-values, as follows:
\begin{align*}
 \max~ & \sum_{j=1}^{q}\sum_{\ell=1}^{p}\omega_{j} \cdot L(\bar{\gamma}_{\ell j}), \\
\text{s.t.~} & \omega_{j}\in\Delta; ~\sum_{j=1}^{q}\omega_{j}=1;
~\omega_{j}\geq 0;
\end{align*}
where $\triangle$ is the partial information about the {attributes weights}
and $L(\bar{\gamma}_{\ell j})
=\frac{1-\bar{\nu}_{\ell j}}{(1-\bar{\mu}_{\ell j})+(1-\bar{\nu}_{\ell j})}$
is the $L$-value of each IFV $\bar{\gamma}_{\ell j}$.

{\bf Step~4:} (Aggregate the IF information)
Aggregate the information for each alternative $\mathfrak{A}_\ell$ ($1\leq \ell\leq p$)
by using the IFMWAG operator $\mathrm{IFMWAG}_{T, \bm{\Omega}}$ for
some fixed {power stable t-norm $T$ as follows:}
$\gamma_{\ell}=\mathrm{IFMWAG}_{T, \bm{\Omega}}
(\bar{\gamma}_{\ell 1}, \bar{\gamma}_{\ell 2}, \ldots, \bar{\gamma}_{\ell q}).$

{{\bf Step~5:} (Arrange the alternatives)
Compute the $L$-value of each $\gamma_{\ell}$ by
$L(\gamma_{\ell})=\frac{1-\nu_{\gamma_{\ell}}}{(1-\mu_{\gamma_{\ell}})+(1-\nu_{\gamma_{\ell}})}$
and rank the alternatives $\mathfrak{A}_1$--$\mathfrak{A}_{p}$ according to the
nonincreasing order of the $L$-values $L(\gamma_{\ell})$ ($1\leq \ell \leq p$).}

{\begin{remark}
Let $\tilde{\mathbb{Q}}$ be the set of all q-rung orthopair fuzzy numbers, i.e.,
$\tilde{\mathbb{Q}}=\{\langle \mu, \nu\rangle\in [0, 1]^2 \mid \mu^q+\nu^q\leq 1\}$.
Noting that the spaces $\tilde{\mathbb{Q}}$ and $\tilde{\mathbb{I}}$ are isomorphic
via the transformation
\begin{align*}
\Gamma: \tilde{\mathbb{Q}} & \longrightarrow \tilde{\mathbb{I}}\\
\langle \mu, \nu \rangle & \longmapsto \langle \mu^{q}, \nu^{q} \rangle
\end{align*}
it can be verified that all results in Sections~\ref{Sec-IV}--\ref{Sec-VI} can be extended
to q-ROFSs according to the following formula:
$$
\Gamma^{-1}\circ \psi (\Gamma(\gamma_1), \ldots, \Gamma(\gamma_n)),
$$
where $\psi$ is an aggregation operator on $\tilde{\mathbb{I}}^{n}$ and $(\gamma_1, \ldots, \gamma_n)
\in \tilde{\mathbb{Q}}^{n}$.
\end{remark}
}

\subsection{Numerical examples}

The developed method has been exemplified with a real-world specimen described as follows:

\begin{example}
\label{Appl-Exm}
In collaboration with the North-Eastern (NE) council, the Department of Science, Government of India, established
the NESAC (North-Eastern Space Applications Center). By the aid of technologically advanced methods and tools
like remote sensing, NESAC explores innovative approaches for utilizing the natural resources. NE states now have greater
access to satellite services, and this space promotes research in these fields. Additionally, Unmanned Aeria Vehicle
(UAV) remote sensing is a significant progress of NESAC's capabilities for monitoring various activities and
mapping at a large scale. UAV, also called a drone, is a flying robot. It is a mechanical aircraft that can be controlled remotely by humans or autonomously by airborne computers. The NESAC's central mission is to assemble ten new UAVs in the NE territory. In conjunction with this plan, NESAC reviews data analysis and processing software for UAV video and image analysis. NESAC confers an IT company that delivers information on five software models for building the needed software denoted: {$\mathbb{S}_{1}$--$\mathbb{S}_5$}. NESAC appoints
an expert to assess $\mathbb{S}_{1}$--$\mathbb{S}_5$ based on the following four attributes,
$\mathscr{O}_1$: ``Image processing capability"; {$\mathscr{O}_2$: ``Measurement error of measurement tools
for distance/co-ordinate/volume/area";} $\mathscr{O}_3$: ``Generation of contour lines using DSM/DEM"; $\mathscr{O}_4$:
``Generation of 3D modelling/texturing capabilities". The evolutions in the software version will simulate these attributes.
The NESAC plans to select the best software(s). It may be more appropriate to use the IFS-based MADM approach to express how {stakeholders} make decisions about which software is best for the region because subjective perceptions towards each assessed {attribute} fundamentally drive the installation of the respective software.
\end{example}

The following are the steps of the proposed MADM algorithm.

\medskip

{\bf Step~1:}
The IF decision matrix related to the given software with respect to
four attributes is summarized in Table~\ref{tab-1}.
\begin{table}[H]
\caption{IF decision matrix in Example~\ref{Appl-Exm}}
\label{tab-1}\centering%
\scalebox{1.0}{
\begin{tabular}{lccccc}
\toprule
~   & $\mathscr{O}_1$  & $\mathscr{O}_2$ & $\mathscr{O}_3$  & $\mathscr{O}_4$   \\
\midrule
$\mathbb{S}_1$  & $\langle 0.6, 0.1 \rangle$ & $\langle 0.1, 0.8\rangle$ & $\langle 0.6, 0.2 \rangle$ & $\langle 0.8, 0.1\rangle$    \\
$\mathbb{S}_2$  & $\langle 0.7, 0.3 \rangle$ & $\langle 0.2, 0.4\rangle$ & $\langle 0.7, 0.2 \rangle$ & $\langle 0.4, 0.3 \rangle$   \\
$\mathbb{S}_3$  & $\langle 0.6, 0.2 \rangle$ & $\langle 0.3, 0.6\rangle$ & $\langle 0.5, 0.3 \rangle$ & $\langle 0.7, 0.1 \rangle$   \\
$\mathbb{S}_4$  & $\langle 0.2, 0.5 \rangle$ & $\langle 0.3, 0.7\rangle$ & $\langle 0.6, 0.1 \rangle$ & $\langle 0.6, 0.3 \rangle$   \\
$\mathbb{S}_5$  & $\langle 0.5, 0.4 \rangle$ & $\langle 0.6, 0.3\rangle$ & $\langle 0.7, 0.1 \rangle$ & $\langle 0.6, 0.3 \rangle$   \\
\bottomrule
\end{tabular}%
}
\end{table}

{{\bf Step~2:}
Because the attributes $\mathscr{O}_1$, $\mathscr{O}_3$, and $\mathscr{O}_4$ are the benefit type attributes,
and because the attribute $\mathscr{O}_2$ is the cost type attribute, the normalized IF decision matrix is shown
in Table~\ref{tab-1-Normal}.
\begin{table}[H]
\caption{Normalized IF decision matrix in Example~\ref{Appl-Exm}}
\label{tab-1-Normal}\centering%
\scalebox{1.0}{
\begin{tabular}{lccccc}
\toprule
~   & $\mathscr{O}_1$  & $\mathscr{O}_2$ & $\mathscr{O}_3$  & $\mathscr{O}_4$   \\
\midrule
$\mathbb{S}_1$  & $\langle 0.6, 0.1 \rangle$ & $\langle 0.8, 0.1\rangle$ & $\langle 0.6, 0.2 \rangle$ & $\langle 0.8, 0.1\rangle$    \\
$\mathbb{S}_2$  & $\langle 0.7, 0.3 \rangle$ & $\langle 0.4, 0.2 \rangle$ & $\langle 0.7, 0.2 \rangle$ & $\langle 0.4, 0.3 \rangle$   \\
$\mathbb{S}_3$  & $\langle 0.6, 0.2 \rangle$ & $\langle 0.6, 0.3 \rangle$ & $\langle 0.5, 0.3 \rangle$ & $\langle 0.7, 0.1 \rangle$   \\
$\mathbb{S}_4$  & $\langle 0.2, 0.5 \rangle$ & $\langle 0.7, 0.3 \rangle$ & $\langle 0.6, 0.1 \rangle$ & $\langle 0.6, 0.3 \rangle$   \\
$\mathbb{S}_5$  & $\langle 0.5, 0.4 \rangle$ & $\langle 0.3, 0.6 \rangle$ & $\langle 0.7, 0.1 \rangle$ & $\langle 0.6, 0.3 \rangle$   \\
\bottomrule
\end{tabular}%
}
\end{table}
}

{{\bf Step~3:}
 The $L$-value matrix for each alternative is given as
 \begin{equation*}
 \mathbb{L}=
 \begin{gathered}
\begin{pmatrix}
0.6923   &0.8182  &0.6667  &0.8182 \\
0.7000   &0.5714  &0.7273  &0.5385 \\
0.6667   &0.6364  &0.5833  &0.7500 \\
0.3846   &0.7000  &0.6923  &0.6364 \\
0.5455   &0.3636  &0.7500  &0.6364 \\
\end{pmatrix}
\end{gathered}.
\end{equation*}
}

By taking partial information about the weight information as $\Delta=\{0.25\leq \omega_1 \leq 0.4;~
0.1\leq \omega_2 \leq 0.5;~ 0.2\leq \omega_3 \leq 0.5;~ 0.2\leq \omega_4\leq 0.45;~ \omega_1+\omega_2\geq \omega_4;
~\omega_1\leq \omega_2\}$, construct an optimization model as
\begin{align*}
  \max~ & 2.9891 \cdot \omega_1 +3.0896 \cdot \omega_2 +3.4196 \cdot \omega_3 +3.3795 \cdot \omega_4, \\
\text{s.t.~} & \omega_j  \in \Delta; ~\sum_{j=1}^{4}\omega_j =1;
~\omega_j \geq 0.
\end{align*}
After solving the model, we get
$\bm{\Omega}=(0.25, 0.25, 0.3, 0.2)^{\top}.$

{\bf Step~4:}
Utilize the weight and information mentioned in Table~\ref{tab-1-Normal},
with different t-norms, we obtain the aggregated and ranking
results as shown in Table~\ref{tab-Agg-Ranking}, implying that
the ranking of $\mathbb{S}_1$--$\mathbb{S}_5$ always is: $\mathbb{S}_1\succ \mathbb{S}_3
\succ \mathbb{S}_2\succ \mathbb{S}_4\succ \mathbb{S}_5$.
\begin{table}[H]	
	\centering
	\caption{Aggregated and ranking results in Example~\ref{Appl-Exm}}
	\label{tab-Agg-Ranking}
\scalebox{0.72}{
\begin{threeparttable}
\begin{tabular}{cccccccc}
\toprule
 T-norm $T$  & $\mathbb{S}_1$  & $\mathbb{S}_2$  & $\mathbb{S}_3$  & $\mathbb{S}_4$  & $\mathbb{S}_5$ & Ranking  \\
\midrule
\multirow{2}*{$(\langle 0, 0.5, T_{\textbf{P}}\rangle,
\langle 0.5, 1, T_{2}^{\textbf{H}}\rangle)$}
& $\langle 0.6883, 0.1272\rangle$ & $\langle 0.5373, 0.2442\rangle$ &
$\langle 0.5481, 0.2297\rangle$ & $\langle 0.5020, 0.3726\rangle$ &
$\langle 0.5149, 0.3953\rangle$ &
\multirow{2}*{$\mathbb{S}_1\succ \mathbb{S}_3\succ \mathbb{S}_2\succ \mathbb{S}_4\succ \mathbb{S}_5$}  \\
\cline{2-6}
~ & $0.7368$ & $0.6203$ & $0.6303$ & $0.5575$ & $0.5549$ & ~ \\
\midrule
\multirow{2}*{$(\langle 0, 0.5, T_{2}^{\textbf{H}}\rangle,
\langle 0.5, 1, T_{\textbf{P}}\rangle)$}
& $\langle 0.6819, 0.1286\rangle$ & $\langle 0.5345, 0.2458\rangle$ &
$\langle 0.5460, 0.2340\rangle$ & $\langle 0.5082, 0.3805\rangle$ &
$\langle 0.5151, 0.4041\rangle$ &
\multirow{2}*{$\mathbb{S}_1\succ \mathbb{S}_3\succ \mathbb{S}_2\succ \mathbb{S}_4\succ \mathbb{S}_5$}  \\
\cline{2-6}
~ & $0.7326$ & $0.6183$ & $0.6279$ & $0.5574$ & $0.5514$ & ~ \\
\midrule
\multirow{2}*{Algebraic product $T_{\textbf{P}}$}
& $\langle 0.6951, 0.1272\rangle$ & $\langle 0.5672, 0.2433\rangle$ &
$\langle 0.5910, 0.2283\rangle$ & $\langle 0.5156, 0.2756\rangle$ &
$\langle 0.5293, 0.3221\rangle$ &
\multirow{2}*{$\mathbb{S}_1\succ \mathbb{S}_3\succ \mathbb{S}_2\succ \mathbb{S}_4\succ \mathbb{S}_5$}  \\
\cline{2-6}
~ & $0.7411$ & $0.6361$ & $0.6536$ & $0.5993$ & $0.5902$ & ~ \\
\midrule
\multirow{2}*{Einstein product $T_{2}^{\textbf{H}}$}
& $\langle 0.6954, 0.1269\rangle$ & $\langle 0.5677, 0.2432\rangle$ &
$\langle 0.5911, 0.2281\rangle$ & $\langle 0.5181, 0.2745\rangle$ &
$\langle 0.5299, 0.3208\rangle$ &
\multirow{2}*{$\mathbb{S}_1\succ \mathbb{S}_3\succ \mathbb{S}_2\succ \mathbb{S}_4\succ \mathbb{S}_5$}  \\
\cline{2-6}
~ & $0.7414$ & $0.6365 $ & $0.6537$ & $0.6009$ & $0.5910$ & ~ \\
		\bottomrule
	\end{tabular}
\end{threeparttable}
}
\end{table}


{
Considering the t-norm $T =(\langle 0, \lambda, T_{\textbf{P}}\rangle,\langle \lambda, 1, T_{2}^{\textbf{H}}\rangle)$ or
$T =(\langle 0, \lambda, T_{2}^{\textbf{H}}\rangle,\langle \lambda, 1, T_{\textbf{P}}\rangle)$, i.e., $T$ is the ordinal
sum of two summands $\langle 0, \lambda, T_{\textbf{P}}\rangle$ and $\langle \lambda, 1,
T_{2}^{\textbf{H}}\rangle$, to show the detailed influence of the parameter $\lambda$
on the ranking orders in Example~\ref{Appl-Exm}, the $L$-values $L(\gamma_\ell)$
of each alternative $\mathbb{S}_{\ell}$ obtained by the proposed MADM method are shown in
Figs.~\ref{Fig-Appl-Exm} (a) and (b), respectively. Observing from Figs.~\ref{Fig-Appl-Exm} (a) and (b),
we see that, when the parameter $\lambda$ varies from $0$ to $1$,
\begin{enumerate}[(1)]
  \item The ranking order of $\mathbb{S}_1$--$\mathbb{S}_3$
   remains unchanged, always $\mathbb{S}_1\succ \mathbb{S}_3\succ \mathbb{S}_2$, and they are all superior to
   $\mathbb{S}_4$ and $\mathbb{S}_5$;
  \item The ranking order of $\mathbb{S}_4$ and $\mathbb{S}_5$ has changed, except for $\lambda\in [0.1, 0.3]$,
  $\mathbb{S}_4$ is always better than $\mathbb{S}_5$, so overall, $\mathbb{S}_4$ is superior to $\mathbb{S}_5$.
\end{enumerate}
 {This indicates that although our approach is relatively stable,
 the parameters can also affect the ranking results.
 Therefore, it is worth investigating the selection of an appropriate t-norm for specific problems.}}
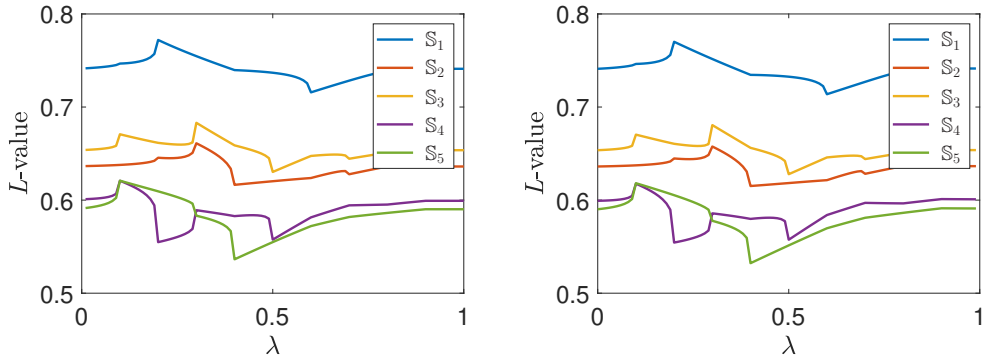
\begin{figure}[H]
\centering
\subfigure[$L$-values aggregated by IFMWAG
for $T=(\langle 0, \lambda, T_{\textbf{P}}\rangle,
\langle \lambda, 1, T_{2}^{\textbf{H}}\rangle)$]
{\scalebox{0.45}{\begin{overpic}{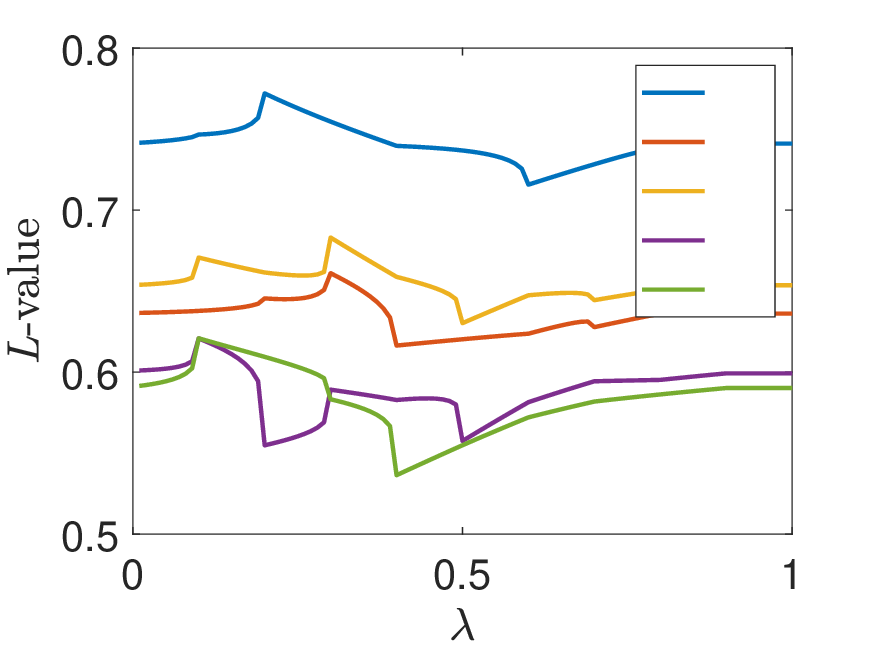}
\put(83,63){\LARGE{$\mathbb{S}_1$}}
\put(83,57.5){\LARGE{$\mathbb{S}_2$}}
\put(83,51.5){\LARGE{$\mathbb{S}_3$}}
\put(83,46){\LARGE{$\mathbb{S}_4$}}
\put(83,40.5){\LARGE{$\mathbb{S}_5$}}
\end{overpic}}}
\subfigure[$L$-values aggregated by IFMWAG
for $T=(\langle 0, \lambda, T_{2}^{\textbf{H}}\rangle,
\langle \lambda, 1, T_{\textbf{P}}\rangle)$]
{\scalebox{0.45}{\begin{overpic}{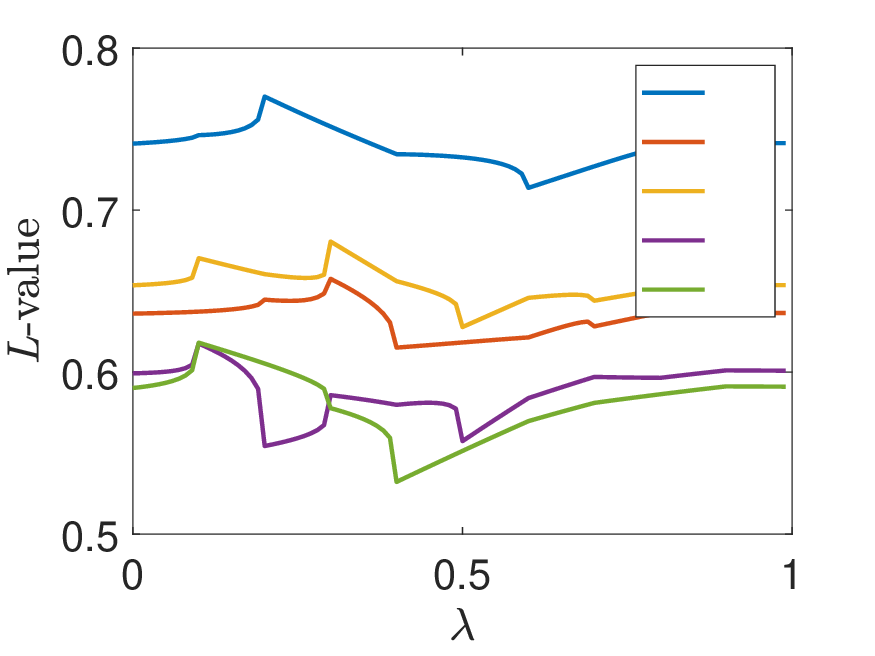}
\put(83,63){\LARGE{$\mathbb{S}_1$}}
\put(83,57.5){\LARGE{$\mathbb{S}_2$}}
\put(83,51.5){\LARGE{$\mathbb{S}_3$}}
\put(83,46){\LARGE{$\mathbb{S}_4$}}
\put(83,40.5){\LARGE{$\mathbb{S}_5$}}
\end{overpic}}}
\caption{$L$-values of $\mathbb{S}_{1}$--$\mathbb{S}_{5}$ in Example~\ref{Appl-Exm}}
\label{Fig-Appl-Exm}
\end{figure}

{
To compare the performances of the ranking order with the existing methods, we carried out comparative
analysis with the MADM approaches in \cite{Xu2007,XY2006,HCZLT2014,Huang2014,CC2016,ZX2017,
Garg2017,Ye2017}. The results obtained by implementing their algorithms (steps are omitted here) are listed
in Table~\ref{Tab-Comparison-Anal}.
\begin{table}[H]
\caption{A comparison of the ranking orders of the alternatives
$\mathbb{S}_1$--$\mathbb{S}_5$ in Example~\ref{Appl-Exm} for different MADM methods}
\label{Tab-Comparison-Anal}\centering%
\scalebox{1.0}{
\begin{tabular}{lcccccccc}
\toprule
Methods & $\mathbb{S}_1$ & $\mathbb{S}_2$ & $\mathbb{S}_3$ & $\mathbb{S}_4$ & $\mathbb{S}_5$ & Ranking\\
\midrule
Xu~\cite{Xu2007}           & 0.5841 & 0.3502 & 0.3786 & 0.3112 & 0.2780 &
$\mathbb{S}_1\succ \mathbb{S}_3\succ \mathbb{S}_2\succ \mathbb{S}_4\succ \mathbb{S}_5$  \\
Xu and Yager~\cite{XY2006}  & 0.5517 & 0.2975 & 0.3469 & 0.1677 & 0.1363 &
$\mathbb{S}_1\succ \mathbb{S}_3\succ \mathbb{S}_2\succ \mathbb{S}_4\succ \mathbb{S}_5$  \\
{Wang and Liu~\cite{WL2012}}  & 0.5517 & 0.2975 & 0.3469 & 0.1677 & 0.1363 &
$\mathbb{S}_1\succ \mathbb{S}_3\succ \mathbb{S}_2\succ \mathbb{S}_4\succ \mathbb{S}_5$  \\
He~et~al.~\cite{HCZLT2014} & 0.5755 & 0.5067 & 0.3540 & 0.3878 & 0.1398 &
$\mathbb{S}_1\succ \mathbb{S}_2\succ \mathbb{S}_3\succ \mathbb{S}_4\succ \mathbb{S}_5$  \\
Huang~\cite{Huang2014}     & 0.5792 & 0.3391 & 0.3734 & 0.2867 & 0.2517 &
$\mathbb{S}_1\succ \mathbb{S}_3\succ \mathbb{S}_2\succ \mathbb{S}_4\succ \mathbb{S}_5$  \\
Chen and Chang~\cite{CC2016} & 0.5764 & 0.1804 & 0.3606 & 0.1147 & 0.2306 &
$\mathbb{S}_1\succ \mathbb{S}_3\succ \mathbb{S}_5\succ \mathbb{S}_2\succ \mathbb{S}_4$  \\
Zhou and Xu~\cite{ZX2017}  & 0.6119 & 0.3502 & 0.4214 & 0.3334 & 0.2997 &
$\mathbb{S}_1\succ \mathbb{S}_3\succ \mathbb{S}_2\succ \mathbb{S}_4\succ \mathbb{S}_5$  \\
Garg~\cite{Garg2017}       & 0.5721 & 0.1669 & 0.3581 & 0.0940 & 0.2183 &
$\mathbb{S}_1\succ \mathbb{S}_3\succ \mathbb{S}_5\succ \mathbb{S}_2\succ \mathbb{S}_4$  \\
Ye~\cite{Ye2017}           & 0.5678 & 0.3234 & 0.3627 & 0.2376 & 0.2050 &
$\mathbb{S}_1\succ \mathbb{S}_3\succ \mathbb{S}_2\succ \mathbb{S}_4\succ \mathbb{S}_5$  \\
{Liu and Chen~\cite{LC2017}}   & 0.5721 & 0.1669 & 0.3581 & 0.0940 & 0.2183 &
$\mathbb{S}_1\succ \mathbb{S}_3\succ \mathbb{S}_5\succ \mathbb{S}_2\succ \mathbb{S}_4$  \\
{Farid and Riaz~\cite{FR2023}}   & 0.5721 & 0.1669 & 0.3581 & 0.0940 & 0.2183 &
$\mathbb{S}_1\succ \mathbb{S}_3\succ \mathbb{S}_5\succ \mathbb{S}_2\succ \mathbb{S}_4$  \\
{Senapati et al.~\cite{SCMY2023}}  & 0.5678 & 0.3234 & 0.3627 & 0.2376 & 0.2050 &
$\mathbb{S}_1\succ \mathbb{S}_3\succ \mathbb{S}_2\succ \mathbb{S}_4\succ \mathbb{S}_5$  \\
{Senapati et al.~\cite{SCY2022}}   & 0.5678 & 0.3234 & 0.3627 & 0.2376 & 0.2050 &
$\mathbb{S}_1\succ \mathbb{S}_3\succ \mathbb{S}_2\succ \mathbb{S}_4\succ \mathbb{S}_5$  \\
Proposed method            & 0.7368 & 0.6203 & 0.6303 & 0.5575 & 0.5549 &
$\mathbb{S}_1\succ \mathbb{S}_3\succ \mathbb{S}_2\succ \mathbb{S}_4\succ \mathbb{S}_5$  \\
\bottomrule
\end{tabular}%
}
\end{table}

Observing from Table~\ref{Tab-Comparison-Anal}, we can conclude that
\begin{enumerate}[(1)]
  \item The best alternative obtained by all methods is the same, which is the alternative $\mathbb{S}_1$,
  although there are some differences in the ranking among the various methods.
  \item Only the method presented in~\cite{HCZLT2014} identifies the alternative
  $\mathbb{S}_2$ as the second-best alternative, while all other methods rank the
  alternative $\mathbb{S}_3$ as the second-best alternative. This discrepancy arises
  from the fact that the aggregation method proposed presented in~\cite{HCZLT2014} is
  the only one considering the interactivity.
  \item All the methods, except those presented in~\cite{CC2016,Garg2017}, consistently rank the alternative
  $\mathbb{S}_5$ as the worst alternative. This result is more reasonable. Intuitively, when comparing
  the alternatives $\mathbb{S}_4$ and $\mathbb{S}_5$ individually, since the weights for the
  attributes $\mathscr{O}_1$ and $\mathscr{O}_2$
  are equal and both equal to $0.25$, we can approximate $\langle 0.2, 0.5\rangle$ and $\langle 0.3, 0.6\rangle$
  as equal in terms of $L$-values. Therefore, the comparison between the alternatives $\mathbb{S}_4$ and $\mathbb{S}_5$
  on the attributes $\mathscr{O}_1$ and $\mathscr{O}_2$ is similar to the comparison of the comparison between
  $\langle 0.7, 0.3\rangle$ and $\langle 0.5, 0.4\rangle$ with the weight of $0.25$. Since the difference
  between $\langle 0.7, 0.3\rangle$ and $\langle 0.5, 0.4\rangle$ at the weight of $0.25$ is greater than the
  difference between $\langle 0.6, 0.1\rangle$ and $\langle 0.7, 0.1\rangle$ at the weight of $0.3$, and the
  alternatives $\mathbb{S}_4$ and $\mathbb{S}_5$ have equal values on the attribute $\mathscr{O}_4$, we can intuitively
  conclude that the alternative $\mathbb{S}_4$ is superior to the alternative $\mathbb{S}_5$.
  \item Including our proposed MADM method, $2/3$ of the MADM methods shown in Table~\ref{Tab-Comparison-Anal} yield the
  ranking result $\mathbb{S}_1\succ \mathbb{S}_3\succ \mathbb{S}_2\succ \mathbb{S}_4\succ \mathbb{S}_5$.
  Therefore, the result obtained by our proposed MADM method is reliable. This also indicates that our method is effective.
\end{enumerate}

Furthermore, to illustrate the superiority of our proposed MADM method, we consider other IF decision matrices in Example~\ref{Appl-Exm} as shown in Tables~\ref{tab-1a}
and \ref{tab-1b}.
\begin{table}[H]
\caption{IF decision matrix in Example~\ref{Appl-Exm}}
\label{tab-1a}\centering%
\scalebox{1.0}{
\begin{tabular}{lccccc}
\toprule
~   & $\mathscr{O}_1$  & $\mathscr{O}_2$ & $\mathscr{O}_3$  & $\mathscr{O}_4$   \\
\midrule
$\mathbb{S}_1$  & $\langle 0.6, 0.1 \rangle$ & $\langle 0.1, 0.8\rangle$ & $\langle 0.6, 0.2\rangle$ & $\langle 1, 0 \rangle$ \\
$\mathbb{S}_2$  & $\langle 0.7, 0.3 \rangle$ & $\langle 0.2, 0.4\rangle$ & $\langle 0.7, 0.2\rangle$ & $\langle 1, 0 \rangle$ \\
$\mathbb{S}_3$  & $\langle 0.6, 0.2 \rangle$ & $\langle 0.3, 0.6\rangle$ & $\langle 0.5, 0.3\rangle$ & $\langle 1, 0 \rangle$ \\
$\mathbb{S}_4$  & $\langle 0.2, 0.5 \rangle$ & $\langle 0.3, 0.7\rangle$ & $\langle 0.6, 0.1\rangle$ & $\langle 1, 0 \rangle$ \\
$\mathbb{S}_5$  & $\langle 0.5, 0.4 \rangle$ & $\langle 0.6, 0.3\rangle$ & $\langle 0.7, 0.1\rangle$ & $\langle 1, 0 \rangle$ \\
\bottomrule
\end{tabular}%
}
\end{table}

\begin{table}[H]
\caption{IF decision matrix in Example~\ref{Appl-Exm}}
\label{tab-1b}\centering%
\scalebox{1.0}{
\begin{tabular}{lccccc}
\toprule
~   & $\mathscr{O}_1$  & $\mathscr{O}_2$ & $\mathscr{O}_3$  & $\mathscr{O}_4$   \\
\midrule
$\mathbb{S}_1$  & $\langle 0.6, 0.1 \rangle$ & $\langle 0.1, 0.8\rangle$ & $\langle 0, 1\rangle$ & $\langle 0.8, 0.1\rangle$ \\
$\mathbb{S}_2$  & $\langle 0.7, 0.3 \rangle$ & $\langle 0.2, 0.4\rangle$ & $\langle 0, 1\rangle$ & $\langle 0.4, 0.3\rangle$ \\
$\mathbb{S}_3$  & $\langle 0.6, 0.2 \rangle$ & $\langle 0.3, 0.6\rangle$ & $\langle 0, 1\rangle$ & $\langle 0.7, 0.1\rangle$ \\
$\mathbb{S}_4$  & $\langle 0.2, 0.5 \rangle$ & $\langle 0.3, 0.7\rangle$ & $\langle 0, 1\rangle$ & $\langle 0.6, 0.3\rangle$ \\
$\mathbb{S}_5$  & $\langle 0.5, 0.4 \rangle$ & $\langle 0.6, 0.3\rangle$ & $\langle 0, 1\rangle$ & $\langle 0.6, 0.3\rangle$ \\
\bottomrule
\end{tabular}%
}
\end{table}


By fixing the weight vector $\bm{\Omega}=(0.25, 0.25, 0.3, 0.2)^{\top}$, we carried out a comparative
analysis with the MADM approaches presented in \cite{CC2016,FR2023,Garg2016,Garg2017,HCZLT2014,HHC2015,Huang2014,Liu2014,LC2017,SCMY2023,SCY2022,WL2012,
Xu2007,Xu2011,XY2006,Ye2017,ZX2017}.
The results are shown in Table~\ref{Tab-Comparison-Anal-a} and Table~\ref{Tab-Comparison-Anal-b}.

\begin{table}[H]	
	\centering
	\caption{A comparison of the ranking orders of the alternatives
$\mathbb{S}_1$--$\mathbb{S}_5$ in Example~\ref{Appl-Exm} with the
decision matrix shown in Table~\ref{tab-1a} for different MADM methods}
	\label{Tab-Comparison-Anal-a}
\scalebox{0.72}{
\begin{threeparttable}
\begin{tabular}{lccccccc}
\toprule
 Methods  & $\mathbb{S}_1$  & $\mathbb{S}_2$  & $\mathbb{S}_3$  & $\mathbb{S}_4$  & $\mathbb{S}_5$ & Ranking  \\
\midrule
\multirow{2}*{Xu~\cite{Xu2007}}
& $\langle 1, 0\rangle$ & $\langle 1, 0\rangle$ &
$\langle 1, 0\rangle$ & $\langle 1, 0\rangle$ &
$\langle 1, 0\rangle$ &
\multirow{2}*{$\mathbb{S}_1\thicksim \mathbb{S}_2\thicksim \mathbb{S}_3\thicksim \mathbb{S}_4\thicksim \mathbb{S}_5$}  \\
\cline{2-6}
~ & $1$ & $1$ & $1$ & $1$ & $1$ & ~ \\
\midrule
\multirow{2}*{Xu~\cite{Xu2011}}
& $\langle 1, 0\rangle$ & $\langle 1, 0\rangle$ &
$\langle 1, 0\rangle$ & $\langle 1, 0\rangle$ &
$\langle 1, 0\rangle$ &
\multirow{2}*{$\mathbb{S}_1\thicksim \mathbb{S}_2\thicksim \mathbb{S}_3\thicksim \mathbb{S}_4\thicksim \mathbb{S}_5$}  \\
\cline{2-6}
~ & $1$ & $1$ & $1$ & $1$ & $1$ & ~ \\
\midrule
\multirow{2}*{Wang and Liu~\cite{WL2012}}
& $\langle 1, 0\rangle$ & $\langle 1, 0\rangle$ &
$\langle 1, 0\rangle$ & $\langle 1, 0\rangle$ &
$\langle 1, 0\rangle$ &
\multirow{2}*{$\mathbb{S}_1\thicksim \mathbb{S}_2\thicksim \mathbb{S}_3\thicksim \mathbb{S}_4\thicksim \mathbb{S}_5$}  \\
\cline{2-6}
~ & $1$ & $1$ & $1$ & $1$ & $1$ & ~ \\
\midrule
\multirow{2}*{Huang~\cite{Huang2014}}
& $\langle 1, 0\rangle$ & $\langle 1, 0\rangle$ &
$\langle 1, 0\rangle$ & $\langle 1, 0\rangle$ &
$\langle 1, 0\rangle$ &
\multirow{2}*{$\mathbb{S}_1\thicksim \mathbb{S}_2\thicksim \mathbb{S}_3\thicksim \mathbb{S}_4\thicksim \mathbb{S}_5$}  \\
\cline{2-6}
~ & $1$ & $1$ & $1$ & $1$ & $1$ & ~ \\
\midrule
\multirow{2}*{Liu~\cite{Liu2014}}
& $\langle 1, 0\rangle$ & $\langle 1, 0\rangle$ &
$\langle 1, 0\rangle$ & $\langle 1, 0\rangle$ &
$\langle 1, 0\rangle$ &
\multirow{2}*{$\mathbb{S}_1\thicksim \mathbb{S}_2\thicksim \mathbb{S}_3\thicksim \mathbb{S}_4\thicksim \mathbb{S}_5$}  \\
\cline{2-6}
~ & $1$ & $1$ & $1$ & $1$ & $1$ & ~ \\
\midrule
\multirow{2}*{Chen and Chang~\cite{CC2016}}
& $\langle 1, 0\rangle$ & $\langle 1, 0\rangle$ &
$\langle 1, 0\rangle$ & $\langle 1, 0\rangle$ &
$\langle 1, 0\rangle$ &
\multirow{2}*{$\mathbb{S}_1\thicksim \mathbb{S}_2\thicksim \mathbb{S}_3\thicksim \mathbb{S}_4\thicksim \mathbb{S}_5$}  \\
\cline{2-6}
~ & $1$ & $1$ & $1$ & $1$ & $1$ & ~ \\
\midrule
\multirow{2}*{Garg~\cite{Garg2017}}
& $\langle 1, 0\rangle$ & $\langle 1, 0\rangle$ &
$\langle 1, 0\rangle$ & $\langle 1, 0\rangle$ &
$\langle 1, 0\rangle$ &
\multirow{2}*{$\mathbb{S}_1\thicksim \mathbb{S}_2\thicksim \mathbb{S}_3\thicksim \mathbb{S}_4\thicksim \mathbb{S}_5$}  \\
\cline{2-6}
~ & $1$ & $1$ & $1$ & $1$ & $1$ & ~ \\
\midrule
\multirow{2}*{Farid and Riaz~\cite{FR2023}}
& $\langle 1, 0\rangle$ & $\langle 1, 0\rangle$ &
$\langle 1, 0\rangle$ & $\langle 1, 0\rangle$ &
$\langle 1, 0\rangle$ &
\multirow{2}*{$\mathbb{S}_1\thicksim \mathbb{S}_2\thicksim \mathbb{S}_3\thicksim \mathbb{S}_4\thicksim \mathbb{S}_5$}  \\
\cline{2-6}
~ & $1$ & $1$ & $1$ & $1$ & $1$ & ~ \\
\midrule
\multirow{2}*{Senapati et al.~\cite{SCY2022}}
& $\langle 1, 0\rangle$ & $\langle 1, 0\rangle$ &
$\langle 1, 0\rangle$ & $\langle 1, 0\rangle$ &
$\langle 1, 0\rangle$ &
\multirow{2}*{$\mathbb{S}_1\thicksim \mathbb{S}_2\thicksim \mathbb{S}_3\thicksim \mathbb{S}_4\thicksim \mathbb{S}_5$}  \\
\cline{2-6}
~ & $1$ & $1$ & $1$ & $1$ & $1$ & ~ \\
\midrule
\multirow{2}*{Senapati et al.~\cite{SSSDRT2023}}
& $\langle 1, 0\rangle$ & $\langle 1, 0\rangle$ &
$\langle 1, 0\rangle$ & $\langle 1, 0\rangle$ &
$\langle 1, 0\rangle$ &
\multirow{2}*{$\mathbb{S}_1\thicksim \mathbb{S}_2\thicksim \mathbb{S}_3\thicksim \mathbb{S}_4\thicksim \mathbb{S}_5$}  \\
\cline{2-6}
~ & $1$ & $1$ & $1$ & $1$ & $1$ & ~ \\
\midrule
\multirow{2}*{Proposed method}
& $\langle 0.8570, 0.0564\rangle$ & $\langle 0.8268, 0.0955\rangle$ &
$\langle 0.8146, 0.1114\rangle$ & $\langle 0.7624, 0.1274\rangle$ &
$\langle 0.7796, 0.1609\rangle$ &
\multirow{2}*{$\mathbb{S}_1\succ \mathbb{S}_2\succ \mathbb{S}_3\succ \mathbb{S}_5\succ \mathbb{S}_4$}  \\
\cline{2-6}
~ & $0.8684$ & $0.8393$ & $0.8274$ & $0.7860$ & $0.7920$ & ~ \\
		\bottomrule
	\end{tabular}
\end{threeparttable}
}
\end{table}

\begin{table}[H]	
	\centering
	\caption{A comparison of the ranking orders of the alternatives
$\mathbb{S}_1$--$\mathbb{S}_5$ in Example~\ref{Appl-Exm} with the
decision matrix shown in Table~\ref{tab-1b} for different MADM methods}
	\label{Tab-Comparison-Anal-b}
\scalebox{0.72}{
\begin{threeparttable}
\begin{tabular}{lccccccc}
\toprule
 Methods  & $\mathbb{S}_1$  & $\mathbb{S}_2$  & $\mathbb{S}_3$  & $\mathbb{S}_4$  & $\mathbb{S}_5$ & Ranking  \\
\midrule
\multirow{2}*{Xu and Yager~\cite{XY2006}}
& $\langle 0, 1\rangle$ & $\langle 0, 1\rangle$ &
$\langle 0, 1\rangle$ & $\langle 0, 1\rangle$ &
$\langle 0, 1\rangle$ &
\multirow{2}*{$\mathbb{S}_1\thicksim \mathbb{S}_2\thicksim \mathbb{S}_3\thicksim \mathbb{S}_4\thicksim \mathbb{S}_5$}  \\
\cline{2-6}
~ & $0$ & $0$ & $0$ & $0$ & $0$ & ~ \\
\midrule
\multirow{2}*{He~et~al.~\cite{HCZLT2014}}
& $\langle 0, 1\rangle$ & $\langle 0, 1\rangle$ &
$\langle 0, 1\rangle$ & $\langle 0, 1\rangle$ &
$\langle 0, 1\rangle$ &
\multirow{2}*{$\mathbb{S}_1\thicksim \mathbb{S}_2\thicksim \mathbb{S}_3\thicksim \mathbb{S}_4\thicksim \mathbb{S}_5$}  \\
\cline{2-6}
~ & $0$ & $0$ & $0$ & $0$ & $0$ & ~ \\
\midrule
\multirow{2}*{He~et~al.~\cite{HHC2015}}
& $\langle 0, 1\rangle$ & $\langle 0, 1\rangle$ &
$\langle 0, 1\rangle$ & $\langle 0, 1\rangle$ &
$\langle 0, 1\rangle$ &
\multirow{2}*{$\mathbb{S}_1\thicksim \mathbb{S}_2\thicksim \mathbb{S}_3\thicksim \mathbb{S}_4\thicksim \mathbb{S}_5$}  \\
\cline{2-6}
~ & $0$ & $0$ & $0$ & $0$ & $0$ & ~ \\
\midrule
\multirow{2}*{Garg~\cite{Garg2016}}
& $\langle 0, 1\rangle$ & $\langle 0, 1\rangle$ &
$\langle 0, 1\rangle$ & $\langle 0, 1\rangle$ &
$\langle 0, 1\rangle$ &
\multirow{2}*{$\mathbb{S}_1\thicksim \mathbb{S}_2\thicksim \mathbb{S}_3\thicksim \mathbb{S}_4\thicksim \mathbb{S}_5$}  \\
\cline{2-6}
~ & $0$ & $0$ & $0$ & $0$ & $0$ & ~ \\
\midrule
\multirow{2}*{Liu and Chen~\cite{LC2017}}
& $\langle 0, 1\rangle$ & $\langle 0, 1\rangle$ &
$\langle 0, 1\rangle$ & $\langle 0, 1\rangle$ &
$\langle 0, 1\rangle$ &
\multirow{2}*{$\mathbb{S}_1\thicksim \mathbb{S}_2\thicksim \mathbb{S}_3\thicksim \mathbb{S}_4\thicksim \mathbb{S}_5$}  \\
\cline{2-6}
~ & $0$ & $0$ & $0$ & $0$ & $0$ & ~ \\
\midrule
\multirow{2}*{Zhou and Xu~\cite{ZX2017}}
& $\langle 0, 1\rangle$ & $\langle 0, 1\rangle$ &
$\langle 0, 1\rangle$ & $\langle 0, 1\rangle$ &
$\langle 0, 1\rangle$ &
\multirow{2}*{$\mathbb{S}_1\thicksim \mathbb{S}_2\thicksim \mathbb{S}_3\thicksim \mathbb{S}_4\thicksim \mathbb{S}_5$}  \\
\cline{2-6}
~ & $0$ & $0$ & $0$ & $0$ & $0$ & ~ \\
\midrule
\multirow{2}*{Ye~\cite{Ye2017}}
& $\langle 0, 1\rangle$ & $\langle 0, 1\rangle$ &
$\langle 0, 1\rangle$ & $\langle 0, 1\rangle$ &
$\langle 0, 1\rangle$ &
\multirow{2}*{$\mathbb{S}_1\thicksim \mathbb{S}_2\thicksim \mathbb{S}_3\thicksim \mathbb{S}_4\thicksim \mathbb{S}_5$}  \\
\cline{2-6}
~ & $0$ & $0$ & $0$ & $0$ & $0$ & ~ \\
\midrule
\multirow{2}*{Senapati et~al.~\cite{SCMY2023}}
& $\langle 0, 1\rangle$ & $\langle 0, 1\rangle$ &
$\langle 0, 1\rangle$ & $\langle 0, 1\rangle$ &
$\langle 0, 1\rangle$ &
\multirow{2}*{$\mathbb{S}_1\thicksim \mathbb{S}_2\thicksim \mathbb{S}_3\thicksim \mathbb{S}_4\thicksim \mathbb{S}_5$}  \\
\cline{2-6}
~ & $0$ & $0$ & $0$ & $0$ & $0$ & ~ \\
\midrule
\multirow{2}*{Proposed method}
& $\langle 0.3073, 0.5998\rangle$ & $\langle 0.2060, 0.6945\rangle$ &
$\langle 0.2514, 0.6561\rangle$ & $\langle 0.2086, 0.7446\rangle$ &
$\langle 0.1798, 0.7751\rangle$ &
\multirow{2}*{$\mathbb{S}_1\succ \mathbb{S}_3\succ \mathbb{S}_2\succ \mathbb{S}_4\succ \mathbb{S}_5$} \\
\cline{2-6}
~ & $0.3662$ & $0.2778$ & $0.3148$ & $0.2440$ & $0.2152$ & ~ \\
		\bottomrule
	\end{tabular}
\end{threeparttable}
}
\end{table}

Observing from Table~\ref{Tab-Comparison-Anal-a} and Table~\ref{Tab-Comparison-Anal-b},
we can conclude that, except for our proposed MADM method, the comprehensive evaluation values obtained by
other MADM methods presented in \cite{CC2016,FR2023,Garg2016,Garg2017,HCZLT2014,HHC2015,Huang2014,Liu2014,LC2017,SCMY2023,SCY2022,WL2012,
Xu2007,Xu2011,XY2006,Ye2017,ZX2017} are either $\langle 0, 1\rangle$ or $\langle 1, 0 \rangle$, which leads
to the inability of these methods to effectively distinguish between any two alternatives.

To conclude, the proposed MADM approach is not only effective, but also comprehensively utilizes t-norms and the
proposed IFMWAG operator to aggregate the IF information, which can completely overcome the disadvantage of
the methods presented in \cite{CC2016,FR2023,Garg2016,Garg2017,HCZLT2014,HHC2015,Huang2014,Liu2014,LC2017,SCMY2023,SCY2022,
WL2012,Xu2007,Xu2011,XY2006,Ye2017,ZX2017}
that either
only utilize the average operator or the geometric operator, resulting in the indistinguishability of the
ranking orders of the alternatives. Therefore, the comprehensive reliability of the proposed MADM approach
outperforms the MADM approaches presented in
\cite{CC2016,FR2023,Garg2016,Garg2017,HCZLT2014,HHC2015,Huang2014,Liu2014,LC2017,SCMY2023,SCY2022,WL2012,Xu2007,Xu2011,
XY2006,Ye2017,ZX2017}.
}

\section{Conclusion}
\label{Sec-VII}

{This paper systematically investigates the power operation of continuous t-norms and applies it to the IF MADM
problems. It is first proved that a continuous t-norm is power stable if and only if every point is a power stable
point, and if and only if it is the minimum t-norm, or it is strict, or it is an ordinal sum of strict t-norms.
Then, an important computing formula is obtained for the power of continuous t-norms by using continuous t-norms'
representation theorem. Based on the power of t-norms, four basic operations for IFSs induced by a continuous t-norm
are introduced, including addition, multiplication, scalar multiplication, and power operations. Besides, some basic
properties are obtained for these four operations. Based on these four operational laws, the IF weighted average
and IF weighted geometric operators induced by a continuous t-norm are formulated. It is shown that they are monotonous,
idempotent, and bounded. Compared to the limitation of all existing IF aggregation operators that only consider
the t-norms with continuous AGs, our results can be more widely applied to the IF MADM problems. Combining the IF weighted
average and IF weighted geometric operators, the IF mean weighted average and geometric operator (IFMWAG) is proposed
and a novel IF MADM method based on this operator is established. Besides, a practical example and
comparative analysis with other decision-making methods further illustrate the effectiveness of the
developed IF MADM method.}



\section*{References}
\bibliographystyle{abbrv}
\bibliography{IEEEexample}

\end{document}